\documentclass{amsart}
\usepackage{amsfonts,amssymb,mathtools,mathabx}
\usepackage[usenames,dvipsnames]{xcolor}

\usepackage{tikz}
\usetikzlibrary{arrows, positioning, calc, intersections}
\usetikzlibrary{decorations.pathreplacing, decorations.markings}
\usepackage{relsize}					
\usepackage{exscale}					

\usepackage{enumerate}   

\usepackage{float}
\usepackage[section]{placeins}

\usepackage{hyperref}
\hypersetup{
    bookmarks=true,         	
    unicode=false,          		
    pdftoolbar=true,        		
    pdfmenubar=true,        	
    pdffitwindow=false,     		
    pdfstartview={FitH},    		
   pdftitle={Global Well-Posedness for DNLS},    					
   	pdfauthor={Robert Jenkins, Jiaqi Liu, Peter Perry, Catherine Sulem},     	
    linktocpage=true,			
    pdfnewwindow=true,      	
    colorlinks=true,       			
    linkcolor=blue,          		
    citecolor=PineGreen,    	
    filecolor=magenta,      		
    urlcolor=cyan           		
}

\theoremstyle{plain}
\newtheorem{theorem}{Theorem}[section]

\newtheorem{lemma}[theorem]{Lemma}
\newtheorem{proposition}[theorem]{Proposition}

\theoremstyle{remark}
\newtheorem{remark}[theorem]{Remark}

\theoremstyle{definition}

\newtheorem{RHP}[theorem]{Riemann-Hilbert Problem}

\newtheorem{definition}[theorem]{Definition}

\numberwithin{figure}{section}
\numberwithin{equation}{section}

\DeclareMathOperator{\ad}{ad}
\DeclareMathOperator{\real}{Re}

\DeclareMathOperator{\imag}{Im}

\DeclareMathOperator{\sech}{sech}

\DeclareMathOperator{\Res}{Res}

\begin{document}

\title[Global Well-Posedness for the DNLS Equation]{Global Well-Posedness  for the Derivative Non-Linear Schr\"{o}dinger Equation}
\author{Robert Jenkins}
\author{Jiaqi Liu}
\author{Peter A. Perry}
\author{Catherine Sulem}
\address[Jenkins]{Department of Mathematics, University of Arizona, Tucson, AZ 85721-0089 USA}
\address[Liu]{Department of Mathematics, University of Toronto, Toronto, Ontario M5S 2E4, Canada}
\address[Perry]{ Department of Mathematics,  University of Kentucky, Lexington, Kentucky 40506--0027}
\address[Sulem]{Department of Mathematics, University of Toronto, Toronto, Ontario M5S 2E4, Canada }
\thanks{Peter Perry supported in part by NSF Grant DMS-1208778}
\thanks {C. Sulem supported in part by NSERC Grant 46179-13}
\date{\today}
\begin{abstract}
We study the Derivative Nonlinear Schr\"odinger (DNLS). equation for  general initial conditions in weighted 
Sobolev spaces that can support bright solitons (but exclude spectral singularities corresponding to algebraic solitons).  We show that the set of such initial data is open and dense in a weighted Sobolev space, and includes data of arbitrarily large $L^2$-norm.  We prove global well-posedness on this open and dense set. In a subsequent paper \cite{JLPS18}, we will use these results and a steepest descent analysis to prove the soliton resolution conjecture for the  DNLS equation with the initial data considered here and asymptotic stability of $N-$soliton solutions. 
\end{abstract}
\maketitle
\tableofcontents

%
%

\begingroup
\let\clearpage\relax
%
%

%
%

\newcommand{\rarr}{\rightarrow}
\newcommand{\darr}{\downarrow}
\newcommand{\uarr}{\uparrow}

\newcommand{\sig}{\sigma_3}

\newcommand{\dotarg}{\, \cdot \, }

\newcommand{\eps}{\varepsilon}
\newcommand{\lam}{\lambda}
\newcommand{\Lam}{\Lambda}

\newcommand{\dee}{\partial}
\newcommand{\dbar}{\overline{\partial}}

\newcommand{\vect}[1]{\boldsymbol{\mathbf{#1}}}

\newcommand{\dint}{\displaystyle{\int}}

\newcommand{\restrict}[2]{\left. {#1} \right|_{#2}}

%
%

\newcommand{\resp}{resp.\@}
\newcommand{\ifff}{if and only if }
\newcommand{\ie}{i.e.}

%
%

\newcommand{\lp}{\left(}
\newcommand{\rp}{\right)}
\newcommand{\lb}{\left[}
\newcommand{\rb}{\right]}
\newcommand{\lw}{\left<}
\newcommand{\rw}{\right>}

\newcommand{\one}{\bm{1}}

%
%


\newcommand{\C}{\mathbb{C}}
\newcommand{\R}{\mathbb{R}}
\newcommand{\N}{\mathbb{N}}


\newcommand{\calB}{\mathcal{B}}
\newcommand{\calC}{\mathcal{C}}
\newcommand{\calD}{\mathcal{D}}
\newcommand{\calE}{\mathcal{E}}
\newcommand{\calF}{\mathcal{F}}
\newcommand{\calG}{\mathcal{G}}
\newcommand{\calI}{\mathcal{I}}
\newcommand{\calK}{\mathcal{K}}
\newcommand{\calL}{\mathcal{L}}
\newcommand{\calM}{\mathcal{M}}
\newcommand{\calN}{\mathcal{N}}
\newcommand{\calR}{\mathcal{R}}
\newcommand{\calS}{\mathcal{S}}
\newcommand{\calU}{\mathcal{U}}
\newcommand{\calZ}{\mathcal{Z}}


\newcommand{\ba}{\breve{a}}
\newcommand{\bb}{\breve{b}}

\newcommand{\br}{\breve{r}}
\newcommand{\bs}{\breve{s}}

\newcommand{\balpha}{\breve{\alpha}}
\newcommand{\bbeta}{\breve{\beta}}
\newcommand{\brho}{\breve{\rho}}
\newcommand{\bgamma}{\breve{\gamma}}

\newcommand{\bC}{\breve{C}}
\newcommand{\bN}{\breve{N}}


\newcommand{\bfe}{\mathbf{e}}
\newcommand{\bff}{\mathbf{f}}
\newcommand{\bfN}{\mathbf{N}}
\newcommand{\bfM}{\mathbf{M}}


\newcommand{\qbar}{\overline{q}}
\newcommand{\rbar}{\overline{r}}
\newcommand{\wbar}{\overline{w}}
\newcommand{\zbar}{\overline{z}}


\newcommand{\etabar}{\overline{\eta}}
\newcommand{\lambar}{\overline{\lambda}}
\newcommand{\lambdabar}{\overline{\lambda}}
\newcommand{\rhobar}{\overline{\rho}}
\newcommand{\btheta}{\overline{\vartheta}}
\newcommand{\zetabar}{\overline{\zeta}}


\newcommand{\hatphi}{\widehat{\phi}}

%
%
%

\newcommand{\bigo}[1]{\mathcal{O} \left( #1 \right)}
\newcommand{\littleo}[2][ ]{ {o}_{#1} \left( #2 \right) }

\newcommand{\bigO}[2][ ]{\mathcal{O}_{#1} \left( {#2} \right)}
\newcommand{\norm}[2][ ]{\left\| {#2} \right\|_{#1}}

%
%

\newcommand{\medcup}{{\mathsmaller{\bigcup}}}

%
%

\newcommand{\upmat}[1]
{
	\begin{pmatrix}	
	0	&	#1	\\
	0	&	0
	\end{pmatrix}
}
\newcommand{\lowmat}[1]
{
	\begin{pmatrix}
	0	&	0	\\
	#1	&	0
	\end{pmatrix}
}
\newcommand{\upunitmat}[1]
{
	\begin{pmatrix}
	1	&	#1	\\
	0	&	1
	\end{pmatrix}
}
\newcommand{\lowunitmat}[1]
{
	\begin{pmatrix}
	1	&	0	\\
	#1	&	1
	\end{pmatrix}
}

%
%

\newcommand{\La}{ {\mathcal{L} } }
\newcommand{\B}{ {\mathcal{B } } }

\newcommand{\qsol}{q_{\mathrm{sol}}}
\newcommand{\usol}{u_{\mathrm{sol}}}
\newcommand{\tqsol}{\widetilde{q}_{\mathrm{sol}}}

\newcommand{\poles}{\Lambda}
\newcommand{\indicator}{\chi_{_\poles}}
\newcommand{\poledist}{ {\mathrm{d}_{\poles} } }
\newcommand{\coeff}{\mathcal{C}}
\newcommand{\data}{\sigma_d}

\newcommand{\NN}{ {N} }
\newcommand{\nn}{ {n} }
\newcommand{\nk}[2][]{ {{n}_{#1}^{(#2)}} }

\newcommand{\mk}[1]{ m^{(#1)} }
\newcommand{\vk}[1]{ {v^{(#1)}} }
\newcommand{\Sk}[1]{ {\Sigma^{(#1)} } }
\newcommand{\Wk}[2][]{ {W_{#1}^{(#2)} } }

\newcommand{\sgnt}{\eta}
\newcommand{\pospoles}{\Delta_{\xi,\sgnt}^+}
\newcommand{\negpoles}{\Delta_{\xi,\sgnt}^-}
\newcommand{\posnegpoles}{\Delta_{\xi,\sgnt}^\pm}
\newcommand{\posint}{ I^{+}_{\xi,\sgnt}} 
\newcommand{\negint}{ I^{-}_{\xi,\sgnt}} 						

\newcommand{\Nrhp}[1][]{ {{\mathcal{N}}_{#1}^{\mathsc{rhp}} } }
\newcommand{\Nsol}[2][]{ {{\mathcal{N}}_{\!\!#1}^{\mathrm{sol},#2} } }
\newcommand{\Nout}[1][]{ {{\mathcal{N}}_{\!\!#1}^{\mathrm{out}} } }
\newcommand{\NsolAlt}[1][]{ {{\mathcal{N}}_{#1}^{\negpoles} } }
\newcommand{\NPC}[1][]{ {{\mathcal{N}}_{#1}^{\mathsc{pc}} } }

\newcommand{\Uxi}{ {\mathcal{U}_\xi} }

\newcommand{\error}{ {{\mathcal{E}}} }

\newcommand{\Ske}[1][]{ { \Sigma^{(\error)}_{#1} }  }
\newcommand{\vke}{ \vk{\error} }

\newcommand{\mout}{M^{(\textrm{out})}}
\newcommand{\mxi}{M^{(\xi)}}
\newcommand{\mPC}{M^{(\textsc{pc})}}

\newcommand{\sol}{\mathrm{sol}}

%
%
%
\newcommand{\spacing}[2]
{ 
{	
	\mathrlap{#1}
	\hphantom{#2} 
}
}
%
%
%

\newcommand{\pd}[3][ ]{\frac{\partial^{#1} #2}{\partial #3^{#1} } }
\newcommand{\od}[3][ ]{\frac{\mathrm{d}^{#1} #2}{\mathrm{d} #3^{#1} } }
\newcommand{\vd}[3][ ]{\frac{ \delta^{#1} #2}{ \delta #3^{#1}} } 

%
%

\renewcommand{\Re}{\mathop{ \mathrm{Re}}\nolimits}
\renewcommand{\Im}{\mathop{ \mathrm{Im} }\nolimits}
\newcommand{\im}{\mathrm{i} }

%
%

\newcommand{\triu}[2][1]{\begin{pmatrix} #1 & #2 \\ 0 & #1 \end{pmatrix}}
\newcommand{\tril}[2][1]{\begin{pmatrix} #1 & 0 \\ #2 & #1 \end{pmatrix}}
\newcommand{\diag}[2]{\begin{pmatrix} #1 & 0 \\ 0 & #2 \end{pmatrix}}
\newcommand{\offdiag}[2]{\begin{pmatrix} 0 & #1 \\ #2 & 0 \end{pmatrix}}

\newcommand{\striu}[2][1]{\begin{psmallmatrix} #1 & #2 \\ 0 & #1 \end{psmallmatrix}}
\newcommand{\stril}[2][1]{\begin{psmallmatrix} #1 & 0 \\ #2 & #1 \end{psmallmatrix}}
\newcommand{\sdiag}[2]{\begin{psmallmatrix*}[c] #1 & 0 \\ 0 & #2 \end{psmallmatrix*}}
\newcommand{\soffdiag}[2]{\begin{psmallmatrix*}[r] 0 & #1 \\ #2 & 0 \end{psmallmatrix*}}
\newcommand{\stwomat}[4]{\begin{psmallmatrix*} #1 & #2 \\ #3 & #4 \end{psmallmatrix*}}

%
%

\newcommand{\mathsc}[1]{ {\text{\normalfont\scshape#1}} }

%
%

\newcommand{\oldnorm}[2][]
{
	 {\left\|  #2 \right\|_{#1} } 
}

\newcommand{\twomat}[4]
{
\begin{pmatrix}
	#1	&	#2	\\
	#3	&	#4
\end{pmatrix}
}

\newcommand{\Twomat}[4]
{
	\left(
		\begin{array}{ccc}
			#1	&&	#2	\\
			\\
			#3	 &&	#4
		\end{array}
	\right)
}

\newcommand{\twovec}[2]
{
	\left(
		\begin{array}{c}
			#1		\\
			#2
		\end{array}
	\right)
}

\newcommand{\Twovec}[2]
{
	\left(
		\begin{array}{c}
			#1		\\
			\\
			#2
		\end{array}
	\right)
}

%
%

\tikzset{->-/.style={decoration={
  markings,
  mark=at position .55 with {\arrow{triangle 45}} },postaction={decorate}}
}

%
%
\newcommand{\FigPhaseA}[1][0.7]{
\begin{tikzpicture}[scale={#1}]
\path[fill=gray!20,opacity=0.5]	(0,0) rectangle(4,4);
\path[fill=gray!20,opacity=0.5]    (-4,-4) rectangle(0,0);
\path[fill=gray,opacity=0.5]		(-4,0) rectangle (0,4);
\path[fill=gray,opacity=0.5]		(0,-4) rectangle (4,0);
\draw[fill] (0,0) circle[radius=0.075];
\draw[thick,->,>=stealth] 	(0,0) -- (2,0);
\draw	[thick]    	(2,0) -- (4,0);
\draw[thick,->,>=stealth]	(-4,0) -- (-2,0);
\draw[thick]		(-2,0) -- (0,0);
\draw[thin,dashed]	(0,4) -- (0,-4);
\node at (2,2) 	{$e^{2it\theta} \gg 1 $};
\node at (-2,-2)  {$e^{2it\theta} \gg 1 $};
\node at (-2,2)	{$e^{2it\theta} \ll 1 $};
\node at (2,-2)	{$e^{2it\theta} \ll 1 $};
\node[above] at (0,4)	{$ \sgnt = +1 $};
\node[below] at (2,-0.05)		{$\posint$};
\node[below] at (-2,-0.05)		{$\negint$};
\node[below right] at (0,0)		{$\xi$};
\end{tikzpicture}
}

%
%

\newcommand{\FigPhaseB}[1][0.7]{
\begin{tikzpicture}[scale={#1}]
\path[fill=gray,opacity=0.5]	(0,0) rectangle(4,4);
\path[fill=gray,opacity=0.5]    (-4,-4) rectangle(0,0);
\path[fill=gray!20,opacity=0.5]		(-4,0) rectangle (0,4);
\path[fill=gray!20,opacity=0.5]		(0,-4) rectangle (4,0);
\draw[fill] (0,0) circle[radius=0.075];
\draw[thick,->,>=stealth] 	(0,0) -- (2,0);
\draw	[thick]    	(2,0) -- (4,0);
\draw[thick,->,>=stealth]	(-4,0) -- (-2,0);
\draw[thick]		(-2,0) -- (0,0);
\draw[thin,dashed]	(0,4) -- (0,-4);
\node at (2,2) 	{$e^{2it\theta} \ll 1 $};
\node at (-2,-2)  {$e^{2it\theta} \ll 1 $};
\node at (-2,2)	{$e^{2it\theta} \gg 1 $};
\node at (2,-2)	{$e^{2it\theta} \gg 1 $};
\node[above] at (0,4)	{$ \sgnt = -1 $};
\node[below] at (-2,-0.05)		{$\posint$};
\node[below] at (2,-0.05)		{$\negint$};
\node[below right] at (0,0)		{$\xi$};
\end{tikzpicture}
}

%
%

\newcommand{\posDBARcontours}{
\resizebox{0.45\textwidth}{!}{
\begin{tikzpicture}
\path [fill=gray!20] (0,0) -- (-4,4) -- (-4.5,4) -- (-4.5,-4) -- (-4,-4) -- (0,0);
\path [fill=gray!20] (0,0) -- (4,4) -- (4.5,4) -- (4.5,-4) -- (4,-4) -- (0,0);
%
\draw [help lines] (-4.5,0) -- (4.5,0);
\draw [thick][->-] (-4,4) -- (0,0);
\draw [thick][->-] (0,0) -- (4,4);
\draw [thick][->-] (0,0) -- (4,-4);
%
\foreach \pos in { 
		(-1.5,1.4), (-1.5,-1.4), (-3.5,2.5), 
		(-3.5,-2.5), (3,2), (3,-2)}
\draw[color=white, fill=white] \pos circle [radius=.2];
%
\node[left] at (3,3) {$\Sigma_{1}\,$};
\node[right] at (-3,3) {$\Sigma_{2}$};
\node[right] at (-3,-3) {$\,\Sigma_{3}$};
\node[left] at (3,-3) {$\Sigma_{4}\,$};
\draw[fill] (0,0) circle [radius=0.025];
\node[below] at (0,0) {$\xi$};
\node[above] at (0,4) {$\sgnt = +1$};
%
\node at (1,.4) {$\Omega_{1}$};
\node at (0,1.08) {$\Omega_{2}$};
\node at (-1,.4) {$\Omega_{3}$};
\node at (-1,-.4) {$\Omega_{4}$};
\node at (0,-1.08) {$\Omega_{5}$};
\node at (1,-.4) {$\Omega_{6}$};
%
\node[left] at (4.5,0.8) {$\tril{-R_1 e^{-2it \theta}} $ };
\node[right] at (-4.5,0.8) {$\triu{ -R_3 e^{2it \theta}} $ };
\node[right] at (-4.5,-0.8) {$\tril{ R_4 e^{-2it \theta}} $ };
\node[left] at (4.5,-0.8) {$\triu{ R_6 e^{2it \theta }} $};
\node at (0,2.5) {$\diag{1}{1}$};
\node at (0,-2.5) {$\diag{1}{1}$};
\end{tikzpicture}
}
}

%
%

\newcommand{\negDBARcontours}{
\resizebox{0.45\textwidth}{!}{
\begin{tikzpicture}
\path [fill=gray!20] (0,0) -- (-4,4) -- (-4.5,4) -- (-4.5,-4) -- (-4,-4) -- (0,0);
\path [fill=gray!20] (0,0) -- (4,4) -- (4.5,4) -- (4.5,-4) -- (4,-4) -- (0,0);
%
\draw [help lines] (-4.5,0) -- (4.5,0);
\draw [thick][->-] (-4,4) -- (0,0) ;
\draw [thick][->-](-4,-4)--(0,0);
\draw [thick][->-] (0,0) -- (4,4);
\draw [thick][->-] (0,0) -- (4,-4);
%
\foreach \pos in { 
		(-1.5,1.4), (-1.5,-1.4), (-3.5,2.5), 
		(-3.5,-2.5), (3,2), (3,-2)}
\draw[color=white, fill=white] \pos circle [radius=.2];
%
\node[left] at (3,3) {$\Sigma_{2}\,$};
\node[right] at (-3,3) {$\Sigma_{1}$};
\node[right] at (-3,-3) {$\,\Sigma_{4}$};
\node[left] at (3,-3) {$\Sigma_{3}\,$};
\draw[fill] (0,0) circle [radius=0.025];
\node[below] at (0,0) {$\xi$};
\node[above] at (0,4) {$\sgnt = -1$};
%
\node at (1,.4) {$\Omega_{3}$};
\node at (0,1.08) {$\Omega_{2}$};
\node at (-1,.4) {$\Omega_{1}$};
\node at (-1,-.4) {$\Omega_{6}$};
\node at (0,-1.08) {$\Omega_{5}$};
\node at (1,-.4) {$\Omega_{4}$};
%
\node[left] at (4.5,0.8) {$\triu {-R_3 e^{2it \theta}} $ };
\node[right] at (-4.5,0.8) {$\tril{ -R_1 e^{-2it \theta}} $ };
\node[right] at (-4.5,-0.8) {$\triu{ R_6 e^{2it \theta}} $ };
\node[left] at (4.5,-0.8) {$\tril{ R_4 e^{-2it \theta }} $};
\node at (0,2.5) {$\diag{1}{1}$};
\node at (0,-2.5) {$\diag{1}{1}$};
\end{tikzpicture}
}
}

%
%

\newcommand{\FigPCjumps}{
\resizebox{0.45\textwidth}{!}{
\begin{tikzpicture}
\draw [help lines] (-4,0) -- (4,0);

\draw[->-] (-3,3) -- (0,0);
\draw[->-] (-3,-3) -- (0,0);
\draw[->-] (0,0) -- (3,3);
\draw[->-] (0,0) -- (3,-3);

\node[left] at (2,2) {$\Sigma_{1}\,$};
\node[right] at (-2,2) {$\Sigma_{2}$};
\node[right] at (-2,-2) {$\,\Sigma_{3}$};
\node[left] at (2,-2) {$\Sigma_{4}\,$};
\draw[fill] (0,0) circle [radius=0.025];
\node[below] at (0,0) {$\xi$};
\node[above] at (0,3) {$\sgnt = 1$};
%
\node at (1,.4) {$\Omega_{1}$};
\node at (0,1.08) {$\Omega_{2}$};
\node at (-1,.4) {$\Omega_{3}$};
\node at (-1,-.4) {$\Omega_{4}$};
\node at (0,-1.08) {$\Omega_{5}$};
\node at (1,-.4) {$\Omega_{6}$};
%
\node[right] at (2.0,1.5) {$\tril[1] { s_\xi }$}; 
\node[right] at (2.0,-1.5) {$\triu[1]{ r_\xi } $}; 
\node[left] at (-2.0,1.5) {$\triu[1]{ \dfrac{r_\xi}{1+r_\xi s_\xi} } $}; 
\node[left] at (-2.0,-1.5) {$\tril[1] { \dfrac{s_\xi}{1+r_\xi s_\xi} } $}; 
%
\end{tikzpicture}
}}

%
%
\newcommand{\FigPCjumpsB}{
\resizebox{0.45\textwidth}{!}{
\begin{tikzpicture}
\draw [help lines] (-4,0) -- (4,0);

\draw[->-] (-3,3) -- (0,0)  ;
\draw[->-] (-3,-3) -- (0,0);
\draw[->-] (0,0) -- (3,3);
\draw[->-] (0,0) -- (3,-3);

\node[left] at (2,2) {$\Sigma_{2}\,$};
\node[right] at (-2,2) {$\Sigma_{1}$};
\node[right] at (-2,-2) {$\,\Sigma_{4}$};
\node[left] at (2,-2) {$\Sigma_{3}\,$};
\draw[fill] (0,0) circle [radius=0.025];
\node[below] at (0,0) {$\xi$};
\node[above] at (0,3) {$\sgnt = -1$};
%
\node at (1,.4) {$\Omega_{3}$};
\node at (0,1.08) {$\Omega_{2}$};
\node at (-1,.4) {$\Omega_{1}$};
\node at (-1,-.4) {$\Omega_{6}$};
\node at (0,-1.08) {$\Omega_{5}$};
\node at (1,-.4) {$\Omega_{4}$};
%
\node[right] at (2.0,1.5) {$\triu[1]{ \dfrac{r_\xi}{1+r_\xi s_\xi} }$}; 
\node[right] at (2.0,-1.5) {$\tril[1] { \dfrac{s_\xi}{1+r_\xi s_\xi}} $}; 
\node[left] at (-2.0,1.5) {$\tril[1] { s_\xi }$}; 
\node[left] at (-2.0,-1.5) {$\triu[1]{ r_\xi }$}; 
%
\end{tikzpicture}
}}

%
%

%
%

\newcommand{\solfig}{
\hspace*{\stretch{1}}
\begin{tikzpicture}[scale=0.5]						
\coordinate (shift) at (-40:1.5);
\coordinate (x1) at (0.7,0);
\coordinate (x2) at (2.1,0);
\coordinate (topleft) at ($(x1)+(95:6)$);
\coordinate (bottomleft) at ($(x1)+(-120:6)$);
\coordinate (topright) at ($(x2)+(60:6)$);
\coordinate (bottomright) at ($(x2)+(-85:6)$);
\coordinate (C) at ($ (x1)!.5!(x2)+(77.5:.5)$);

\begin{scope}
  \clip (-4,-4) rectangle (4,4);
  \path[name path=top] (-4,4) -- (4,4);
  \path[name path=bottom] (-4,-4) -- (4,-4);
  \path[name path=right] (4,4) -- (4,-4);  	
  \path [fill=gray!15] (x1) -- (topleft) -- (topright) -- 
    (x2) -- (bottomright) -- (bottomleft) -- (x1);
  \draw[thick,name path=leftcone] (bottomleft) -- (x1) -- (topleft);
  \draw[thick,name path=rightcone] (bottomright) -- (x2) -- (topright);
    \draw [help lines, name path=axis][->] (-4,0) -- (3.3,0)
    	node[label=0:$x$] {};
    \draw [help lines][->] (-0.5,-4) -- (-0.5,3.0)
    	node[label=90:$t$] {};
    \path [name intersections={of= leftcone and top, by=TL}];
    \path [name intersections={of= rightcone and right, by=TR}];
    \path [name intersections={of= leftcone and bottom, by=BL}];
    \path [name intersections={of= rightcone and bottom, by=BR}];
\end{scope}
    \node [fill=black, inner sep=1.5pt, circle, label=-70:$x_2$] at (x2) {};
    \node [fill=black, inner sep=1.5pt, circle, label=-177:$x_1$] at (x1) {};
	\node [label=180:${x-v_1 t = x_1}$] at (TL) {};
	\node [label=90:${x-v_2 t = x_2}$] at (TR) {};
    \node [label=170:${x-v_2 t = x_1}$] at (BL) {};
    \node [label=10:${x-v_1 t = x_2}$] at (BR) {};
    \node at (C) {$\mathcal{S}$};
\end{tikzpicture}
\hspace*{\stretch{1}}
\begin{tikzpicture}[scale=0.5]			
	\coordinate (v1) at (0.75,0);
	\coordinate (v2) at (-2,0);
	\path [fill=gray!15] ($(v1)+(0,-1)$) -- ($(v1)+(0,7)$) 
	--($(v2)+(0,7)$) -- ($(v2)+(0,-1)$) --  ($(v1)+(0,-1)$);
	\draw[thick] ($(v1)+(0,-.4)$) -- ($(v1)+(0,7)$);
	\draw[thick] ($(v2)+(0,-.4)$) -- ($(v2)+(0,7)$);
	\draw[help lines][->] (-4,0) -- (4,0) 
		node[label=0:$\Re \lam$] {};
	\node [label=${-v_1/4}$] at (.75,-1.2) {};
	\node [label=${-v_2/4}$] at (-2,-1.2) {};
\node [fill=black, inner sep = 1pt, circle, label=-90:$\lambda_1$] at (3,6) 			{};
\node [fill=black, inner sep = 1pt, circle, label=-90:$\lambda_2$] at (1.5,5) 		{};
\node [fill=black, inner sep = 1pt, circle, label=-90:$\lambda_3$] at (0,5.5) 		{};
\node [fill=black, inner sep = 1pt, circle, label=-90:$\lambda_5$] at (-3.3,4.4) 	{};
\node [fill=black, inner sep = 1pt, circle, label=-90:$\lambda_8$] at (-0.8,0.8) 	{};
\node [fill=black, inner sep = 1pt, circle, label=-90:$\lambda_4$] at (-2.5,6.5) 	{};
\node [fill=black, inner sep = 1pt, circle, label=-90:$\lambda_6$] at (-1.6,3) 		{};
\node [fill=black, inner sep = 1pt, circle, label=-90:$\lambda_9$] at (2,1) 			{};
\node [fill=black, inner sep = 1pt, circle, label=-90:$\lambda_7$] at (-3.8,1.1) 	{};
\node [fill=black, inner sep = 1pt, circle, label=-90:$\lambda_{10}$] at (3.8,2.3) 	{};
\end{tikzpicture}
\hspace*{\stretch{1}}
}

%
%

\section{Introduction}

In this paper, we will prove global well-posedness for the derivative nonlinear Schr\"{o}dinger equation (DNLS)
\begin{subequations}
\begin{align}
\label{DNLS1}
i u_t + u_{xx} - i \varepsilon (|u|^2 u)_x &=0\\
\label{data1}
u(x,t=0) &= u_0
\end{align}
\end{subequations}
for initial data in a dense and open subset  of the  weighted Sobolev  space $H^{2,2}(\R)$. This set contains both $0$ and initial data of arbitrarily large $L^2$-norm. It is spectrally determined and includes initial data with at most finitely many bright  soliton components and no algebraic solitons.  Here 
$\eps = \pm 1$ and 
$H^{2,2}(\R)$ is the completion of $C_0^\infty(\R)$ in the norm
$$ \norm[H^{2,2}]{u} = \left( \norm[L^2]{(1+(\dotarg))^2 u(\dotarg)}^2 + \norm[L^2]{u''}^2 \right)^{1/2}. $$

It is known that the Cauchy problem is locally well-posed in $H^{1/2}(\R)$ (see, for example, Takaoka \cite{Takaoka99}) and globally well-posed for small data \cite{HO92,Wu14}. Precisely, for any $u_0\in H^{1/2}(\R)$ such that $\| u_0 \|_{L^2} < \sqrt{4\pi}$ {(or $\|u_0\|_{L^2} = \sqrt{4\pi}$ with negative momentum)}, 
there exists a unique solution $u\in C(\R, H^{1/2}(\R))$ \cite{FHI17,GuoWu17}.  A key structural property of DNLS  discovered by Kaup and Newell \cite{KN78} is that it is integrable by inverse scattering: that is, there is a linear spectral problem with $u(x,t)$ as potential whose spectral data (consisting of a reflection coefficient, describing the continuous spectrum of 
the linear problem,  together with eigenvalues and norming constants, describing the discrete spectrum of the linear problem) evolve linearly under the flow. This linearizing transformation, together with an inverse map expressed through a Riemann-Hilbert problem defined in terms of the spectral data, gives a method to integrate the equation explicitly. 

In recent works \cite{LPS15,LPS16} (referred to as Papers I and II), we used the inverse scattering method to prove
global existence and  long-time behavior of solutions to DNLS
for initial conditions in  $H^{2,2}(\R)$,  restricting  to  initial conditions  that  do not support solitons.  We proved that the amplitude of the solution
 decays like the solution of the linear problem, namely $|t|^{-1/2}$ as $|t| \to \infty$, and its phase behaves like the phase of the free dynamics modified by  a logarithmic correction. These results are analogous to those of Deift and Zhou \cite{DZ03} for the defocussing NLS, and improve earlier results of Kitaev-Vartanian \cite{KV97} on DNLS. 
The asymptotic state is fully described in  terms of the scattering data associated to  the initial condition.  
Using a different approach, Pelinovsky and Shimabukuro  proved global existence  to the DNLS equation for initial conditions that do not support solitons in \cite{PS17}
and recently complemented  their study allowing a finite number of discrete eigenvalues \cite{PSS17}.

In a subsequent paper \cite{JLPS18}, we use the results obtained here to prove soliton resolution and fully describe the dispersive part of the solution.  
Soliton resolution refers to the property that the solution decomposes into the sum of a finite number of separated solitons and a radiative part as $|t| \to \infty$.  The limiting solitons parameters  are slightly modulated, due to the soliton-soliton and soliton-radiation interactions.  The dispersive part  contains two components,  one coming from the continuous  spectrum and another one from the interaction of the discrete and continuous spectrum.

To prove global well-posedness of \eqref{DNLS1}, we first note that if $u(x,t)$ solves \eqref{DNLS1} with $\eps = +1$, then $u(-x,t)$ solves \eqref{DNLS1} with $\eps = -1$. For this
reason it suffices to prove well-posedness for one sign. We will choose $\eps=-1$ since this choice is most amenable to solution by inverse scattering. 
Because we want to establish long-time  formulas and results for either sign of $\eps$, we discuss both cases. Theorems \ref{thm:dense}, \ref{thm:R}, \ref{thm:I}, and \ref{thm:GWP} are only proved here for $\eps=-1$. Complete proofs for both signs can be found in \cite{JLPS17}.

Equation \eqref{DNLS1} is gauge-equivalent to the equation
\begin{subequations}
\begin{align}
\label{DNLS2}
iq_t + q_{xx} + i \eps q^2 \bar q_x + \frac{1}{2} |q|^4 q = 0, \\
\label{data}
q(x,t=0) = q_0(x)
\end{align}
\end{subequations}
via the gauge transformation\footnote{  In papers I and II, we use {{a}} slightly different gauge transformation, namely $\exp\left(- i\eps \int_{-\infty}^x |u(y)|^2 \, dy \right) u(x)$. Both
transformations are clearly equivalent, up to the constant phase factor $\exp\left( i\eps \int_{-\infty}^{\infty} |u_0(y)|^2 \, dy \right)$. The current transformation has the advantage
of slightly  simplifying some formulae in the {analysis of long time behavior \cite{JLPS18}.}} 
\begin{equation}
\label{G}
q(x)  :=\calG (u)(x) = \exp\left( i\eps \int_x^{\infty} |u(y)|^2 \, dy \right) u(x).
\end{equation}
This nonlinear, invertible mapping is an isometry of $L^2(\R)$, maps soliton solutions to soliton solutions, and maps dense open sets to dense open sets in weighted Sobolev spaces.
Consequently,  global well-posedness for \eqref{DNLS2} on an open and dense set $U$ in $H^{2,2}(\R)$ containing data of arbitrary $L^2$-norm implies 
global well-posedness of \eqref{DNLS1} on a subset $\calG^{-1}(U)$ of $H^{2,2}(\R)$ with the same properties. We will establish the global well-posedness result for \eqref{DNLS2}.

Our analysis exploits the discovery of Kaup and Newell \cite{KN78} that \eqref{DNLS2} has the Lax representation
\begin{equation*}
\begin{aligned}
L	&=	-i\lam \sigma_3 + Q_\lam - \frac{i}{2} \sigma_3 Q^2\\
A	&=	2 \lam L + i \lam (Q_\lam)_x + \frac{1}{2}\left[ Q_x , Q \right] + \frac{i}{4}\sigma_3 Q^4
\end{aligned}
\end{equation*}
where
\begin{equation*}
\sigma_3 = \begin{pmatrix} 1 & 0 \\ 0 & -1 \end{pmatrix}, \quad
Q(x) =  
\begin{pmatrix} 0 & q(x) \\  \eps  \overline{q(x)} & 0 \end{pmatrix},
\quad
Q_\lam(x) = 
\begin{pmatrix} 0 & q(x) \\  \eps \lam \overline{q(x)} & 0 \end{pmatrix}.
\end{equation*}
That is, a
smooth function $q(x,t)$ solves \eqref{DNLS2} if and only if the operator identity
$$ L_t - A_x + [L,A] = 0 $$ 
holds for the operators above with $q=q(x,t)$ (so that both operators depend on $t$).

We will consider the spectral problem 
\begin{equation}
\label{LS}
\Psi_x = L\Psi
\end{equation} for  $\lam \in \C$ and $2 \times 2$ matrix-valued solutions $\Psi(x;\lam)$. As we will see, \eqref{LS} defines a map $\calR$ from $q \in H^{2,2}(\R)$ to spectral data, and has an inverse $\calI$ defined by a Riemann-Hilbert problem (Problem \ref{RHP2} below) which recovers the potential $q(x)$. Moreover, the spectral data for a solution 
 $q = q(x,t)$ of \eqref{DNLS2} obey a linear law of evolution. Thus the solution 
$\calM$ for the Cauchy problem \eqref{DNLS2}-\eqref{data} is given by
\begin{equation}
\label{sol.op}
\calM (q_0,t) = \left( \calI \circ\Phi_t \circ \calR \right) q_0 
\end{equation}
where $\Phi_t$  represents  the linear evolution on spectral data. To state our results, we first describe the set $U$ and the maps $\calR$,  $\Phi_t$, and $\calI$ in greater detail.

The direct scattering map $\calR$ maps $q \in H^{2,2}(\R)$ into spectral data defined by special solutions of \eqref{LS}. First
note that, if $q=0$, all solutions of \eqref{LS} are matrix multiples of $e^{-i x \lam \sigma_3}$ and 
therefore bounded if $\lam \in \R$. If $\lam \in \R$ and $q \in L^1(\R) \cap L^2(\R)$, there exist
unique solutions $\Psi^\pm(x;\lam)$ of $L\Psi = 0$,  the \emph{Jost solutions}, with
$\lim_{x \rarr \pm \infty} \Psi^\pm(x;\lam) e^{i\lam x \sigma_3} =I$, where $I$ denotes the $2\times 2$ identity matrix.  Any solution of $L\psi=0$ has constant determinant and any nonsingular solution may be obtained from any other through postmultiplication by a constant matrix. In particular
\begin{equation}
\label{Jost.T}
\Psi^+(x;\lam) = \Psi^-(x;\lam) T(\lam), \quad
T(\lam) \coloneqq
\begin{pmatrix} 
\alpha(\lam) 	& \beta(\lam) \\
\bbeta(\lam) 	&	\balpha(\lam)
\end{pmatrix}
\end{equation}
where, since $\det \Psi^\pm = 1$,
$$\alpha \balpha - \beta \bbeta = 1.$$
Symmetries of \eqref{LS} imply that, also
\begin{equation}
\label{alpha.beta.sym}
\balpha(\lam) = \overline{\alpha(\lam)}, \quad \bbeta(\lam) = \eps \lam \overline{\beta(\lam)}.
\end{equation}
We introduce the reflection coefficient
\begin{equation} \label{scatt-rho}
  \rho(\lam) = \beta(\lam)/\alpha(\lam) 
 \end{equation}
and note that 
$ 1-\eps \lam |\rho(\lam)|^2 = \dfrac{1}{|\alpha(\lam)|^2} > 0. $
We then define  the set
\begin{equation}
\label{S}
P = \left\{ \rho\in H^{2,2}(\R): 1- \eps \lam |\rho(\lam)|^2 > 0 \right\} .
\end{equation}
A careful analysis of \eqref{LS} shows that if $q \in L^1(\R) \cap L^2(\R)$, $\balpha$ has an analytic continuation to $\C^+$, and $\balpha(z) \rarr 1$ as $|z| \rarr \infty$. If $q_0 \in H^{2,2}(\R)$, zeros of the corresponding function $\balpha$ in $\C^+$ signal the presence of \emph{bright solitons}, as our results on large-time asymptotics show. Zeros of $\balpha$ on $\R$ are known to occur for algebraic solitons \cite{KN78}.

\begin{definition}
\label{def:U}
We denote by $U$ the subset of $H^{2,2}(\R)$ consisting of  functions $q$ for which $\balpha$ has no zeros on $\R$ and at most finitely many simple zeros in $\C^+$.
\end{definition}

We  prove the following theorem in Section \ref{sec:direct.generic}.

\begin{theorem}
\label{thm:dense}
For $\eps = -1$, the set $U$ is open and dense in $H^{2,2}(\R)$.
\end{theorem}

Clearly $U = \bigcup_{N=0}^\infty \, U_N$ where $U_N$ consists of those $q$ for which $\balpha$ has exactly $N$ zeros in $\C^+$. If $N \neq 0$, we denote by $\lam_1, \ldots, \lam_N$ the simple zeros of $\balpha$ in $\C^+$. Associated to each $\lam_j$ is a \emph{norming constant} $C_j \in \C \setminus \{0\} := \C^\times$ which can be described in terms of the Jost solutions (see section \ref{sec:direct}). Thus, the spectral problem \eqref{LS} associates to each $q \in U_N$ the spectral data
$$ 
\calD(q) = 
\left( 
	\rho, \{ (\lam_j, C_j )\}_{j=1}^N 
\right) \in 
P \times \left(\C^+ \times \C^\times\right)^N.
$$
The map $q \mapsto \calD(q)$ is called the \emph{direct scattering map}. To describe its properties, let 
\begin{equation}
\label{V}
V=\bigcup_{N=0}^\infty V_N, \quad V_N = P \times \left( \C^+ \times \C^\times \right)^N.
\end{equation} 
For an element $\calD$ of $V$, write 
\begin{equation*}
\Lambda = \Lambda^+ \cup \overline{\Lambda^+}, \quad \Lambda^+ = \left\{ \lam_1, \ldots , \lam_N \right\}
\end{equation*}
and define
\begin{equation}
\label{distances}
d_\Lambda = \inf_{\lam, \mu \in \Lam, \, \lam \neq \mu} |\lam - \mu|. 
\end{equation}
Note that for any $\lambda_j \in \Lambda, |\imag \lambda_j|  \geq (1/2) d_\Lambda$.
We say that a subset $V'$ of $V_N$ is bounded if:
\begin{itemize}
\item[(i)]		$1+\lam |\rho(\lam)|^2 \geq c_1>0$ for a fixed strictly positive constant $c_1$,
\item[(ii)]	$\sup_j \left(|C_j| + |\lam_j|\right) \leq C$ for a fixed strictly positive constant $C$,
\item[(iii)]   $d_\Lambda \geq c_2 > 0$ for a fixed strictly positive constant $c_2$.
\end{itemize} 
We say that a subset $U'$ of $U$ is bounded if the set $\{\calD(q): q \in U\}$ is a bounded subset of $V_N'$ and if, also, $U'$ is a bounded subset of $H^{2,2}(\R)$. We will prove:

\begin{theorem}
\label{thm:R}
For $\eps=-1$, the direct scattering map 
\begin{equation}
\label{def:R} 
 \calR: U \rarr V 
\end{equation}
maps bounded subsets of $U_N$ into bounded subsets of $V_N$ for each $N$,  and is
uniformly Lipschitz continuous on bounded subsets of $U_N$. 
\end{theorem}

\begin{remark}
One-soliton solutions corresponding to eigenvalue $\lam = |\lam| e^{i\phi} \in \C^+$
have $L^2$-norm $\sqrt{4(\pi - \phi)}$ if $\eps=1$ and $\sqrt{4\phi}$ if $\eps=-1$. In the limit $\phi \darr 0$ ($\eps=+1$) or $\phi \uarr \pi$ ($\eps=-1$), $\lambda$ approaches the real axis and the bright soliton becomes an algebraic soliton. As we will show in \cite{JLPS18} (see also \cite{JLPS17}) $N$-soliton solutions separate into a train of one-soliton solutions so that an $N$-soliton solution may have $L^2$-norm in  the interval $(0,N\sqrt{4\pi})$. It follows that the set $U_N$ contains 
elements of $L^2$ norm arbitrarily close to $0$.
\end{remark}

It follows from the Lax representation 
 that the spectral data for a solution
$q(x,t)$ of \eqref{DNLS2} obey the linear evolution
\begin{equation*}
\dot{\rho}(\lam,t) = -4i \lam^2 \rho(\lam,t), \quad	\dot{\lam_j} = 0, \quad \dot{C_j} = -4i \lam_j^2 C_j.
\end{equation*}

\begin{definition}
For each nonnegative integer $N$, $t \in \R$, $\rho \in P$, and $\{ \lam_j, C_j \}_{j=1}^N \subset \C^+ \times \C^\times$, the linear evolution $\Phi_t: V_N \rarr V_N$ is given by
$$ \Phi_t \left(\rho, \{ (\lam_j, C_j) \}_{j=1}^N \right) = \left( e^{-4i(\dotarg)^2 t}  \rho(\dotarg), \{ ( \lam_j, C_j e^{-4i\lam_j^2 t } )\}_{j=1}^N \right). $$
\end{definition}
It is easy to see that the map $\Phi_t$ preserves $V_N$ for each $N$ and is jointly continuous in $t$ and the data $\left(\rho, \{ (\lam_j, C_j) \}_{j=1}^N \right)$.

Let us introduce the phase function 
\begin{equation*}
\theta(x,t,\lam) = -\left(\frac{x}{t} \lam + 2 \lam^2\right).
\end{equation*}
The (time-dependent)  inverse spectral problem is defined by a Riemann-Hilbert problem (RHP) and a reconstruction formula.

\begin{RHP}
\label{RHP2}
Given $x,t \in \R$, $\rho \in P$ and {$\{ (\lam_j, C_j) \}_{j=1}^N \subset \C^{+} \times \C^\times$}, find a matrix-valued function 
$$N(\lam;x,t) : \C \setminus (\R \cup \Lam 
) \rarr SL(2,\C)$$ 
with the following properties:
\medskip
\begin{itemize}
\item[(i)]		$N_{22}(\lam;x,t) = \overline{N_{11}(\lambar;x,t)}$, 
$N_{21}(\lam;x,t) = \eps \lam \overline{N_{12}(\lambar;x,t)}$, 
\smallskip
\item[(ii)]	$N(\lam;x,t) = \stril{q^*} + \bigO{\dfrac{1}{\lam}}$ as $|\lam| \rarr \infty$, 
\smallskip
\item[(iii)]	$N$ has continuous boundary values $N_\pm$ on $\R$ and
				$$ N_+(\lam;x,t) = N_-(\lam;x,t) e^{it\theta \ad \sigma_3} v(\lam),
				\quad 
				v(\lam) = \begin{pmatrix}
							1-{\eps} \lam |\rho(\lam)|^2 	&	\rho(\lam)	\\
								-\eps \lam \overline{\rho(\lam)}	&	1
							\end{pmatrix}
				$$
\smallskip
\item[(iv)] 	{
				$N(\lam;x,t)$ has simple poles at each point $p \in \Lam$:  
				$$
				\Res_{\lam = p} N(\lam;x,t) 
				= \lim_{\lam \rarr p} N(\lam;x,t) e^{it\theta \ad \sigma_3} v(p) 
				$$ 
				where for each $\lam_j \in \Lam^+$
				$$ v(\lam_j)  = \twomat{0}{0} {\lam_j C_j}{0}, \quad 
				v(\overline{\lam_j}) = \twomat{0}{ \eps  \overline{C_j} }{0}{0}. 
				$$
				}
\end{itemize}
\end{RHP}

Given the solution $N(\lam;x,t)$ of Problem \ref{RHP2}, we can recover the solution
$q(x,t)$ from the reconstruction formula
\begin{equation}
\label{q.lam}
	q(x,t)= \lim_{\lam \rarr \infty}  2i \lam  N_{12}(\lam;x,t).
\end{equation}
\begin{remark}
If there exists at least one solution to Problem \ref{RHP2}, it is unique. 
For, given two solutions $N$ 
and $N^\sharp$, it is easy to see that $N \left(N^\sharp\right)^{-1}$ has no jumps across 
$\R$ and no singularities at $\lam_j$ or $\overline{\lam_j}$. It follows that
$N \left(N^\sharp\right)^{-1}$ is bounded and holomorphic, hence 
$$ N \left(N^\sharp\right)^{-1} = \begin{pmatrix} 1 & 0 \\ \gamma & 1 \end{pmatrix}. $$
Applying the symmetry condition (i) we conclude that $\gamma=0$ and $N=N^\sharp$.
\end{remark}

For each $N$, the \emph{inverse scattering map} is the map 
\begin{align*}
\calI: V_N 	&\rarr U_N		\\
\left(\rho, \left\{ (\lam_j, C_j) \right\}_{j=1}^N \right)	&	\rarr	q
\end{align*}
defined by Problem \ref{RHP2} (with $t=0$) and \eqref{q.lam}.
Our next result  which is the object of Section \ref{sec:inverse} is:

\begin{theorem}
\label{thm:I}
For $\eps=-1$, the inverse scattering map $\calI : V \rarr U$ takes bounded subsets of $V_N$ to bounded subsets of $U_N$ for each $N$, and is uniformly Lipschitz continuous on bounded subsets of $V_N$. Moreover, $\calR \circ \calI$ is the identity map on $V$ and $\calI \circ \calR$ is the identity map on $U$.
\end{theorem}

The composition $\calI \circ \Phi_t$ is given by RHP \ref{RHP2} and the reconstruction formula \eqref{q.lam}.
In his thesis \cite{Lee83}, Lee proved that the right-hand side of eq. \eqref{sol.op} is the correct solution operator
for \eqref{DNLS2} for $q_0 \in U \, \cap \, \calS(\R)$. As a consequence of Lee's result, Theorem \ref{thm:R}, and Theorem
\ref{thm:I}, we conclude:

\begin{theorem}
\label{thm:GWP}
Suppose $\eps=-1$ and let $U$ be the set of $q_0$ given in Definition \ref{def:U}.
\begin{itemize}
\item[(i)]		The Cauchy problem for \eqref{DNLS2} has a unique global
				solution for initial data $q_0 \in U$, 
\item[(ii)]	The solution map $\calM: (q_0,t) \rarr q(x,t)$ is a continuous 
				map from $U \times [0,T]$ to $H^{2,2}(\R)$ for any $T>0$, and
\item[(iii)]	For any $T>0$, 
				$$
					\sup_{t \in [0,T]} \norm[H^{2,2}]{\calM(q_1,t) - \calM(q_2,t)} 
					\lesssim \norm[H^{2,2}]{q_1-q_2}
				$$
				with constant uniform in $q_1,q_2$ in a bounded subset of $U$ and $t \in [0,T]$.
\end{itemize}
\end{theorem}

Together with the local well-posedness result of Takaoka \cite{Takaoka99}, Theorem \ref{thm:GWP} establishes global well-posedness for the DNLS equation in an open and dense subset of $H^{2,2}(\R)$. 

We close this introduction by sketching the content of this paper. After a review of the Beals-Coifman approach to Riemann-Hilbert problems  and  Zhou's uniqueness theorem
\cite{Zhou89} in Section \ref{sec:prelim}, we consider the direct scattering map in Section \ref{sec:direct}. Section \ref{sec:direct} should be read in concert with Section 3 of \cite{LPS15}; the main 
new results of this section are the proof that the set $U$ from Definition \ref{def:U} is open and dense (see Propositions \ref{prop:direct.generic} and Proposition \ref{prop:direct.open}) and that, restricted to the set $U$, the maps from $q$ to the discrete scattering data are Lipschitz continuous (Proposition \ref{prop:direct.discrete.lip}). Together with the results from Section 3 of \cite{LPS15}, they imply Theorem \ref{thm:R}. All of these results are based on analysis of the Volterra integral equations for the normalized Jost solutions $N^\pm$ (\eqref{direct.n1}--\eqref{direct.n2}).

In Section \ref{sec:inverse}, we analyze the inverse map and give  the proof of Theorem \ref{thm:I}. We  replace RHP \ref{RHP2} by an equivalent one for
row vector solutions, RHP \ref{RHP2.row}. 
In Section \ref{sec:EU}, we use Zhou's vanishing theorem \cite[Theorem 9.3]{Zhou89} to prove existence and uniqueness. In Section \ref{sec:Lip}, we reduce the Beals-Coifman integral equations for RHP \ref{RHP2.row} to algebraic-integral 
equations and use these equations together with the reconstruction formula \ref{q.lam} to prove Lipschitz continuity of the inverse scattering map.
The proof of Theorem \ref{thm:unique}, specialized to the case of no solitons, corrects a gap in Paper I. The proof of Lemma 5.9 given there establishes uniqueness of solutions to Problem 5.1, but the details of the existence proof were not provided. 

Appendix \ref{app:BC} collects the Beals-Coifman integral equations for Riemann-Hilbert problems \ref{RHP2.row} and 
\ref{RHPX}.  
Finally, in 
Appendix \ref{app:empty}, we construct initial data for \eqref{DNLS2} with arbitrarily large $L^2$ norm and empty 
discrete spectrum, showing that \eqref{DNLS2} (and hence, by gauge transformation, \eqref{DNLS1}) admits global 
solutions with arbitrarily large $L^2$ norm. 
Our computation is modelled on similar computations by Tovbis-Venakides \cite{TV00} and DiFranco-Miller \cite{dFM08}.

The contents of this  paper and its companion paper \cite{JLPS18} appeared in somewhat different form in the preprint \cite{JLPS17}.

\subsection*{Notation Index}

$~$
\vskip  0.5cm

\subsubsection*{General}

\begin{itemize}
\item[$\Lambda$]	The set $\{ \lam_j , \overline{\lambda}_j \}_{j=1}^N$
\smallskip

\item[$\calF,\ \calF^{-1}$]	
  Fourier transform pair 
  $\left(\calF f\right)(\lam) = \dfrac{1}{\pi}\dint_{-\infty}^\infty e^{-2i\lam x} f(x) \, dx$
\smallskip

\item[$C,\ C^\pm$]				        
	Cauchy integral over a contour; Cauchy projectors

\smallskip

\item[$\sigma_1, \sigma_2, \sigma_3, \sigma_\eps$]
	Pauli matrices $\soffdiag{1}{1}$, 
	$\soffdiag{i}{-i}$, 
	$\sdiag{1}{-1}$, 
	$\soffdiag{1}{\eps}$ respectively

\end{itemize}

\subsubsection*{Maps and Spaces}

\begin{itemize}
\item[$\calG$]				Gauge transformation \eqref{G} linking \eqref{DNLS1} and \eqref{DNLS2}

\item[$\calR,\calI$]		Direct and inverse scattering maps 
\item[$\Phi_t$]				Flow on scattering data 
\item[$\calM$]				Solution operator $\calI \circ \Phi_t \circ \calR$ for \eqref{DNLS2}
\item[$U$]						Spectrally determined, open and dense subset of $H^{2,2}
									(\R)$	(Definition \ref{def:U}) and domain of $\calR$
\item[$V$]						Image of $U$ under $\calR$ (eq. \eqref{V}) and 
									domain of $\calI$
\item[$U_N,V_N$]				$N$-soliton sector of $U$; its image under $\calR$
\item[$P$]						Set of $\rho \in H^{2,2}(\R)$ with 
									$1- \eps \lam|\rho(\lam)|^2 > 0$ defined in \eqref{S}
\end{itemize}

\subsubsection*{Scattering Solutions,  Scattering Data}

\begin{itemize}
			
\item[$\alpha,\beta$]		Scattering coefficients  
									\eqref{alpha.beta.sym}
\item[$\rho$]				Reflection coefficient 
                                                           \eqref{scatt-rho}
\item[$\lam_j, C_j$] 		Zero of $\balpha$ and associated norming constant

\end{itemize}

\subsubsection*{Riemann-Hilbert Problems}

\begin{itemize}

\item[$\Gamma_j,\Gamma_j^*$]	
  Small contour surrounding $\lam_j$ with clockwise orientation; its complex conjugate path
\item[$\R'$] Augmented contour \eqref{R'}	
\item[$N,n$] Matrix and row vector-valued solutions to RHP \ref{RHP2}
\item[$V,\theta$] Jump matrix and phase function for RHP~\ref{RHP2}
\item[$\xi$]						Unique critical point of $\theta$, $\xi=-x/4t$
\end{itemize}

%
%

\section{Preliminaries}
\label{sec:prelim}

We recall several important results of the Beals-Coifman theory of Riemann-Hilbert problems. References include the original paper of Beals-Coifman \cite{BC84}, the seminal paper of Zhou \cite{Zhou89}, and the recent monograph on Riemann-Hilbert problems by Trogdon and Olver \cite{TO16}, especially chapter 2, Sections 2.6 and 2.7. Our discussion follows closely the exposition in \cite{TO16}. 


In order to state the main results of the Beals-Coifman theory, we first 
need to define the class of admissible contours.
We say that a contour $\Gamma$ in $\C$ is \emph{complete} if $\Gamma$  divides $\C\setminus \Gamma$ into two regions $\Omega_+$ and $\Omega_-$ so that  $\Omega_+ \cap \Omega_- = \varnothing$,   $\Omega_+$ lies to the left of $\Gamma$, and $\Omega_-$ lies to the right of $\Gamma$. 

In the following, we will assume that $\Gamma$ is \emph{admissible}:  that is, $\Gamma$ is complete,  the 
set $\gamma_0$ of self-intersection points is finite,  all self-intersections are transverse, and
the following smoothness condition holds: for any $D \Subset \Omega_\pm$, $\dee D$ is piecewise smooth and, if $\dee D$ is unbounded, the infinite component of $\dee D$ is a straight line (this last assumption is far more stringent than necessary but will suffice for our applications). 

If $\Gamma$ is an admissible contour, the Cauchy integral of a function $f \in L^2(\Gamma)$ is given by 
$$ (Cf)(z) = \frac{1}{2\pi i} 
	\int_\Gamma 
		\frac{f(s)}{s-z} \, ds.
$$
We will sometimes write $C_\Gamma$ for $C$ to emphasize the dependence on the contour $\Gamma$.
If $\Gamma$ has no self-intersections and $f \in H^1(\Gamma)$, then $(Cf)(z)$ has H\"{o}lder 
continuous boundary values $C^\pm f$ on $\Gamma$ where
$$ (C^\pm f)(\zeta) = \lim_{\substack{z \rarr \zeta \\ z \in \Omega_\pm}} (Cf)(z). $$
If $\Gamma = \Gamma_1 \cup \ldots \cup \Gamma_\ell$ for non-self-intersecting $\Gamma_i$, we define $H^1(\Gamma)$ 
to be the set of $f \in L^2(\Gamma)$ with $f_i \coloneqq \left. f \right|_{\Gamma_i}$ in $H^1(\Gamma_i)$ for each $i$. 
If  $\Gamma$ is the union of two disjoint contours $\Gamma_1$ and $\Gamma_2$ and $h \in H^1(\Gamma)$, then
\begin{equation*}
\left(C^\pm h\right)(\zeta)=
	\begin{cases}
		\left(C^\pm_{\Gamma_1} h \right)(\zeta)
			+ \left(C_{\Gamma_2} h\right)(\zeta),
			&	\zeta \in \Gamma_1, \\
		\\
		\left(C^\pm_{\Gamma_2} h\right)(\zeta)
			+ \left(C_{\Gamma_1} h\right)(\zeta),
			&	\zeta \in \Gamma_2.
	\end{cases}
\end{equation*}

For any admissible contour, the operators $C^\pm$ extend to bounded operators on $L^2(\Gamma)$ with 
\begin{equation}
\label{Cpm}C^+ - C^- = I,
\end{equation} 
where $I$ denotes the identity operator.

We now give a precise statement of the normalized Riemann-Hilbert problem for the contour $\Gamma$ with jump matrix $v$.  We write $g_\pm \in \dee C(L^2)$ if $(g_+,g_-)$ are a pair of functions in $L^2(\Gamma)$ with $g_\pm = C^\pm h$ for a fixed function $h \in L^2(\Gamma)$.

\begin{RHP}
\label{RHP.model}
Given a $2 \times 2$ matrix-valued function $v$ on $\Gamma$ with 
$v, v^{-1} \in L^\infty(\Gamma)$, 
find an analytic function
$$ M(\dotarg): \C \setminus \Gamma \rarr SL(2,\C) $$
having continuous boundary values 
$$ M_\pm(\zeta) = \lim_{ \substack{ z \rarr \zeta, \\ z \in \Omega_\pm}}
	M(z) $$
for $\zeta \in \Gamma$ satisfying the jump relation $M_+(\zeta) = M_-(\zeta) v(\zeta)$, and 
$$M_\pm(\dotarg) - I \in \dee C(L^2). $$
\end{RHP}

\begin{remark}
We will also consider the case where we seek a row-vector-valued function $M$ with
the same jump relation and the condition 
$$ M_\pm(\cdot) - (1, 0)
\in \dee C(L^2).$$
The results cited below apply to this problem with obvious changes in hypotheses.
\end{remark}

We now derive the Beals-Coifman equations for \eqref{RHP.model}. Suppose that $v$ admits
the factorization
\begin{equation*}
v=(I-w_-)^{-1} (I+w_+),
\end{equation*}
$w_\pm \in L^\infty(\Gamma) \cap L^2(\Gamma)$.  We may then introduce a new matrix-valued (or row vector-valued)  function $\mu$ defined in terms of the putative boundary values $M_\pm$ by
\begin{equation*}
\mu = M_+(I+w_+)^{-1} = M_-(I-w_-)^{-1}, \quad \mu - I \in L^2(\Gamma) .
\end{equation*}
It is easy to see that
$$M_+ - M_- = \mu(w^-+w^+)$$
so that 
\begin{equation}
\label{mu-to-M}
M(z) = I + \int_\Gamma \frac{\mu(s)(w^-(s) + w^+(s))}{s-z} \, \dfrac{ds}{2\pi i}.
\end{equation}
Thus, in order to find $M(z)$ on $\C \setminus \Gamma$, it suffices to compute $\mu$.
Taking the boundary values
we recover that
$$M_\pm(\zeta) = I + C^\pm
		\left[ \mu(\dotarg)
			\left(w^-(\dotarg) + w^+(\dotarg)\right)
		\right](\zeta) .
$$
A short computation using  \eqref{Cpm} shows that $\mu$ solves the \emph{Beals-Coifman integral equation}
\begin{equation}
\label{BC.int}
\mu = I + \calC_w \mu 
\end{equation}
where 
\begin{equation*}
\calC_w h = C^+(hw_-) + C^-(hw_+).
\end{equation*}

A basic result of the Beals-Coifman theory is:

\begin{theorem}
There exists a unique solution $M$ to the  Riemann-Hilbert Problem \ref{RHP.model} if and only if there exists a unique solution $\mu$ with $\mu-I \in L^2(\Gamma)$ to the  Beals-Coifman integral equation
\eqref{BC.int}.
\end{theorem}

For proof and discussion see for example Zhou \cite[Proposition 3.3]{Zhou89} and \cite[Section 2.7]{TO16}. 

%
%

A second important result of the theory concerns problems with residues such as RHP \ref{RHP2}.  Associated to the residues are discrete data $\{ \zeta_j, c_j\}_{j=1}^{2N}$, $\zeta_j \in \C \setminus \Gamma$ and $c_j \in \C^\times$,  which describe the location of residues and the leading term of the Laurent series. 
We restrict to the case of simple poles and assume a special structure (see Problem \ref{RHP.model.disc}(ii)) for the residue conditions which will suffice for our applications. We denote by $\calZ$ the finite set $\{ \zeta_j \}_{j=1}^{2N}$. 

\begin{RHP}
\label{RHP.model.disc}
Find an analytic function 
$$M(\dotarg):\C \setminus \left(\Gamma \cup \calZ\right) \rarr SL(2,\C)$$ 
so that:
\begin{itemize}
\item[(i)]		$M(\dotarg)$ has continuous boundary values 
$$ M_\pm(\zeta) = \lim_{z \rarr \zeta, z \in \Omega_\pm}
	M(z) $$
for $\zeta \in \Gamma$ satisfying the jump relation $M_+(\zeta) = M_-(\zeta) v(\zeta)$ and 
$$M_\pm(\dotarg) - I \in \dee C(L^2). $$
\item[(ii)]	For each $\zeta_j \in \calZ$,  
$$ 
\Res_{z=\zeta_j} M(z) = \lim_{z \rarr \zeta_j} M(z) v_j
$$
where
$$ v_j = 	\begin{pmatrix}
					0		&	0	\\
					c_j	&	0
				\end{pmatrix},
	\quad 1  \leq j \leq N, 
	\qquad
	v_j =		\begin{pmatrix}
					0	& c_j	\\
					0	&	0
				\end{pmatrix},
	\quad N+1 \leq j \leq 2N.
$$
\end{itemize}
\end{RHP}

The residue condition means that for $1 \leq j \leq N$, the first column of 
$M(z)$ has a pole given by $c_j$ times the value of the second column, while, for $N+1 \leq j \leq 2N$, the second column of $M(z)$ has a residue given by $c_j$ times the value of the first column. 

Problem \ref{RHP.model.disc} is equivalent to a Riemann-Hilbert problem with no discrete data but having an augmented contour 
$$ \Gamma' = \Gamma \cup \{ \gamma_j \}_{j=1}^{2N} $$
where each $\gamma_j$ is a simple closed curve in $\C \setminus \Gamma$ surrounding $\zeta_j$ and no other element of $\calZ$. The new curves are given an orientation consistent with the orientation of the original contour $\Gamma$ and so $\Gamma'$ also divides $\C \setminus \Gamma'$ into 
two disjoint sets ${\Omega'}^+$ and ${\Omega'}^-$.  We assume that, with this orientation, the interior of $D_j$ lies in $\Omega_-$ for $1 \leq j \leq n$ and in $\Omega_+$ for $n+1 \leq j \leq 2n$. The additional jump matrices are constructed from the discrete data.  In what follows, $D_j$ is the interior of $\gamma_j$, and $C^\pm$ are the Cauchy projectors for the contour $\Gamma'$. As before, we say that $g_\pm \in \dee C (L^2)$ if
$g_\pm = C^\pm h$ for a fixed function $h \in L^2(\Gamma')$. 

\begin{RHP}
\label{RHP.model.aug}
Find an analytic function 
$$M: \C \setminus {\Gamma}': \rarr SL(2,\C)$$ 
with continuous  boundary values 
$$M_\pm(\zeta) = \lim_{\substack{z \rarr \zeta, \\z \in {\Omega'}^\pm}} M(z),$$
so that $M_\pm(\dotarg) - I \in \dee C_{{\Gamma}'}(L^2)$ and
$ M_+(\zeta) = M_-(\zeta) v'(\zeta) $
for $\zeta \in {\Gamma}'$, where for $1 \leq j \leq n$,
$$
\left. v'\right|_\Gamma= 	
					v(\zeta)	,	\quad	
\left. v'\right|_{\gamma_j}=
					\begin{pmatrix}
						1										&	0	\\
						\dfrac{c_j}{\zeta - \zeta_j}	&	1
					\end{pmatrix}, 
					\quad
\left. v' \right|_{\gamma_{j+n}}=
					\begin{pmatrix}
						1	&	\dfrac{c_{j+n}}{\zeta-\zeta_{j+n}}	\\
						0	&	1
					\end{pmatrix},
$$
\end{RHP}

\begin{theorem}
\label{thm:RHP.model.disc}
Problems \ref{RHP.model.disc} and \ref{RHP.model.aug} are equivalent, i.e., there exists a unique solution to Problem \ref{RHP.model.disc} if and only if there exists a unique solution to Problem \ref{RHP.model.aug}. Their respective solutions $M$ and $M'$ are related by
$$ M(z) = \begin{cases}
					M'(z), 	&	z \notin \medcup_{j=1}^{2N} D_j\\
					\\
					M'(z) \left(I + \dfrac{v_j}{z-\zeta_j}\right), & z \in D_j
				\end{cases}
$$
\end{theorem}

For a proof see \cite[Section 6]{Zhou89}.    			
%
%

\section{The Direct Scattering Map}
\label{sec:direct}

This section is devoted to the study of the  direct map.  We prove Theorem \ref{thm:dense} in section \ref{sec:direct.generic} and Theorem \ref{thm:R} in section \ref{sec:direct.lip.discrete}. We omit the proof of the continuous dependence of $\rho$  on $q$ since the analysis 
 is essentially unchanged from Section 3 of Paper I. Full details can also be found in the thesis of the second author \cite{Liu17}.

In Section  \ref{sec:analytic}, we  establish analyticity of $\alpha(\lam)$ in $\lambda$ with Lipschitz continuity in $q$. We  then prove via perturbation theory that the set $U$ is  open and dense in $H^{2,2}(\R)$.  Theorem \ref{thm:R}  follows using perturbation theory since $\alpha$ has only finitely many simple poles and the Jost functions
are analytic. 

\subsection{Analytic Continuation}
\label{sec:analytic}


\begin{theorem}
\label{thm:alpha.analytic}
The function $\alpha$ extends to an analytic function on $\C^-$ with
\begin{multline*}
\left| \alpha(\lam;q_1) - \alpha(\lam;q_2) \right| \\
\leq \exp
				C
				\left( 
					\norm[H^{2,2}]{q_1}^2 + \norm[H^{2,2}]{q_2}^2
				\right)
			\left( \norm[H^{2,2}]{q_1} + \norm[H^{2,2}]{q_2} \right)
			\left( \norm[H^{2,2}]{q_1 - q_2}\right).
\end{multline*}
where the constants are uniform in $\lam$ with $\imag \lam \leq 0$. Moreover,
\begin{equation} 
\label{a.infty}
\lim_{R \rarr \infty} \sup_{|\lam| \geq R, \imag \lam \leq 0} |\alpha(\lam)-1| = 0
\end{equation}
where the convergence is uniform in $q$ in a bounded subset of $H^{2,2}(\R )$.
\end{theorem}

\begin{remark} Eq. \eqref{a.infty} implies that  zeros of $\alpha$ are contained in a relatively compact subset of $\C^+$. 
Using the formula
\begin{equation}
\label{direct:alpha.breve}
\balpha(\lam) = \overline{\alpha(\lambdabar)},
\end{equation}
we analytically continue $\balpha$ to $\C^+$ with similar properties.
\end{remark}

In order to prove Theorem \ref{thm:alpha.analytic}, we study analyticity of the normalized Jost
solutions
\begin{equation}
\label{Jost.N}
N^\pm(x,\lam) = \Psi^\pm(x,\lam) e^{ix\lam  \sigma_3} 
\end{equation}
which are unique solutions of the  spectral problem
\begin{equation}
\label{direct.n}
\begin{aligned}
\frac{dN^\pm}{dx}	&=	-i\lam [\sigma_3, N^\pm] + Q_\lam N^\pm -\frac{i}{2}\sigma_3 Q^2 N^\pm\\
\lim_{x \rarr \pm \infty} N^\pm(x,\lam)	&=	\twomat{1}{0}{0}{1}
\end{aligned}
\end{equation}
(see \eqref{LS}). 
It follows from \eqref{Jost.T} that
\begin{equation}
\label{direct.n.jc}
N^+(x,\lam) = N^-(x,\lam) 
		e^{-i\lam x \ad(\sigma_3)} T(\lambda).
\end{equation}
From \eqref{alpha.beta.sym} with $\eps=-1$, \eqref{direct.n.jc} at $x=0$, and the fact that $\det N^\pm = 1$, we derive the Wronskian formulae
\begin{align}
\label{alpha}
\alpha(\lam)	&=	N_{11}^+(0,\lam) \overline{N_{11}^-(0,\lam)} + \lam^{-1} N_{21}^+(0,\lam)\overline{N_{21}^-(0,\lam)} ,\\[10pt]
\label{beta}
\beta(\lam)		&=	\frac{1}{\lam}\left(\overline{ N_{11}^-(0,\lam)} \overline{N_{21}^+(0,\lam)} - 
\overline{N_{11}^+(0,\lam)} \overline{N_{21}^-(0,\lam)}\right)
\end{align}
which reduce the study of the analyticity properties of $\alpha$ and $\beta$ to those of the normalized Jost functions $N^\pm$.

 It follows from \eqref{direct.n} that $N_{11}^+(x,\lam)$ and $N_{21}^+(x,\lam)$ satisfy the Volterra integral equations
\begin{align}
\label{direct.n1}
N_{11}^+(x,\lam)	
	&=	1	-	 \int_x^\infty 
								q(y) N_{21}^+(y)  
					\, dy
				-	\int_x^\infty
							\frac{i}{2} |q(y)|^2 N_{11}^+(y,\lam) 
					\, dy \\
\label{direct.n2}
N_{21}^+(x,\lam)
	&=	\int_x^\infty 
				e^{2 i\lam(x-y)}   
						 \lam \overline{q(y)} N_{11}^+(y,\lam) 
				\, dy \\
\nonumber
	&\quad
		+  \frac{i}{2} \int_x^\infty 
				 e^{2i\lambda(x-y)} |q(y)|^2 N_{21}^+(y,\lam)
			 \, dy.			
\end{align}
Integrating by parts to remove the factor  $\lam$ in \eqref{direct.n2} and iterating the resulting equations lead to the following system of integral equations (see Paper I, equation (3.4)):
\begin{align}
\label{direct:n11.lam}
N_{11}^+(x,\lam)	&=	 1 - \frac{1}{2i} \int_x^\infty q(y) \int_y^\infty e^{2i\lam(y-z)}
q^\sharp(z) N_{11}^+(z,\lam) \, dz \, dy \\
\label{direct:n21.lam}
N_{21}^+(x,\lam) &=	\frac{1}{2i}  \overline{q(x)} N_{11}^+(x,\lam) 
	+\frac{1}{2i} \int_x^\infty e^{2i\lam(x-y)} q^\sharp(y) N_{11}^+(y,\lam) \, dy
\end{align}
where 
$$
q^\sharp(x) = \overline{q'(x)} + \frac{i}{2} |q(x)|^2\overline{q(x)} .
$$
\begin{lemma}
\label{lemma:direct.n}
There exist unique solutions of \eqref{direct:n11.lam}--\eqref{direct:n21.lam} with
$$ \sup_{\imag \lam \leq 0} \left|N_{11}^+(x,\lam) \right|
\leq\exp\left(\frac{1}{2}\|q\|_{L^1} \|q^\sharp\|_{L^1}  \right)$$
and
$$ 
\sup_{\imag \lam \leq 0} \left| N_{21}^+(x,\lam) \right| \leq 
\exp\left(\frac{1}{2}\|q\|_{L^1}\|q^\sharp \|_{L^1}\right)\left(\|q\|_{L^1}+\|q^\sharp\|_{L^1} \right)
$$
which are analytic functions of $\lam \in \C^-$ for each $x$. Moreover,
\begin{multline}
\label{direct.n11.c}
\left| N_{11}^+(x,\lam;q_1) - N_{11}^+(x,\lam;q_2) \right|	\\
\leq
\exp\left[
		 	C
				 \left( 	
				 	\|q_1\|_{L^1} \|q_1^\sharp\|_{L^1} + 
				 	\|q_2\|_{L^1} \|q_2^\sharp\|_{L^1}
		 		\right)
		\right] \\
		\times
		\left(  
				\|q_1-q_2\|_{L^1}    \|q_1^\sharp \|_{L^1} 
+ \|q_2\|_{L^1} \|q_1^\sharp-q_2^\sharp\|_{L^1} 
		\right)
\end{multline}
where $C$ is independent of $\lam$ with 
$\imag(\lam) \leq 0$ and $x \geq 0$. 
\end{lemma}

\begin{proof}
Equation \eqref{direct:n11.lam} is of Volterra type  
\begin{equation}\label{volt}
  N_{11}^+ = 1 + K_q N_{11}^+.
 \end{equation}
Using that
$\left| e^{2i\lam(x-y)} \right| \leq 1$ for $\imag \lam \leq 0$ and $x \leq y$,
the estimate
$$ \left| (K_q h)(x) \right| \leq \gamma(x) \sup_{y>x} \left| h(y) \right| $$
holds  with
$$
\gamma(x) = \int_x^\infty |q(y)| \int_y^\infty |q^\sharp(z)| \, dz \, dy.
$$
From  the classical  theory of Volterra integral equations,  the solutions of \eqref{direct:n11.lam}--\eqref{direct:n21.lam} are given by Volterra series that converge uniformly in $\lam$ and $x \geq 0$. Each term in the series defines an analytic function of $\lam \in \C^+$ so that analyticity follows
from uniform convergence. Moreover, the estimate
\begin{equation}
\label{direct.S.res}
\norm[C(\R ^+) \rarr C(\R ^+)]{(I-K_q)^{-1}} 
	\leq \exp \left(\frac{1}{2} \|q\|_{L^1} \|q^\sharp\|_{L^1} \right),
\end{equation}
holds uniformly in $\lam$ with $\imag \lam \leq 0$. Using eq.  \eqref{direct:n21.lam} and the fact that  
$\|q\|_{L^1}$ and $\|q^\sharp\|_{L^1}$ are controlled by 
$\norm[H^{2,2}]{q}$, 
 $\norm[C(\R ^+)]{N_{11}^+}$ and  $\norm[C(\R ^+)]{N_{21}^+}$  
 are bounded uniformly in  $\lam$ with  $\imag \lam \leq 0$
 and $q$ in a bounded subset of $H^{2,2}(\R )$.  Finally, the resolvent estimate \eqref{direct.S.res} and the second resolvent formula imply \eqref{direct.n11.c}.
Similar estimates are obtained for $N_{11}^-$ and $N_{21}^-$ for $x \in \R ^-$.
\end{proof}

\begin{proof}[Proof of Theorem \ref{thm:alpha.analytic}]
From the Wronskian formula \eqref{alpha}  
and the uniform bounds on $N_{21}^-$ and $N_{21}^+$, estimate \eqref{a.infty} will follow from 
\begin{equation}
\label{n11.infty}
\lim_{R \rarr \infty} \sup_{|\lam| \geq R, \imag \lam \leq 0} |N_{11}^\pm(0,\lam)-1| = 0.
\end{equation}
We sketch the proof of  \eqref{n11.infty} for $N_{11}^+$; the proof for $N_{11}^-$ is similar.
From \eqref{direct:n11.lam} and an integration by parts we see that
\begin{align}
\label{n11-1}
N_{11}^+(x,\lam) -1 	
	&=	-\dfrac{1}{2i}\int_x^\infty q(y) \int_y^\infty e^{-2i\lam(z-y)} q^\sharp(z) \, dz \, dy \\
\nonumber
	&\quad -
		\frac{1}{4\lam}\int_x^\infty q(y) \left[ G_1(y,\lam) + G_2(y,\lam)+ G_3(y,\lam)\right] \, dy,
\end{align}
where
\begin{align*}
G_1(x,\lam)	&=	-q^\sharp(x) \left(N_{11}^+(x,\lam)-1 \right)	\\
G_2(x,\lam)	&=	-\int_x^\infty e^{-2i\lam(y-x)} \left(q^\sharp\right)'(y) \left(N_{11}^+(y,\lam)-1 \right)\, dy\\
G_3(x,\lam)	&=	-\int_x^\infty e^{-2i\lam(y-x)} q^\sharp(y) \frac{\dee N_{11}^+}{\dee x}(y,\lam) \, dy .
\end{align*}

Reversing the orders of integration in the first right-hand term of \eqref{n11-1} and integrating by parts
we may estimate
$$ \left| \int_x^\infty q(y) \int_y^\infty e^{-2i\lam(z-y)} q^\sharp(z) \, dz \, dy \right|
\leq \frac{1}{|\lam|} \|q^\sharp\|_{L^1}  \left( \|q\|_{L^\infty} + \|q'\|_{L^1}  \right). $$
From Lemma \ref{lemma:direct.n} we have $|G_1(x,\lam)| \lesssim 1$ where the implied constants
depend only on $\|q\|_{L^1}$ and $\|q^\sharp\|_{L^1}$. Differentiating \eqref{direct:n11.lam} to compute $\dee N_{11}^+/\dee x$ we may similarly estimate $|G_3(x,\lam)|$. To estimate $G_2(x,\lam)$, we need to show that $\norm[L^2(\R ^+)]{N_{11}^+(\dotarg,\lam) - 1}$ is bounded uniformly in $\lam$ with $\imag \lam \leq 0$ and $q$ in a bounded subset of $H^{2,2}(\R )$. This is done in Lemma \ref{lemma:direct.n11-1} below.
\end{proof}

To prove the $L^2$ estimate on $N_{11}^+(x,\lam) -1$, we return to the integral equation 
\eqref{volt} and note that the operator $K_q$ 
is a Hilbert-Schmidt operator on $L^2(\R ^+)$ uniformly in $\lam$ for
$\imag \lam \leq 0$.  
{  Indeed}  its integral kernel is given by
\begin{equation}
\label{S.Ker}
K(x,z) = 	\begin{cases}
						\dint_{\hspace{-1mm}x}^z q(y) e^{-2i\lam(z-y)} q^\sharp(z) \ dy, 	& x < z\\
						\\
						0,			 													& x > z
					\end{cases}
\end{equation}
{  with}
$$ \norm[L^2(\R ^+ \times \R ^+)]{K} \leq \norm[L^{2,1/2}]{q^\sharp} \|q\|_{L^1} .$$
One {  checks  }  that
$$ \ker_{L^2(\R ^+)} (I - K_q) \subset \ker_{C(\R ^+)}(I-K_q)  = \{ 0 \} $$
where the last equality follows from the existence of the resolvent $(I-K_q)^{-1}$ on
$C(\R ^+)$. Writing $K_q = K_q(\lam)$ to display the dependence of the operator $K_q$ on $\lam$, 
we can show that 
\begin{equation}
\label{direct:S.small}
\lim_{|\lam| \rarr \infty} \norm[\mathrm{HS}]{K_q(\lam)} = 0 
\end{equation}
uniformly in $\lam$ with $\imag \lam \leq 0$ and $q$ in a bounded subset of $H^{2,2}(\R )$. This follows from the integration by parts
$$ \int_x^z q(y) e^{-2i\lam(z-y)} \, dy = \frac{1}{2i\lam} \left[ q(z) - q(x)e^{-2i\lam(x-y)} + \int_x^z e^{-2i\lam(z-y)} q'(y) \, dy \right]$$
and a straightforward estimate of the Hilbert-Schmidt norm using \eqref{S.Ker}. We may also estimate
$$ 
\norm[L^2(\R ^+ \times \R ^+)]{K_{q_1}(\lam)-K_{q_2}(\lam)} 
	\leq \norm[L^{2,1/2}]{q^\sharp_1-q^\sharp_2} \|q_1\|_{L^1} + \norm[L^{2,1/2}]{q^\sharp_2}\|q_1-q_2\|_{L^1}
$$
uniformly in $\lam$ with $\imag \lam \leq 0$. On the other hand, it follows from the Dominated Convergence Theorem that $\norm[L^2(\R ^+ \times \R ^+)]{K_q(\lam_1)- K_q(\lam_2)} \rarr 0$ as $\lam_1 \rarr \lam_2$ for any fixed $q \in L^1(\R) \cap L^{2,1/2}(\R)$.   We now use a `continuity-compactness argument'  
as well as  \eqref{direct:S.small} to prove:

\begin{lemma}
\label{lemma:direct.S.res}
The resolvent $(I-K_q(\lam))^{-1}$ exists as a bounded operator on $L^2(\R ^+)$,
and for any $M>0$,
$$\sup_{\imag \lam \leq 0, \, \norm[H^{2,2}]{q} \leq M} 
	\norm[L^2(\R ^+) \rarr L^2(\R ^+)]{(I-K_q(\lam))^{-1}} < \infty.
$$
\end{lemma}

\begin{proof}
For any $M>0$, $R>0$, the identity map takes the set
$$ 
\left\{ \lam \in \C  : \imag \lam \leq 0, |\lam| \leq R \right\} \times 
\left\{ q \in H^{2,2}(\R ): \norm[H^{2,2}]{q} \leq M \right\}
$$
into a subset of $\C   \times (L^{2,1/2} \cap L^1)$ with compact closure.
By the second resolvent formula, the map $(\lam,q) \mapsto (I-K_q(\lam))^{-1}$ is continuous into the bounded operators on $L^2(\R ^+)$. It follows by compactness and continuity that the set
$$ \left\{ (I-K_q(\lam))^{-1}: \imag \lam \leq 0, |\lam| \leq R, \norm[H^{2,2}]{q} \leq M \right\} $$
is compact in $\calB(L^2(\R ^+))$, hence bounded. On the other hand, for sufficiently large $R$ depending on $M$, we have from \eqref{direct:S.small}
that $\sup_{|\lam| \geq R} \norm[\calB(L^2(\R ^+))]{(I-K_q(\lam))^{-1}} \leq 2$
for any $q$ with $\norm[H^{2,2}]{q} \leq M$. 
\end{proof}

\begin{lemma}
\label{lemma:direct.n11-1}
If $q \in H^{2,2}(\R )$, the estimate
$$\norm[L^2(\R ^+)]{N_{11}^+(\dotarg,\lam) - 1} \lesssim 1$$
holds.
\end{lemma}

\begin{proof}
The function $\eta = N_{11}^+ -1$ obeys the integral equation
$ \eta = K_q(1) + K_q(\eta)$ where
$$
K_q(1) = \frac{1}{2i} 
				\int_x^\infty q(y) 
					\int_z^\infty 
						e^{-2i\lam(z-y)}q^\sharp(z) 
					\, dz 
				\, dy.
$$
We may estimate
$$\norm[L^2(\R ^+)]{K_q(1)} \leq  \lw x \rw^{-3/2} \norm[L^{2,2}]{q} \|q^\sharp\|_{L^1}$$
which shows that $K_q(1) \in L^2(\R ^+)$ uniformly in $\lam$ with $\imag \lam \leq 0$. 
The desired bound is obtained
using Lemma \ref{lemma:direct.S.res}.

\end{proof}

\subsection{Generic Properties of Spectral Data}
\label{sec:direct.generic}

In this subsection we prove Theorem \ref{thm:dense}. 
Lee \cite{Lee83} showed that a dense and open subset of potentials $q$ in the Schwartz class have at most finitely many simple zeros of $\alpha$ and no spectral singularities.
His proof is based on a general argument of Beals and Coifman \cite{BC84}. Here we give  a more precise functional analytic argument inspired by analogous results in Schr\"{o}dinger scattering theory (see the manuscript of Dyatlov and Zworski \cite[Chapter 2, Theorem 2.2]{DZ17}). 

{We will now write $\alpha(\lam;q) \equiv \alpha(\lam)$ to emphasize the dependence of $\alpha$ on $q$.}
The set of $C_0^\infty(\R)$ potentials $q$ is dense in $H^{2,2}(\R)$. 
The following fact is well-known and easy to prove; see
for example Chapter 2 of  \cite{Lee83}. 

\begin{lemma}
Suppose that $q \in C_0^\infty(\R )$. 
Then $\alpha(\lam;q)$ is an entire function of $\lambda$ and has at most finitely many zeros in $\imag \lam \leq c$ for any $c \in \R $. 
\end{lemma}

Using this lemma, a perturbation argument and Rouch\'e's theorem, we will construct a dense set of potentials in $H^{2,2}(\R )$ for which $\alpha$ has at most finitely many \emph{simple} zeros in $\C ^-$ and no zeros on $\R $. We will then exploit Theorem \ref{thm:alpha.analytic}  to show that this set is also open.  These steps are carried out in Propositions \ref{prop:direct.generic} and \ref{prop:direct.open} below which together give the proof of Theorem \ref{thm:dense}.

To construct the dense set, we need two perturbation lemmas for the transition matrix $T(\lambda)$ defined in  \eqref{Jost.T}. The first concerns perturbation from the zero potential for which $\alpha(\lam) \equiv 1$ and $\beta(\lam) \equiv 0$. 

\begin{lemma}
Suppose that $\varphi \in C_0^\infty([-R,R])$, $\lam \neq 0$, and $\mu$ is a small parameter. Let $q=\mu \varphi$.
Then the associated  transition matrix has the form
\begin{equation*}
{T(\lam,q)} = \twomat{1}{0}{0}{1} + \twomat{0}{\mu c_\varphi}{ {- \lam \overline{\mu} \overline{c_\varphi}}}{0} + \bigO{\mu^2}
\end{equation*}
where
$c_\varphi = -  \int_{-\infty}^\infty e^{2i\lam y} \varphi(y) \, dy$
and the correction term depends on $\norm[H^{1,1}]{\varphi}$.
\end{lemma}

\begin{proof}
It suffices to show that 
\begin{align}
\label{direct.a.mu}
\alpha(\lam)	
	&\sim	1 + \bigO{\mu^2}, \\
\label{direct.b.mu}
\lam  \overline {\beta(\lam)} \
	&= \lam \overline{\mu} \int e^{-2i\lam y} \overline{\varphi(y)} \, dy
	+ \bigO{\mu^2}.
\end{align}
We know from Section 3.2 of Paper I that, for $\lam \in \R $, 
the uniform estimate
$$ \left| \left(N_{11}^+(x,\lam), N_{21}^+(x,\lam) \right) \right| \lesssim 1 $$
holds with the implied constants depending only on $\norm[H^{2,2}]{q}$ (the key issue is that the large-$\lambda$ behavior is controlled despite the $\lam$-dependence of the perturbing term in \eqref{direct.n}; see equations (3.4) of Paper I  for the integration by parts that removes this term). 
Taking limits as $x \rarr -\infty$ in  equations \eqref{direct.n1}--\eqref{direct.n2} 
 {   for $N^+$ (and as $x\to -\infty $ in the corresponding  equations for $N^-$)} 
and using the relation \eqref{direct.n.jc}, we 
 deduce that
\eqref{direct.a.mu} and \eqref{direct.b.mu} hold.
\end{proof}

The next lemma  gives a mechanism for splitting multiple poles and perturbing zeros on the real axis.

\begin{lemma}
\label{lemma:TT}
Suppose that $q_1$ and $q_2$ are compactly supported potentials with disjoint supports, and that the support of $q_2$ on the real line lies to the left of the support of $q_1$. Then:
\begin{enumerate}
\item[(i)] The identity
$$T(\lam,q_1+q_2) = T(\lam,q_2) T(\lam,q_1) $$
holds.
\item[(ii)]
If ${  q_1 } \in C_0^\infty((-R,R))$ and  $q_2 = \mu \varphi$ with $\varphi \in C_0^\infty((-2R,-R))$, the formula
\begin{equation}
\label{TT}
T( {  \lam }, q_1+\mu \varphi)	
		=	 \twomat{1}{\mu c_\varphi }{-\lam \overline{\mu c_\varphi} }{1} T(\lam,q_1) + \bigO{\mu^2}
\end{equation}
holds.

\end{enumerate}
\end{lemma}

\begin{proof}
Consider the normalized solution $N^+(x,\lam,q)$. It is not difficult to see that
$$
N^+(x,\lam,q_1+q_2) = N^+(x,\lam,q_2) N^+(x,\lam,q_1). 
$$

We may now compute
$$
T(\lam, q_1+q_2) 	= \lim_{x \rarr -\infty} e^{i \lam x \ad(\sigma_3)}
								\left[N^+(x,\lam,q_2) N^+(x,\lam,q_1)\right]\\
						=	T(\lam,q_2) T(\lam,q_1)
$$
The second assertion is an immediate consequence of the first.
\end{proof}

Suppose that $q_1$ and $q_2$ are chosen as in Lemma \ref{lemma:TT}(ii). To simplify the notation, let us write  $\alpha(\lambda,\mu) $
to denote $\alpha(\lambda, q_1+\mu\varphi)$.   It follows from \eqref{TT} that
\begin{equation*}
\alpha(\lam,\mu)  = \alpha(\lam,0)  +  { \mu c_{\varphi} } \lam \overline{\beta(\lam)}  
	+ \bigO{\mu^2}.
\end{equation*}
 In the next proposition, we   expand the above formula near $\lambda= \lambda_0$:
\begin{equation}
\label{balpha.p}
\alpha(\lam,\mu)  = \alpha(\lam,0) +   { \mu c_{\varphi_0} } \lam_0 \overline{\beta(\lam_0)}  +C_0(\lambda-\lambda_0)\mu
	+ \bigO{\mu^2}.
\end{equation}
where  
$$c_{\varphi_0}=-  \int_{-\infty}^\infty e^{2i\lam_0 y} \varphi(y) \, dy, \quad {\rm and} \quad 
\left|C_0\right|
\leq 
\norm[L^\infty]{\dfrac{d}{d\lambda}\left(\lam \overline{\beta(\lam)}\right)}.
$$ 
From the compactly  supported  potential $q$ and the asymptotic behavior of  $\alpha(\lambda)$ given in \eqref{a.infty},  we know that $\alpha$ has finitely many zeros in $\mathbb{C}^- \cup \mathbb{R}$. Let   $D(\lam_i,r_i)$ be non-overlapping discs  centered at the finitely many zeros $\lam_i$ as shown in Figure \ref{figure-1}.

\begin{proposition}
\label{prop:direct.generic}
Suppose that $R>0$ and $q \in C_0^\infty([-R,R])$. Let $\alpha(\lam)$ be the 
$(1,1)$ entry of the transition matrix for $q$. For $\varphi \in C_0^\infty(\R )$, let $\alpha(\lam,\mu)$ be the $(1,1)$ entry for the transition matrix of $q+\mu \varphi$. In particular, $\alpha(\lam,0) = \alpha(\lam)$. 
\begin{itemize}
\item[(i)]		Suppose that $\lbrace \lambda_i \rbrace_{i=1}^m$ are the isolated zeros of $\alpha(\lam)$ in $\mathbb{C}^- \cup \mathbb{R}$ and $\lam_i \neq 0$ is one of the zeros of $\alpha(\lam)$ of multiplicity $n \geq 2$, i.e. 
$\alpha(\lambda)=(\lambda-\lambda_i)^n g(\lambda)$
 for some analytic function $g$ with $g(\lam_i)\neq 0$.    Then, for some $\varphi \in C_0^\infty(\R )$ and all sufficiently small $\mu \neq 0$, $\alpha(\lam,\mu)$ has $n$ simple zeros in the disc $D(\lambda_i,r_i)$.

\item[(ii)]  Suppose that after the perturbation described in part (i),  $\Lambda_j$ is a simple zero of $\alpha(\lam,\mu)$ on the real axis and $\Lambda_j \neq 0$. Then, there is a function $\varphi \in C_0^\infty(\R )$ such that, for all real,  nonzero and sufficiently small $\mu'$, $\alpha(\lam,\mu')$ has no zeros on the real axis near $\Lambda_j$. 
\end{itemize}
In both cases, we may choose $\varphi$ to have support in $(-2R,-R)\cup (R, 2R)$.
\end{proposition}

\begin{proof}
(i)  We first claim that there exists a function $\varphi\in C^\infty_0(\mathbb{R})$, $\varphi\geq 0$ such that 
$$ \widehat{\varphi}(\lam_i) = \int_{-\infty}^\infty { e^{ 2 i\lam_i x} } \varphi(x) \, dx \neq 0$$ for all $i$, $1 \leq i \leq m$.
Indeed, let $2\lambda_i=\xi_i+i\eta_i$, 
$$\widehat{\varphi}(\lam_i) = \int_{-\infty}^\infty (\cos \xi_i x + i\sin\xi_i x) e^{\eta_i x}\varphi(x) \, dx$$
 with $ e^{\eta_i x}\varphi(x)\geq 0$ for all $x$. \\
 
Let $\xi=\text{max} \lbrace \xi_1,...\xi_i,...\xi_m\rbrace$,  $r=\pi/2\xi$ and assume $|\text{supp}(\varphi)|\leq r$, then for all $i$, at least one of $\cos \xi_i x$ and $\sin \xi_i x$ does not change sign on  the support of $\varphi$. Thus
$ \widehat{\varphi}(\lam_i) {  \ne 0} $ for all $i$.
Using the Taylor expansion of $\widehat{\varphi}(\lambda)$ and $\lambda {  \overline{\beta(\lam)}}$ we can write  (\ref{balpha.p}) as
\begin{equation*}
\alpha(\lam,\mu)  = (\lambda-\lambda_i)^n g(\lambda) {  +}  { \mu c_{\varphi_i} } \lam_i 
\overline{\beta(\lam_i)}
 +C_0(\lambda-\lambda_i)\mu
	+ \bigO{\mu^2}.
\end{equation*}
If we can establish the following inequalities
 \begin{equation}
 \label{bound-1}
 \left| \lambda   \overline{\beta(\lam)}\right|_{L^\infty}\lesssim_q 1
 \end{equation}
  and
\begin{equation}
\label{bound-2}
\left| (\lambda \overline{\beta(\lam)})'\right|_{L^\infty}\lesssim_q 1
\end{equation}
for  $\lambda\in\overline{ D(\lambda_i, r_i) }$,
 then 
\begin{align*}
|\alpha(\lam,\mu) -  (\lambda-\lambda_i)^n g(\lambda)  |&=|  { \mu c_{\varphi_i} } \lam_i  \overline{\beta(\lam_i)}
 +C_0(\lambda-\lambda_i)\mu+ \bigO{\mu^2}  |
                                                        \leq |\lambda-\lambda_i|^n |g(\lambda)|  
\end{align*}
for $\mu$ sufficiently small and $\lambda\in \partial D(\lambda_i, r_i)$. Rouch\'{e}'s Theorem shows that the number of zeros of $\alpha(\lambda, \mu)$ and $\alpha(\lambda, 0)$ agree (with multiplicities) in $D(\lambda_i, r_i )$. That is, $\alpha(\lambda, \mu)$ has exactly $ n$ zeros there.

To prove  estimates \eqref{bound-1} and \eqref{bound-2}, we use  that
$$\lam \overline{\beta(\lam)} = \lim_{x \rarr \infty} e^{-2i\lam x} N_{21}^-(x,\lam)$$ and the analogue of \eqref{direct.n2} for $N_{21}^-$ to compute
\begin{align*}
\lambda   \overline{\beta(\lam)}&= \int_{-R}^R e^{-2i\lambda y}\left(-\lambda 
		\overline{q(y)}N^-_{11}(y,\lambda)  +p_2(y)N^-_{21}(y,\lambda) \right)dy.
\end{align*}
From direct computation, its derivative is
\begin{align*}
\dfrac{d}{d\lambda}\left(\lambda{  \overline{\beta(\lam)}}\right)
	&= 
	\int_{-R}^R  
			e^{-2i\lambda y} (-2iy) 
				\left(
					-  \lambda 
					\overline{q(y)}N^-_{11}(y,\lambda)  +
					p_2(y)N^-_{21}(y,\lambda) 
				\right)	\, dy\\
	&\quad + \int_{-R}^R 
			e^{-2i\lambda y}
					\left( -
						\overline{q(y)}N^-_{11}(y,\lambda)  +
					p_2(y)N^-_{21}(y,\lambda) _\lambda 
				\right)	\, 	dy\\ 
	&\quad -
				\int_{-R}^R 
			e^{-2i\lambda y}
					\left(
						-\lambda 
							\overline{q(y)}N^-_{11}(y,\lambda)_\lambda  +
						p_2(y)N^-_{21}(y,\lambda)_\lambda 
					\right) \,  dy.                                                                                       
\end{align*}
Inequalities (\ref{bound-1}) and (\ref{bound-2}) follow from these expressions and Lemma \ref{lemma:direct.n}. \\

Now we want to show that the zeros  of $\alpha(\lambda, \mu)$ are simple. For $0\leq k\leq n-1$, consider the disc  around the $k^{th}$ root of unity of $\gamma_i$
(see figure \ref{figure-1})
\begin{equation*}
D_k:=D\left(| \gamma_i |^{\frac{1}{n}} e^{i (\phi+2\pi k)/n }+\lam_i ,\,  \rho |\gamma_i|^{\frac{1}{n}} \right)
\end{equation*}
where
$$ \gamma_i=\frac{{ \mu c_{\varphi_i} } \lam_i 
 \overline{\beta(\lam_i)}
 }{g(\lam_i)} , \,\,\, \phi=\arg \gamma_i+\pi.$$
Notice that if 
$\rho< \pi/n$, then $D_k \cap D_\ell $ is empty for $k\neq \ell.$ 
We now expand $g(\lambda)$ at $\lambda=\lambda_i$ and get
$$\alpha(\lambda, \mu)=(\lambda-\lambda_i)^n g(\lambda_i)+\mathcal{O}(\lambda-\lam_i)^{n+1}.$$
For $\lambda\in \partial D_k$, 
$$\left\vert(\lambda-\lambda_i)^n g(\lambda_i)  + { \mu c_{\varphi_i } }  \lam_i 
 \overline{\beta(\lam_i)}
-\alpha(\lambda, \mu) \right\vert   
{  \lesssim} C_0 |\gamma_i |^{1+\frac{1}{n}}.$$
On the other hand, if we choose $\rho > 2 C_0 |\gamma_i|^{\frac{1}{n}} $ then for $\lambda\in  \partial D_k$, 
\begin{align*}
\left\vert(\lambda-\lambda_i)^n g(\lambda_i)-\gamma_i g(\lambda_i)\right\vert
&=|\gamma_i | \rho\left(1+\bigO {\rho^2} \right)\\
&
{   \gtrsim} C_0 |\gamma_i|^{1+\frac{1}{n}}
\geq | (\lambda-\lambda_i)^n g(\lambda_i)  + { \mu c_{\varphi_i} }   \lam_i  \overline{\beta(\lam_i)}
-\alpha(\lambda, \mu) |.
\end{align*}
Since the discs $D_k$ are disjoint, Rouch\'{e}'s theorem now shows that there is exactly one zero of $\alpha(\lambda, \mu)$ in each $D_k$. Consequently,  all $n$ zeros are simple.

(ii)  After the  perturbation step (i),  $\alpha(\lambda, \mu)$ has simple zeros $\lbrace \widetilde \lambda_j \rbrace_{j=1}^l$. Suppose that one of them, $\widetilde \lambda_j$, is 
located  on the real axis. We perform another small perturbation of the potential and write
\begin{align*}
\alpha(\lambda, \mu')= 
	(\lambda-\widetilde \lambda_j) h(\lambda, \mu) + 
	 \mu' c_{\psi_j } \widetilde \lambda_j  \overline{\beta(\widetilde \lambda_j, \mu)}  +
	 C'_0(\lambda-\widetilde \lambda_j)\mu' +
	\bigO{\mu'^2}
\end{align*}
where $$(\lambda-\widetilde \lambda_j)h(\lambda, \mu)=\alpha(\lambda, \mu) . $$
Let 
\begin{equation*}
D_j:=D\left( \widetilde\lambda_i +\widetilde \delta_j ,\,  \rho'|\widetilde \delta_j| \right)
\quad {\rm with } \quad 
\widetilde \delta_j=
	\frac	{  \mu' c_{\psi_j} \widetilde \lambda_j \overline{\beta(\widetilde \lambda_j, \mu)}}
			{h(\widetilde \lambda_j, \mu)} .
\end{equation*}
Given $\widetilde \lambda_j\in \mathbb{R}$, we can make appropriate choices of small parameter $\mu'$ and $\psi\in C^\infty_0(\mathbb{R})$ such that $\imag( \widetilde \lambda_j +\widetilde \delta_j)$ is strictly nonzero and $D_j\cap \mathbb{R}$ is empty. Since there are only finitely many zeros,  we can choose $\mu$  that satisfies this property for all $j=1,2,..., l$. Repeating the argument in (i), we get the desired conclusion.
\end{proof}

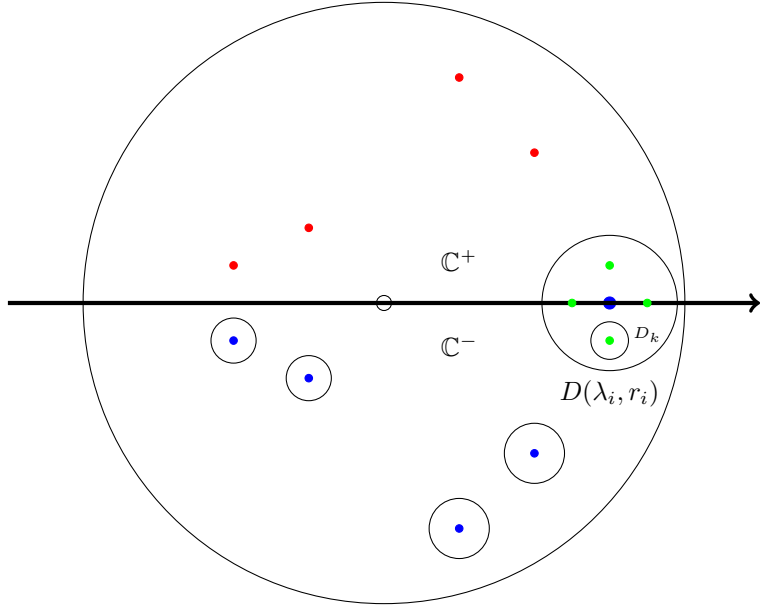
\begin{figure}
\begin{tikzpicture}
\draw [red, fill=red] (-2,0.5) circle [radius=0.05];
\draw [red, fill=red] (-1,1) circle [radius=0.05];
\draw [red, fill=red] (1,3) circle [radius=0.05];
\draw [red, fill=red] (2,2) circle [radius=0.05];
\draw [blue, fill=blue] (3, 0) circle [radius=0.08];
\draw  (0,0) circle [radius=0.1];
\draw (0,0) circle [radius=4];
\node [above] at (1,0.3) {$\mathbb{C^+}$};
\draw [->][ultra thick] (-5,0) -- (5,0);
\node [below] at (1,-0.3) {$\mathbb{C^-}$};
\draw [blue, fill=blue] (-2,-0.5) circle [radius=0.05];
\draw [blue, fill=blue] (-1,-1) circle [radius=0.05];
\draw [blue, fill=blue] (1,-3) circle [radius=0.05];
\draw [blue, fill=blue] (2,-2) circle [radius=0.05];
\draw [green, fill=green] (3,-0.5) circle [radius=0.05];
\draw [green, fill=green] (3,0.5) circle [radius=0.05];
\draw [green, fill=green] (3.5,0) circle [radius=0.05];
\draw [green, fill=green] (2.5,0) circle [radius=0.05];
\node [below] at (3.5, -0.2){{\tiny $D_k$}};
\draw (3, -0.5) circle [radius=0.25];
\draw (3, 0) circle [radius=0.9];
\node [below] at (3, -0.9) { $D(\lam_i, r_i)$};
\draw (-2, -0.5) circle [radius=0.3];
\draw (-1, -1) circle [radius=0.3];
\draw (1, -3) circle [radius=0.4];
\draw (2, -2) circle [radius=0.4];
\end{tikzpicture}
\caption{Zeros of $\balpha$ and $\alpha$ in the $\lambda$ plane}
\label{figure-1}
\end{figure}

\begin{figure}
\begin{tikzpicture}
\draw [red, fill=red] (-2,0.5) circle [radius=0.05];
\draw [blue, fill=blue] (2.5, 0) circle [radius=0.08];
\draw  (0,0) circle [radius=0.1];
\node [above] at (0,0.3) {$\mathbb{C^+}$};
\draw [->][ultra thick] (-5,0) -- (5,0);
\node [below] at (0,-0.3) {$\mathbb{C^-}$};
\draw [blue, fill=blue] (-2,-0.5) circle [radius=0.05];
\node [above] at (2.5, 0) {$\widetilde \lambda_j$};
\draw (2.5, 0) circle [radius=1.2];
\draw [blue, fill=blue] (3.0, -0.5) circle [radius=0.05];
\node[left] at (2.8,- 0.6) {{\tiny $D_j$}};
\draw (3.0,- 0.5) circle [radius=0.3];
\node [below] at (3, -1.3) {$D_{\widetilde \lambda_j}$};
\end{tikzpicture}
\caption{Simple zero of $\alpha(\lambda, \mu)$ on $\mathbb{R}$}
\end{figure}
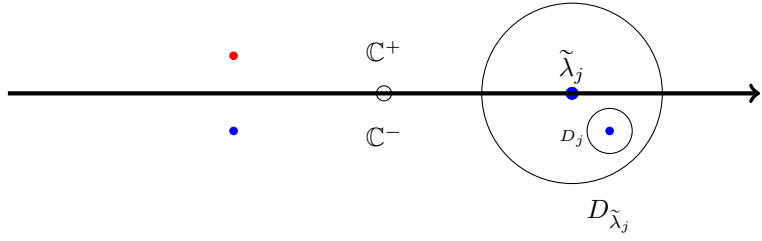

Proposition \ref{prop:direct.generic} shows that there is a dense subset of $q \in H^{2,2}(\R )$ for which $\alpha(\lam;q)$ has at most finitely many simple zeros 
in $\C  ^-$ and no zeros on $\R $. 
Due to the symmetry relation between $\alpha$ and $\balpha$, on has a similar result for $\balpha$.
The continuity of $\alpha$ in $q$,
the fact that   $\alpha$  is analytic in $\C  ^-$ and the continuity of the map $q \mapsto \alpha(\dotarg,q)$ imply that this set is also open.

\begin{proposition}
\label{prop:direct.open}
Suppose  $q_0 \in H^{2,2}(\R )$ and  $\alpha(\dotarg;q_0)$ has exactly $n$ simple zeros in $\C  ^-$ and no zeros on $\R $. There is a neighborhood $\calN$ of $q_0$ in $H^{2,2}(\R )$ so that all $q \in \calN$ have the same properties.
\end{proposition}

\begin{proof}
Since $|\alpha(\lam;q_0)|$ does not vanish on $\R $, we have $|\alpha(\lam;q_0)| \geq c$ for some $c>0$. It follows from Lipschitz continuity of $q \mapsto \alpha(\dotarg;q)-1$ as maps $H^{2,2}(\R ) \rarr H^1(\R)$  that there is an $r_0 > 0$ so that 
$|\alpha(\lam;q)| \geq c/2$ for all $q$ with $\norm[H^{2,2}]{q-q_0} < r_0$. Next, 
let $\eta_1 = \inf_{j \neq k} |\lam_j- \lam_k|$, $\eta_2 = \inf_{k} |\imag \lam_k|$, and
$\eta = \frac{1}{2} \inf(\eta_1,\eta_2)$. By Theorem \ref{thm:alpha.analytic}  there is an $r_1>0$ so that the $n$ simple zeros of $\alpha$ remain simple and move a distance no more than $\eta$ for $q \in H^{2,2}(\R )$ with 
$\norm[H^{2,2}]{q-q_0} < r_1$. Take $\calN = B(q_0,r)$ where $r < \inf(r_1,r_2)$.
\end{proof}

\subsection{Lipschitz Continuity of Spectral Data for Generic Potentials}
\label{sec:direct.lip.discrete}

We complete the proof of Theorem \ref{thm:R} by showing the Lipschitz continuity of the eigenvalues $\lam_k$ and the norming constants $C_k$ on open subsets of $U_N$ with bounds uniform in bounded subsets of $U_N$. 
This is a result  in eigenvalue perturbation theory and is done in Proposition \ref{prop:direct.discrete.lip}.

We order the zeros of $\balpha$ by the size of their modulus and, given two zeros with the same modulus, we  order them  by increasing phase in $(0,\pi)$. The norming constants $C_k$ are defined in terms of the normalized Jost solutions $N^\pm$ of \eqref{direct.n} as follows. If $\balpha(\lam_k) = 0$, there is a constant $B_k$ with the property
\begin{equation}
\label{direct:bk.bis}
\twovec{N_{11}^-(x,\lam_k) }{N_{21}^-(x,\lam_k)  }
= B_k \lam_k e^{2i\lam_k x}\twovec{ N_{12}^+(x,\lam_k)}{N_{22}^+(x,\lam_k) }
\end{equation} 
Since $\balpha'(\lambda_k) \ne 0$,  $C_k : = {B_k}/{( \balpha'(\lambda_k))}$ where  $\lambda_k \in \C^{+}$. The discrete scattering
data are composed of the pairs $(\lambda_k, C_k)$.
\begin{proposition}
\label{prop:direct.discrete.lip}
Suppose that $q_0 \in U$ with $N$ simple zeros of $\balpha$ in $\C  ^+$. Let $\Lam = \{ \lam_1, \ldots, {  \lambda_N\}}$ be the zeros of 
$\balpha$ with the ordering as described above, and recall the definition \eqref{distances} of $d_\Lam$. 
There is a neighborhood $\calN$ of $q_0$ so that:
\begin{itemize}
\item[(i)]		For any $q \in \calN$, $\balpha(\lam;q)$ has exactly $n$ simple zeros in $\C  ^+$,  no zeros on $\R $, and $|\lam_j(q) - \lam_j(q_0)| \leq 
\frac{1}{4} d_\Lam$. 
\item[(ii)]	The estimate $|\lam_j(q) - \lam_j(q_0)| \leq C \norm[H^{2,2}]{q-q_0}$
holds for $C$ uniform in $q \in \calN$. 
\item[(iii)]	The estimate $|B_j(q) - B_j(q_0)| \leq C \norm[H^{2,2}]{q-q_0}$ holds for $C$
uniform in $q \in \calN$.
{ 
\item[(iv)]	The estimate $|C_j(q) - C_j(q_0)| \leq C \norm[H^{2,2}]{q-q_0}$ holds for $C$
uniform in $q \in \calN$.
}
\end{itemize}
\end{proposition}

\begin{proof}
(i) From Proposition \ref{prop:direct.open} and \eqref{direct:alpha.breve}, we immediately conclude that there is a neighborhood $\calN$ of $q_0$ for which if $q \in \calN$,
$\balpha(\lambda;q)$ has exactly $N$ simple zeros in $\C ^+$ with no singularities on the real axis. We can establish continuity of the simple zeros as a function of $q$ (and hence  estimate $|\lam_j(q)-\lam_j(q_0)| \leq \frac{1}{4} d_\Lambda$) as 
follows. The equation  $\balpha(\lam_j(q); q) = 0$, defines $\lam_j$ as an analytic function of $q$ in a neighborhood of $q$ owing to the fact that $\balpha'(\lam_j(q_0), q_0) \neq 0$ and the implicit function theorem on the Banach space $\C \times H^{2,2}(\R)$.
The function $\balpha(\lam; q)$ is  analytic  in $q$ because the functions appearing in the Wronskian formula \eqref{alpha} may be computed by convergent Volterra series which are analytic in $q$. 

(ii) The implicit function theorem also guarantees that the function $\lam_j(q)$ will be $C^1$ as a function of $q$, and hence Lipschitz continuous.

(iii) Uniqueness for the equation \eqref{direct.n} guarantees that at least one of $N_{12}^+(0,\lam_j)$ and $N_{22}^+(0,\lam_j)$ is nonzero at $q=q_0$. Suppose that $N_{12}^+(0,\lam_j(q_0)) \neq 0$. By shrinking the neighborhood if needed, we may assume that 
$N_{12}^+(0,\lam_j(q)) >0$ strictly for all $q \in \calN$. We then obtain  from \eqref{direct:bk.bis} that $B_j(q) = \lam_j(q) N_{21}^-(0,\lam_j(q))/N_{22}^+(0,\lam_j(q))$ which, as a product and quotient of Lipschitz continuous functions of $q$, is itself Lipschitz continuous in $q$.

(iv) Finally, $\balpha'(\lambda_k)$ can easily be expressed in terms of $\balpha$ through a Cauchy integral over a small circle around $\lambda_k$ due to the analyticity of 
$\balpha$ in $\C^+$. The Lipschitz continuity 
of $B_j$ and  $\balpha(\lambda_j)$  in $q$ extends to the norming constants $C_j$.
\end{proof}
%
%

\section{The Inverse Scattering Map}
\label{sec:inverse}

This section is devoted to the analysis of  the inverse scattering map. 
The following Riemann-Hilbert problem is equivalent to Problem \ref{RHP2} and is technically
simpler. 

\begin{RHP}
\label{RHP2.row}
Given $x,t \in \R$, $\rho\in { P}$, and $\{ (\lam_j, C_j)\}_{j=1}^N$ in $(\C^{+} \times \C^\times)^N$, find a row vector-valued function $n(x,t,z): \C \setminus (\R \cup\Lam)  \rarr \C^2$ with the following properties:
\begin{itemize}
\item[(i)]		$n(\lam;x,t) = \begin{pmatrix} 1 & 0 \end{pmatrix} + \bigO{\dfrac{1}{\lam}}$ as $|z| \rarr \infty$,
\item[(ii)]  $n$ has continuous boundary values $n_\pm$ on $\R$ and
				$$ n_+(\lam;x,t) = n_-(\lam;x,t) e^{it\theta \ad (\sigma_3)} V(\lam),
				\quad 
				V(\lam) = \begin{pmatrix}
								1-\eps\lam |\rho(\lam)|^2 	&	 \rho(\lam)	\\
								-\eps \lam \overline{\rho(\lam)}	&	1
							\end{pmatrix},
				$$
\item[(iii)]	For each $\lam \in \Lam$, 
				$$ \Res_{z = \lam} n(\lam;x,t) = \lim_{z \rarr \lam} n(z;x,t) e^{it\theta \ad (\sigma_3)} V(\lam) $$
				where for each $\lam \in \Lam^+$
				$$ V(\lam)  = \lowmat{\lam C_\lam}, \quad V( \overline{\lam}) = \upmat{\eps \overline{C_\lam}}. $$
\end{itemize} 
\end{RHP}

\begin{remark}
\label{rem:RHP2c}
As discussed in Section \ref{sec:prelim}, we can replace the residue conditions
in (iii) by introducing new circular contours $\Gamma_j$ and $\Gamma_j^*$ about $\lambda_j$ and $\overline{\lambda_j}$ for each $j$. We denote the augmented contour
by 
\begin{equation}
\label{R'} 
\R' = \R \cup \left( \cup_{j=1}^N \Gamma_j \right) \cup \left( \cup_{j=1}^N \Gamma_j^*\right). 
\end{equation}
Along the added contours we get
new jump matrices
$$ 
\left. V(\lam) \right|_{\Gamma_j} = 
\lowmat{\frac{\lam_j C_j}{\lam-\lam_j}}, \quad 
\left. V(\lam) \right|_{\Gamma_j^*} = 
\upmat{\frac{-\eps \overline{C_j}}{\lam-\overline{\lam_j}}}.
$$
\end{remark}

We will study the inverse scattering map defined by Problem \ref{RHP2.row} and the reconstruction formula \eqref{q.lam}. We will first prove that Problem \ref{RHP2.row} is uniquely
solvable, and then reduce by symmetry to a scalar Fredholm type integral equation and an integral reconstruction formula to show that the map from scattering data to $q$ is Lipschitz continuous.

\subsection{Existence and Uniqueness}
\label{sec:EU}

For convenience we take $t=0$ so that the jump matrix
in Problem \ref{RHP2.row} 
is $e^{-ix\lam \ad (\sigma_3)} V$ rather than $e^{-it\theta \ad (\sigma_3)} V$.  We will consider
$\rho$ in the  space\footnote{{This choice is natural  but in our Lemma \ref{lemma:null}, we need to check that Zhou's vanishing theorem applies to Problem \ref{RHPX} since the jump matrix in Problem \ref{RHPX} does not belong to the space $H^1(\Sigma)$ usually considered
(see, for example, \cite[Definition 2.54 and Theorem 2.73 ]{TO16}).
 One can avoid this problem by using the somewhat less natural space 
$$ 
\widetilde{Y} = \{ \rho \in L^2(\R): (1+|\lam|)^{3/4} \rho'(\lam) \in L^2(\R)\}
$$
into which $H^{2,2}(\R)$  is compactly embedded  and yields $r \in H^1(\R)$ for $r$ as defined in \eqref{rho->r}.}}
\begin{equation}
\label{Y}
Y = 
\left\{ 
	\rho \in L^2(\R): \widehat{\rho}, \, \widehat{\rho}\,' \in L^1(\R)
\right\},
\end{equation}
into which $H^{2,2}(\R)$ is compactly embedded, and we will  impose the constraint
\begin{equation}
\label{Y.con}
1- \eps\lam |\rho(\lam)|^2 > 0.
\end{equation} 
We will need the unique solvability of Problem \ref{RHP2.row} for $\rho$ in this larger space in order to carry out a uniform resolvent estimate in the next subsection (see Lemmas \ref{lemma:Kx.small} and \ref{lemma:res}). Note that any $\rho \in Y$ satisfies $\rho$, $\lam \rho(\lam) \in L^\infty(\R)$.

\begin{theorem}
\label{thm:unique}
There exists a unique solution to Problem \ref{RHP2.row} for any $\rho \in Y$ and $\{(\lam_j, C_j) \} \in \left( \C^{+} \times \C^\times\right)^N$.
\end{theorem}

We will give a complete proof of Theorem \ref{thm:unique} for the case $\eps=-1$. At the end of this subsection, we 
give the key transformation needed to extend the proof to the $\eps=+1$ case.

By Theorem \ref{thm:RHP.model.disc}, Problem \ref{RHP2.row} is uniquely solvable if and only if the modified RHP with augmented contour $\R' = \R \cup \left( \cup_j \Gamma_j \right) \cup \left( \cup_j \Gamma_j^*\right)$ and new jump matrices as described in Remark \ref{rem:RHP2c} is uniquely solvable. Make the factorization $V=(I-W^-)^{-1}(I+W^+)$
where
\begin{equation}
\label{RHP2.W}
\left. (W^+,W^-) \right|_\R = 
\left( 
	\begin{pmatrix}
	0	&	0	\\	-\eps \lam \overline{\rho} 	& 0
	\end{pmatrix},
	\begin{pmatrix}
	0	& 	\rho	\\	0 & 0
	\end{pmatrix}
\right), \quad
\left. W^-\right|_{\Gamma_j} = 0, \quad
\left. W^+\right|_{\Gamma_j*} =0,
\end{equation}
and let $\calC_W$ be the associated Beals-Coifman integral operator. We wish to show that
$(I-\calC_W)$ is invertible on $L^2(\R')$. It follows, for example, from Theorem 5.10 of \cite{Lenells14} that $I-\calC_W$ is a Fredholm operator on $L^2(\R')$: hence, to show that $(I-\calC_W)^{-1}$ exists, is is enough to show that $\ker_{L^2}(I-\calC_W)$ is trivial.  We will show that $\ker_{L^2}(I-\calC_W)=\{ 0 \}$ assuming 
that $\rho$ belongs to the space \eqref{Y} and the constraint \eqref{Y.con} holds.

To prove that $\ker_{L^2}(I-\calC_W)$ is trivial,  we will show that any $\nu^{(0)} \in \ker_{L^2}(I-\calC_W)$ induces a vector $\mu^{(0)} \in \ker_{L^2}(I-\calC_w)$, where $\calC_w$ is the Beals-Coifman operator associated to a uniquely solvable Riemann-Hilbert problem, Problem \ref{RHPX} below, related to 
Problem \ref{RHP2} by a quadratic change of variables. 
The following observation will play a crucial role: $\left. \nu^{(0)} \right|_{\Gamma_i}$ (resp. $\left. \nu^{(0)}\right|_{\Gamma_i^*}$) is the boundary value of a function analytic in the interior of $\Gamma_i$
(resp.\ the interior of $\Gamma_i^*$). This fact can easily be deduced from 
the homogeneous versions of \eqref{RHP2c.11.R}--\eqref{RHP2c.11.Lam*} in Appendix \ref{app:BC.RHP2}. Hence, $\nu^{(0)}$ has an analytic continuation to the small discs bounded by $\Gamma_i$ and $\Gamma_i^*$.

To describe the new RHP, let $\rho \in Y$ obey \eqref{Y.con} and 
$\{ (\lam_j, C_j)\}_{j=1}^N \in \left(\C^+ \times \C^\times\right)^N$ be given. 
Let $\Sigma$ be the oriented contour shown on the left-hand side of Figure \ref{fig:Sigma}.

\begin{figure}
\caption{The Contours $\Sigma$ and $\Sigma'$}

\medskip

\begin{minipage}{0.45\textwidth}
\begin{tikzpicture}[scale=0.7]
\draw[thick,->-] 	(0,4)	--	(0,0);
\draw[thick,->-]	(0,-4)	--	(0,0);
\draw[thick,->-]	(0,0)	--	(-4,0);
\draw[thick,->-]	(0,0)	--	(4,0);
\draw[black,fill=black]	(2,2)		circle(0.075cm);
\node[right] at (2,2)		{\footnotesize{$\zeta_j$}};
\draw[black,fill=black]	(-2,-2)	circle(0.075cm);
\node[left] at (-2,-2)		{\footnotesize{$-\zeta_j$}};
\draw[black,fill=black]	(2,-2)		circle(0.075cm);
\node[right] at (2,-2)		{\footnotesize{$\overline{\zeta_j}$}};
\draw[black,fill=black]	(-2,2)		circle(0.075cm);
\node[left] 	at (-2,2)		{\footnotesize{$-\overline{\zeta_j}$}};
\end{tikzpicture}
\end{minipage}
\qquad
\begin{minipage}{0.45\textwidth}
\begin{tikzpicture}[scale=0.7]
\draw[thick,->-] 	(0,4)	--	(0,0);
\draw[thick,->-]	(0,-4)	--	(0,0);
\draw[thick,->-]	(0,0)	--	(-4,0);
\draw[thick,->-]	(0,0)	--	(4,0);
\draw[gray!60,fill=gray!60]	(2,2)		circle(0.075cm);
\node[gray!60,right] at (2,2)		{\footnotesize{$\zeta_j$}};
\draw[gray!60,fill=gray!60]	(-2,-2)	circle(0.075cm);
\node[gray!60,left] at (-2,-2)		{\footnotesize{$-\zeta_j$}};
\draw[gray!60,fill=gray!60]	(2,-2)		circle(0.075cm);
\node[gray!60,right] at (2,-2)		{\footnotesize{$\overline{\zeta_j}$}};
\draw[gray!60,fill=gray!60]	(-2,2)		circle(0.075cm);
\node[gray!60,left] 	at (-2,2)		{\footnotesize{$-\overline{\zeta_j}$}};
\draw[thick,->,>=stealth]	(3,2) 		arc(0:-180:1cm);
\draw	[thick]					(1,2) 		arc(180:0:1cm);
\draw[thick,->,>=stealth]	(-3,2) 	arc(180:360:1cm);
\draw[thick]					(-1,2) 	arc(0:180:1cm);
\draw[thick,->,>=stealth]	(-3,-2)	arc(180:0:1cm);
\draw[thick]					(-1,-2)	arc(0:-180:1cm);
\draw[thick,->,>=stealth]	(3,-2)		arc(0:180:1cm);
\draw[thick]					(1,-2)		arc(180:360:1cm);
\node[right] 	at (3.2,2) 		{$\gamma_j$};
\node[left]		at (-3.2,2) 	{$-\overline{\gamma_j}$};
\node[left]		at	(-3.2,-2)	{$-\gamma_j$};
\node[right]	at	(3.2,-2)		{$\overline{\gamma_j}$};
\end{tikzpicture}
\end{minipage}

\label{fig:Sigma}
\end{figure}
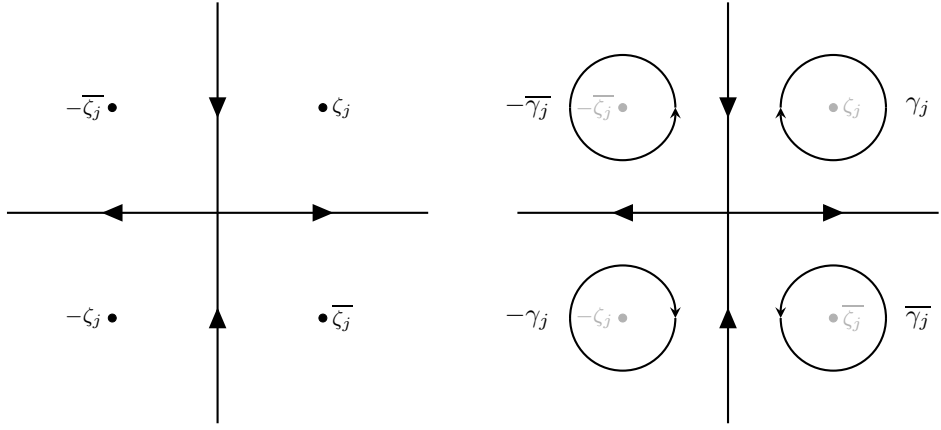

The contour $\Sigma$ is mapped  to $\R$ with its usual orientation by the orientation-preserving map $\zeta \mapsto \zeta^2$. Define
\begin{equation}
\label{rho->r}
r(\zeta) = \zeta \rho(\zeta^2), \quad \zeta_j = \sqrt{\lam_j}, \quad c_j = \frac{1}{2}C_j
\end{equation}
where we take the principle branch of the square root so that $\zeta_j \in \C^{++}$.
Note that $r$ is differentiable, odd, and $r \in L^2(\Sigma)$. If we set
$$\br(\zeta) = \eps \overline{r(\overline{\zeta})},$$ then $1 - r \br > 0$ strictly since $\rho \in Y$. 
If $\eps=-1$, the following RHP with data $\left(r, \{ (\zeta_j, c_j)\}_{j=1}^N \right)$ is  amenable to Zhou's vanishing theorem \cite[Theorem 9.3]{Zhou89}  We denote by $\pm \gamma_j$ and $\pm \overline{\gamma_j}$ small
circular contours around the points $\pm \zeta_j$ and $\pm \overline{\zeta_j}$ and denote by 
$\Sigma'$ the augmented contour
$$ \Sigma' = \Sigma \cup \left( \cup_j \pm \gamma_j \right) \cup \left( \cup_j \pm \overline{\gamma_j} \right) $$
shown on the right-hand side of Figure \ref{fig:Sigma}.

\begin{RHP}
\label{RHPX}
Given $r \in L^\infty(\R) \cap L^2(\R)$ with 
$$
r(-\zeta) = -r(\zeta), \quad 
\inf_{\zeta \in \Sigma} \left(1 - r(\zeta) \br(\zeta) \right) >0 ,
$$ 
and 
$\left\{ (\zeta_j, c_j )\right\}_{j=1}^N \in \left( \C^{++} \times \C^\times \right)^N$, find a function
$M: \C \setminus \Sigma' \rarr SL(2,\C)$ with continuous boundary values $M_\pm(\zeta;x)$ on $\Sigma'$ so that 
$ M_\pm(\dotarg;x)-I  \in \dee C(L^2)$ and $ M_+(\zeta;x) = M_-(\zeta;x) e^{-ix\zeta^2 \ad \sigma_3}v(\zeta)$
where
$$
\left. v \right|_{\Sigma} = 
\begin{pmatrix}
1- r(\zeta) \br(\zeta)	&	r(\zeta)	\\
\br(\zeta)	&	1
\end{pmatrix},
\quad
\left. v \right|_{\pm\gamma_j}=
\begin{pmatrix}
1 & 0 \\ \frac{c_j}{\zeta \mp \zeta_j} & 1
\end{pmatrix},
\quad
\left. v \right|_{\pm \overline{\gamma_j}}=
\begin{pmatrix}
1 & \frac{-\eps \overline{c_j}}{\zeta \mp \overline{\zeta_j}} \\ 0 & 1
\end{pmatrix}.
$$
\end{RHP}

For the remainder of this subsection we take $\eps=-1$. We may factorize $v = (I-w_-)^{-1}(I+w_+)$ where
\begin{equation}
\label{RHPX.w}
\left. (w^+,w^-)\right|_\Sigma =
\left(
	\begin{pmatrix}
	0	&	0	\\	-\br		&	0	
	\end{pmatrix},
	\begin{pmatrix}
	0	&	r	\\	0		&	0
	\end{pmatrix}
\right),
\quad
\left. w^- \right|_{\pm \gamma_j} = 0, 
\quad
\left. w^+\right|_{\pm \overline{\gamma_j}} = 0.
\end{equation}
It is easy to see that $v$
satisfies the symmetry hypothesis of \cite[Theorem 9.3]{Zhou89}. Although our data $w^\pm$ are
slightly less regular than those considered in \cite{Zhou89}, one can mimic the proof there to
conclude that the following vanishing lemma holds.

\begin{lemma}
\label{lemma:null}
If $\mu^{(0)} \in \ker_{L^2}(I - \calC_w)$, then $\mu^{(0)} = 0$.
\end{lemma}

By the remarks following Theorem \ref{thm:unique}, we can complete its proof by constructing from $\nu^{(0)}$ a solution $\mu^{(0)}$ of $\mu^{(0)} =\calC_w \mu^{(0)}$. In what follows, we will drop
$x$-dependence of $\nu^{(0)}$ and $\mu^{(0)}$ since $x$ remains fixed throughout the discussion.
Choose the contours $\gamma_j$, $1 \leq j \leq N$, in Problem \ref{RHPX} to be the boundaries of neighborhoods $D_j$ centered at $\zeta_j$ so that $D_j$ maps to the interior of $\Gamma_j$ under
the map $z \mapsto z^2$. Recall that $\nu^{(0)}$ extends to an analytic function on the interior
of each $\Gamma_j$. We now define
\begin{equation*}
\mu_{11}^{(0)}(\zeta)= \nu_{11}^{(0)}(\zeta^2), \quad \mu_{12}^{(0)}(\zeta) = \zeta \nu_{12}^{(0)}(\zeta^2), \quad \zeta \in \Sigma'
\end{equation*}
Note that $\mu_{11}^{(0)}(\zeta)$ and $\mu^{(0)}_{12}(\zeta)$ are analytic in a neighborhood of
$\pm \gamma_j$ and $\pm \overline{\gamma_j}$ for $1 \leq j \leq N$. We need to study regularity and decay of $\mu^{(0)}_{11}$ and $\mu^{(0)}_{12}$ for $\zeta \in \Sigma$. 
The first step is:

\begin{lemma}
\label{lemma:mu1}
The measurable functions $r(\zeta) \mu_{11}^{(0)}(\zeta)$ and $\br(\zeta) \mu^{(0)}_{12}(\zeta)$ belong to $L^1(\Sigma) \cap L^2(\Sigma)$. 
\end{lemma}

\begin{proof}
This is an easy consequence of the bound $|\rho(\lam)| \leq C(1+|\lam|)^{-1}$, H\"{o}lder's inequality, and the change of variables formula
\[ 
\int_\Sigma g(\zeta^2) \, |d\zeta| = \int_{-\infty}^\infty |g(\lam)| \, \frac{d\lam}{|\lam|^{1/2}}.  
\] 
\end{proof}

Using this fact, we can prove:

\begin{lemma}
\label{lemma:mu2}
The measurable functions $\mu_{11}^{(0)}$ and $\mu_{12}^{(0)}$ obey the homogeneous analogue of  integral equations 
\eqref{RHP1c.11.Sig}--\eqref{RHP1c.11.Gam*}
pointwise a.e.\ on $\Sigma$ and pointwise on $\pm \gamma_j$ and $\pm \overline{\gamma_j}$,
$1 \leq j \leq N$. In particular, $\restrict{\mu^{(0)}_{11}}{\Sigma}$ and $\restrict{\mu^{(0)}_{12}}{\Sigma}$ belong to $L^2(\Sigma)$.
\end{lemma}

\begin{proof}
This can be computed from the following change-of-variables formulas. Define $\sqrt{-u}=i\sqrt{u}$ for $u>0$. If $f$ is a measurable function on $\Sigma$ whose restriction to each ray is 
smooth and compactly supported away from zero, it is easy to prove the pointwise identities
\begin{equation}
\label{Sigma->Lam}
\left( C_\Sigma^\pm f \right)(\zeta)  = \left(C_\R^\pm g\right)(\zeta^2) + \zeta \left( C_\R^\pm h \right)(\zeta^2)
\end{equation}
where
\begin{equation}
\label{gh}
g(u) = \frac{1}{2}\left( f(\sqrt{u}) + f(-\sqrt{u}) \right), \quad
h(u) = \frac{1}{2\sqrt{u}} \left( f(\sqrt{u}) - f(-\sqrt{u}) \right).
\end{equation}
If $f \in L^2(\Sigma)$ and the corresponding functions $g, h$ defined by \eqref{gh} belong to $L^2(\R)$, one can approximate
$f$ by smooth, compactly supported functions and use the continuity of $C^\pm_\Sigma$ and 
$C^\pm_\R$ on their respective $L^2$-spaces to prove that \eqref{Sigma->Lam} holds pointwise a.e.
Applying \eqref{Sigma->Lam} as extended to the homogeneous analogue of \eqref{RHP2c.11.R}--\eqref{RHP2c.11.Lam*} for $\nu^{(0)}$, we obtain the homogeneous analogue of \eqref{RHP1c.11.Sig}--\eqref{RHP1c.11.Gam*} for $\mu^{(0)}_{11}$ and $\mu^{(0)}_{12}$ as identities holding pointwise a.e. From these identities and Lemma \ref{lemma:mu1} we can read off that $\restrict{\mu^{(0)}_{11}}{\Sigma}$ and $\restrict{\mu^{(0)}_{12}}{\Sigma}$ belong to $L^2(\Sigma)$.
\end{proof}

\begin{proof}[Proof of Theorem \ref{thm:unique}, $\eps=-1$]
By Lemma \ref{lemma:mu2}, given $\nu^{(0)} \in \ker_{L^2}(I-\calC_W)$, we can construct a matrix-valued solution $\mu^{(0)}$ to the Beals-Coifman equations
$$ (I - \calC_w) \mu = 0$$
using the extension by symmetry
$$ \mu(x,\zeta) = 
\begin{pmatrix}
\mu_{11}(x,\zeta)	&	\mu_{12}(x,\zeta)	\\[5pt]
-\overline{\mu_{12}(x,\zetabar)}	&	 \overline{\mu_{11}(x,\zetabar)}
\end{pmatrix}.
$$
and conclude from Lemma \ref{lemma:null} that $\mu^{(0)} \equiv 0$, hence $\nu^{(0)} \equiv 0$.
\end{proof}

We close this subsection by discussing uniqueness for Problem \ref{RHP2} when $\eps=+1$. There is a simple bijection (see the following lemma) between the solution spaces for Problem \ref{RHPX} with $\eps=-1$ and that of an analogous RHP with different scattering data for $\eps=+1$. Thus  Lemma \ref{lemma:null} also holds for $\eps=+1$. One can then mimic the proof of Theorem \ref{thm:unique} to show that Problem \ref{RHP2.row} is also uniquely solvable when $\eps=+1$.

\begin{lemma}
The map 
\begin{equation*}
M(x,\zeta) \mapsto M^\sharp(x,\zeta)=\sigma_2 M(x, i \zeta ) \sigma_2 
\end{equation*}
maps the solution space of  the Riemann-Hilbert problem  for $\eps=-1$  and scattering data 
$$\left( r, \{ (\zeta_k,c_k) \}_{k=1}^N \right)$$  
onto the solution space of the Riemann-Hilbert problem for $\eps=1$ and scattering data
$$\left( s, \{ (\eta_k, d_k) \}_{k=1}^N \right)$$ 
where
\begin{equation*}
 s(\zeta) = \overline{r(i\zetabar)}, \quad \eta_k = i\overline{\zeta_k},
\quad d_k = i\overline{c_k}.
\end{equation*}
\end{lemma}

We omit the proof.

\subsection{Lipschitz Continuity of the Scattering Map}
\label{sec:Lip}

In this section, we prove Theorem \ref{thm:I}. We denote by $V'$ a bounded subset of 
$V_N$. 
The complete proof of Theorem \ref{thm:I} follows in outline the proof of Theorem 1.9 in Section 6 of Paper I.  That is,
\begin{itemize}
\item[(1)]	We prove that the ``right'' Riemann-Hilbert problem, Problem \ref{RHP2}, and the reconstruction formula \eqref{q.lam} yield a Lipschitz continuous map from any bounded subset $V'$ of
$V_{N}$ to $H^{2,2}(a,\infty)$ for any $a \in \R$.
\item[(2)]	We prove that an analogous ``left'' Riemann-Hilbert problem and its reconstruction formula yield a Lipschitz continuous map from any bounded subset $V'$ of $V_{N}$ to $H^{2,2}(-\infty,a)$ for any $a \in \R$.
\item[(3)]	We show that the left- and right-hand reconstructions of $q$ from the same scattering data coincide on $(-a,a)$ for any $a \in \R$. 
\item[(4)]	We show that $\calI \circ \calR$ is the identity on $U_{N}$. 
\end{itemize}

We show how the analysis of Paper I, Section 6 must be modified to accommodate solitons in order to carry out step (1). The modifications to step (2) are analogous and the argument for step (3) is identical to that given in Paper I, Lemma 6.16. The argument for step (4) rests (as is standard) on the uniqueness of Beals-Coifman solutions, compare Proposition 6.17 in Paper I.  Thus, we will discuss the details of step (1) only. A statement and analysis of the left-hand Riemann-Hilbert problem may be found in \cite[Section 6.3]{Liu17}

To carry out step (1), 
we study the unique solution to Problem \ref{RHP2} at $t=0$ and the reconstructed
potential given by \eqref{q.lam}.  We solve the Beals-Coifman integral equations for Problem \ref{RHP2} as recorded in Appendix \ref{app:BC.RHP2}.  
Here and in Appendix~\ref{app:BC}, we denote
\begin{equation}
\label{rhoCx}
\rho_x(\lam) = e^{-2i\lam x} \rho(\lam), \quad C_{j,x} = C_j e^{2i\lam_j x}. 
\end{equation}

We begin with several key observations about the integral equations \eqref{RHP2c.11.R}--\eqref{RHP2c.11.Lam*}. 
First, the solution of \eqref{RHP2c.11.R}--\eqref{RHP2c.11.Lam*} is completely determined by 
the functions 
\begin{equation}
\label{keydata.BC}
\restrict{\nu_{11}}{\R}, \quad
\restrict{\nu_{12}}{\R}, \quad
\left\{ \restrict{\nu_{12}}{\Gamma_j} \right\}_{j=1}^N, \quad
\left\{\restrict{\nu_{11}}{\Gamma_j*}  \right\}_{j=1}^N.
\end{equation}
Given these functions, $\restrict{\nu_{12}}{\Gamma_j^*}$
and $\restrict{\nu_{11}}{\Gamma_j}$ can be read off from equations \eqref{RHP2c.12.Lam}
and \eqref{RHP2c.11.Lam*} respectively, which give an analytic continuation for $\nu_{12}$ and 
$\nu_{11}$ to the respective domains solely in terms of the data \eqref{keydata.BC}. For this reason,
it suffices to consider the system consisting of \eqref{RHP2c.11.R}, \eqref{RHP2c.12.R},
\eqref{RHP2c.12.Lam}, and \eqref{RHP2c.11.Lam*} for the data \eqref{keydata.BC}.

Second, we can use the Cauchy integral formula and the fact that $\nu_{11}$ and $\nu_{12}$ 
have analytic continuations to the interiors of $\Gamma_j^*$ and $\Gamma_j$ to reduce 
equations \eqref{RHP2c.11.R}, \eqref{RHP2c.12.R},
\eqref{RHP2c.12.Lam}, and \eqref{RHP2c.11.Lam*} to the following \emph{algebraic integral equations} for the data
\begin{equation*}
\restrict{\nu_{11}(x,\dotarg)}{\R}, \quad
\restrict{\nu_{12}(x,\dotarg)}{\R}, \quad
\left\{ \nu_{12}(x,\lambda_j) \right\}_{j=1}^N, \quad
\left\{ \nu_{11}(x,\overline{\lam_j}) \right\}_{j=1}^N
\end{equation*}
The equations are:
\begin{align}
\label{AI.11.R.pre}
\left. \nu_{11}(x,\lam)\right|_\R
	&=	1	+	
				C^-	\left(
							 (\dotarg) \overline{\rho_x (\dotarg)}
							\nu_{12}(x,\dotarg)
						\right)(\lam)
			+  \sum_j 
							\frac
								{C_{j,x} \lam_j \nu_{12}(x,\lam_j) } 
								{\lam - \lam_j}
						\\[5pt]
\label{AI.12.R}
\left. \nu_{12}(x,\lam) \right|_\R
	&=	
				C^+	\left(
								\rho_x(\dotarg) \nu_{11}(x,\dotarg)
						\right)(\lam)
				-	\sum_j 
							\frac
								{  \overline{C_{j,x}}
									\nu_{11}(x,\overline{\lam_j}) 
								}
								{\lam- \overline{\lambda_j} }
						\\[5pt]
\label{AI.11.Lam*.pre}
\nu_{11}(x,\overline{\lam_i})
	&=	1	+	
				C_\R	
						\left(
							 (\dotarg) \overline{\rho_x (\dotarg)}
							\nu_{12}(x,\dotarg)
						\right)(\overline{\lam_i})
			+  \sum_j 
							\frac
								{\nu_{12}(x,\lam_j) C_{j,x} \lam_j}
								{\overline{\lam_i}- \lam_j}
						 \\[5pt]
\label{AI.12.Lam}
\nu_{12}(x,\lam_i)
	&=	
				C_\R	\left(
								\rho_x(\dotarg) \nu_{11}(x,\dotarg)
						\right)(\lam_i)
				-	\sum_j 
							\frac
								{  \overline{C_{j,x}}
									\nu_{11}(x,\overline{\lam_j}) 
								}
								{\lam_i- \overline{\lambda_j} }
\end{align}
These equations are actually equivalent to \eqref{RHP2c.11.R}, \eqref{RHP2c.12.R},
\eqref{RHP2c.12.Lam}, and \eqref{RHP2c.11.Lam*} since we can construct $\restrict{\nu_{12}}{\Lam_j}$ and $\restrict{\nu_{11}}{\Lam_j^*}$ by analytically continuing the right-hand sides of
\eqref{AI.12.R} and \eqref{AI.11.R.pre}. Recall the set $Y$ defined in \eqref{Y}. From Theorem \ref{thm:unique} and the remarks above, we have:

\begin{proposition}
\label{prop:AI.unique}
There exists a unique solution to the algebraic integral equations \eqref{AI.11.R.pre}--\eqref{AI.12.Lam} for any $\rho \in Y$ satisfying \eqref{Y.con} and $\{ (\lam_j, C_j) \} \in \left(\C^+ \times \C^\times \right)^N$.
\end{proposition}

Using \eqref{q.lam} with $t=0$ and the Cauchy integral representation
for $n(\lam;x)$ in terms of $\nu(x,\lam)$ (compare \eqref{mu-to-M}) we obtain the reconstruction formula
\begin{equation}
\label{q.lam2}
q(x)	=	-\frac{1}{\pi} \int_\R \rho_x(s) \nu_{11}(x,s) \, ds - 2i \sum_{j=1}^N \overline{C_{j,x}} \nu_{11}(x,\overline{\lam_j}).
\end{equation}
It is convenient to write 
$$\nu_0(x,\lam)=\nu_{11}(x,\lam)  - 1, \quad v_j = \nu_{11}(x,\overline{\lam_j})$$ 
so that the reconstruction formula \eqref{q.lam2} 
reads 
$$ q(x) = q_1(x) + q_2(x) + q_3(x) $$
where 
\begin{gather*}
q_1(x)	=	-\frac{1}{\pi} \int_\R \rho_x(s) \, ds, \quad
q_2(x)	=	-\frac{1}{\pi}	\int_R \rho_x(s) \nu_0(x,s) \, ds,\\
q_3(x)	=	-2i \sum_{j=1}^N \overline{C_{j,x}} v_j.
\end{gather*}
We seek to prove that the map $\calD \rarr q$ is Lipschitz continuous as a map from $V_N$ to $H^{2,2}([a,\infty))$ for any $a$. As $\rho \in H^{2,2}(\R)$ and $q_1$ is a Fourier transform of $\rho$, the map $\calD \rarr q_1$ need not be considered further. The term $q_3$ will have the required continuity provided that $\calD \rarr v_j$ is Lipschitz continuous from $V_N$ to the bounded $C^2$ functions of $x$ on $\R$ since the $C_{j,x}$ are exponentials with decay at least $\bigO{e^{-d_\Lam x/2}}$ as $x \rarr +\infty$. The formulas
\begin{align*}
x^2 q_2(x)	
	&=	\frac{1}{4\pi} \int_\R e^{-2i\lam x} (\rho {\nu_0})_{\lam\lam} (x,\lam) \, d\lam,\\
q_2''(x)
	&=	\frac{1}{\pi} 
				\int_\R e^{-2i\lam x} \left( 4\lam^2 \rho(\lam) \nu_0(x,\lam)\right)  \, d\lam \\
	&\quad + \frac{1}{\pi} 
				\int_\R e^{-2i\lam x} \bigl( 4i\lam\rho(\lam) (\nu_0)_x(x,\lam)- \, \rho(\lam) (\nu_0)_{xx}(x,\lam)\bigr) \, d\lam,
\end{align*}
imply that, in order to show that $\calD \mapsto q_2$ is Lipschitz continuous from $V'$ to $H^{2,2}(\R)$, it suffices to show that
\begin{equation}
\label{nu.todo}
\nu_0, \quad (\nu_0)_x, \quad (\nu_0)_{xx}, \quad (\nu_0)_\lam,
\text{ and }\langle \lam \rangle^{-1} (\nu_0)_{\lam\lam}
\end{equation}
are Lipschitz continuous from $V'$ to $L^2((a,\infty) \times \R)$.

Thus, it suffices to prove continuous dependence of 
$$\nu^\sharp(x) = \left(\nu_0(x,\dotarg), \left\{v_j(x)-1\right\}_{j=1}^N\right)$$ 
on the scattering
data $\calD$.  For each fixed $x$, we view $\nu^\sharp$ as an element of
\begin{equation}
\label{X}
X = L^2(\R) \oplus \C^N.
\end{equation}
Substituting \eqref{AI.12.R} into \eqref{AI.11.R.pre} and \eqref{AI.12.Lam} into \eqref{AI.11.Lam*.pre} and using the definition of $\nu^\sharp$, we 
obtain a closed system of equations for $\nu^\sharp(x)$ of the form
\begin{equation}
\label{nusharp.int}
\nu^\sharp(x) = \calK_x \mathbf{e} + \calK_x \nu^\sharp 
\end{equation}
where
$\mathbf{e} = \left(1, \{ 1, 1, \dots 1 \}\right) \in L^\infty(\R) \oplus \C^N$. In terms of the direct sum decomposition \eqref{X}, 
\begin{equation}
\label{calK}
 \calK_x (f,h) = \begin{pmatrix}
\calK_{00}	&	 	\calK_{01}	\\
\calK_{10}	&		\calK_{11}
\end{pmatrix}
\begin{pmatrix}
f	\\	h
\end{pmatrix}
\end{equation}
where,  for $f \in L^2(\R)$ and $h \in C^N$, 
\begin{align}
(\calK_{00} f)(\lam)
	&=	\left(S f\right)(\lam)
		+ \sum_j 	\frac{C_{j,x} \lam_j}{\lam-\lam_j} \,
								C_\R
									\left(
										\rho_x(\dotarg) f(\dotarg)
									\right)(\lam_j)	
\label{K00}\\
(\calK_{01} h)(\lam) 
	&=	- \sum_j h_j \, \overline{C_{j,x}} \,
								C^-
									\left[ 
										\frac{(\dotarg)\overline{\rho_x(\dotarg)}}
											{\dotarg - \lam_j} 
									\right](\lam)
		+ \sum_{j,k}	-h_k\, \frac{\lam _j C_{j,x} \overline{C_{k,x}}}
											{(\lam-\lam_j)(\lam_j - \overline{\lam_k})}	
\label{K01}\\
(\calK_{10} f)_i
	&=	 C_\R \left[ (\dotarg) \overline{\rho_x(\dotarg)} 
									C^+\left( \rho_x(\diamond) f(\diamond) \right)
										(\dotarg)
							\right](\overline{\lam_i})
\label{K10}\\
\nonumber
	&\quad
			+	\sum_j	\frac{C_{j,x} \lam_j}{\overline{\lam_i}-\lam_j}
									C_\R \left(\rho_x(\dotarg) f(\dotarg)\right)(\lam_j) \\
(\calK_{11} h)_i
	&=
		- 	\sum_{j}  h_j \overline{C_{j,x}} 
									C_\R 
										\left[ \frac{(\dotarg) 
												\overline{\rho_x(\dotarg)}}
												{\dotarg - \overline{\lam_j}}
										\right] (\overline{\lam_i})
		-	\sum_{j,k}	\frac{ \lam_j C_{j,x} \overline{C_{k,x}}}
 											{\overline{\lam_i}-\lam_j} h_k	
\label{K11}
\end{align}
In \eqref{K00}
\begin{equation}
\label{Sx}
(S h)(\lam) =  C^-
	\left[ 
				(\dotarg) \overline{\rho_x(\dotarg)} 
				\, C^+\left(\rho_x (\diamond) h(\diamond)\right) 
	\right] (\lam),
\end{equation}
and all of the remaining operators are of finite rank. We will usually suppress the dependence of $\calK_x$ on the scattering data $\calD = \left(\rho, \{(\lam_j, C_j)\}_{j=1}^N\right)$ but sometimes write $\calK_x(\calD)$ when needed to make the dependence explicit.

\begin{proposition}
\label{prop:nusharp.int.unique}
For any scattering data $\left( \rho, \{ (\lam_j, C_j)\}_{j=1}^N \right) \in Y \times \left(\C^+ \times \C^\times\right)^N$, there exists a unique solution to \eqref{nusharp.int}.
\end{proposition}

\begin{proof}
If $\widetilde{\nu}^{(0)} \in \ker_X(I-\calK_x)$, we may construct a solution to the homogeneous algebraic integral equations (i.e., \eqref{AI.11.R.pre}--\eqref{AI.12.Lam} with ``$1$'' deleted in \eqref{AI.11.R.pre}
and \eqref{AI.11.Lam*.pre}). It now follows from Proposition \ref{prop:AI.unique} that $\widetilde{\nu}^{(0)}  =0$. By the Fredholm alternative, $(I - \calK_x)\nu = f$ has a unique solution for any $f \in X$.
\end{proof}

In order to prove the required continuity of $\nu^\sharp$, it sufffices to prove that:
\begin{enumerate}
\item[(1)]	The resolvent $(I - \calK_x)^{-1}$ is a bounded operator from $X$ to itself uniformly in $x \in (a,\infty)$ so that it induces a bounded operator from $L^2((a,\infty),X)$ to itself;
\item[(2)]	The inhomogeneous term $\calK_x \mathbf{e}$ together with appropriate derivatives in $x$ and $\lam$ belong to the space $L^2((a,\infty),X)$;
\item[(3)]	The integral equation \eqref{nusharp.int} can be differentiated in $x$ and $\lam$ to obtain estimates on $x$ and $\lam$ derivatives of $\nu_0$ and $x$ derivatives of $\nu_j$.
\end{enumerate}

We will prove the following propositions in three subsections.  In what follows, we denote by $V'$ a bounded subset of $V_N$ as defined in the remarks after Theorem \ref{thm:dense}, and by $\calD$ a generic element of $V'$.

\begin{proposition}
\label{prop:nu.L2}
The unique solution $\nu^\sharp$ of \eqref{nusharp.int}
satisfies
$$\norm[L^2([a,\infty),X) \cap L^\infty([a,\infty),X)]{\nu^\sharp} \lesssim 1$$ 
with estimates uniform in $\calD \in V'$.
\end{proposition}

\begin{proposition}
\label{prop:nu.lam}
The estimates
$$
\norm[L^2([a,\infty) \times \R)]{{\nu_0}_\lam} +
\norm[L^2([a,\infty) \times \R)]
		{\langle \diamond \rangle^{-1}
			{\nu_0}_{\lam\lam}(\dotarg,\diamond)}
\lesssim 1
$$
and 
$$
\norm[L^\infty([a,\infty) \times \R)]{{\nu_0}_\lam} +
\norm[L^\infty([a,\infty) \times \R)]
		{\langle \diamond \rangle^{-1}
			{\nu_0}_{\lam\lam}(\dotarg,\diamond)}
\lesssim 1
$$
holds uniformly for $\calD \in V'$.  Moreover, the maps 
$\calD \rarr \nu_\lam$ and $\calD \rarr \langle \dotarg \rangle^{-1}\nu_{\lam\lam}$ are Lipschitz continuous.
\end{proposition}

\begin{proposition}
\label{prop:nu.x}
Denote by $\nu^\sharp(x,\dotarg)$ the unique solution to \eqref{nusharp.int}. Then $\nu^\sharp$, $(\nu^\sharp)_x$,
and $(\nu^\sharp)_{xx}$ all belong to $L^2([a,\infty),X)$
with norms bounded uniformly in $\rho$ in a fixed bounded subset of $H^{2,2}(\R)$ with $\left(
1+\lam|\rho(\lam)|^2 \right) \geq c>0$
for a fixed $c$, and $\{ \lam_j, C_j\}$ in a fixed compact subset of $(\C^+ \times \C^\times)^N$. 
\end{proposition}

\begin{proof}[Proof of Theorem \ref{thm:I}, given Propositions \ref{prop:nu.L2}--\ref{prop:nu.x}]
Recall from the discussion following the statement of Theorem \ref{thm:I} that we give details only for Step (1) of the four steps since step (2) is similar to step (1) while steps (3) and (4) follow closely the argument of Paper I. 

To carry out step (1), we note that Propositions \ref{prop:nu.L2}--\ref{prop:nu.x} immediately imply that the functions \eqref{nu.todo} are Lipschitz in the scattering data and that the functions $\nu(x,\overline{\lam_j})$ are bounded Lipschitz functions of the scattering data. This gives the necessary control to prove that $q_2$ and $q_3$ in \eqref{q.lam2} have the required continuity properties as maps from $V'$ to $H^{2,2}(\R)$.
\end{proof}

\subsubsection{Resolvent Estimates, $L^2$ Estimates on $\nu^\sharp$}

We begin the study of \eqref{nusharp.int}. It is easy to see that $\calK_x \mathbf{e} \in X$ if $\rho \in H^{2,2}(\R)$.  Thus, to solve \eqref{nusharp.int}  for $\nu^\sharp$, we need good estimates on 
$(I -\calK_x)^{-1}$. First, we observe:

\begin{lemma}
\label{lemma:Kx.large}
The operator $I-\calK_x : X \rarr X$ is Fredholm and
$$\lim_{x \rarr \infty} \sup_{\calD \in V'} \norm[X \rarr X]{\calK_x} = 0$$ where $\calD$ in a bounded subset $V'$ of $V_N$.
\end{lemma}

\begin{proof}
The operator $\calK_x$ is a finite-rank perturbation of the operator
$S:L^2(\R) \rarr L^2(\R)$ given by \eqref{Sx}. The operator $S$
was shown to be compact in Paper I, Lemma 6.7. In the same Lemma we showed that $\norm[L^2 \rarr L^2]{S} = 0$ uniformly for $\rho$ in a bounded subset of $H^{2,2}(\R)$. Coefficients of all remaining terms have exponential decay at least $\bigO{e^{-d_\Lambda x/2}}$ as $x \rarr \infty$ owing to factors $C_{j,x}$ or their complex conjugates. This decay rate is uniform in bounded subsets of $V_N$ since $d_\Lambda$ has a fixed lower bound on such sets. 
\end{proof}

We seek an estimate on $\norm[X \rarr X]{(I-\calK_x)^{-1}}$ uniform in $\calD$ in a bounded subset of $V_N$ and $x \in [a,\infty)$ for any fixed $a$. It follows from Lemma \ref{lemma:Kx.large} that, given a bounded subset $V'$ of $V_N$, there is an $R>0$ depending on $V'$ so that 
$$ \sup_{\calD \in V', x \geq R} \norm[X \rarr X]{\calK_x}  < 1/2 $$
so that
\begin{equation*}
\sup_{\calD \in V', x \geq R} \norm[X \rarr X]{(I-\calK_x)^{-1}} \leq 2. 
\end{equation*}
Thus, to obtain a uniform resolvent bound, it suffices to estimate the resolvent
for $\calD \in V'$ and $a \leq x \leq R$ for a fixed $a$. To do so will use a continuity-compactness argument similar to the one used in Lemma \ref{lemma:direct.S.res}.  Define  $Y$ as in \eqref{Y}. Recall that $H^{2,2}(\R)$ is compactly embedded in $Y$, and that
$\norm[L^\infty(\R)]{\rho}$ and $\norm[L^\infty(\R)]{(\dotarg)\rho(\dotarg)}$ are bounded for
$\rho \in Y$.

We will prove:
\begin{lemma}
\label{lemma:Kx.small}
For any bounded subset $V'$ of $V_N$ and any $-\infty < a < r < \infty$,
$$\sup_{(\calD,x) \in V' \times [a,r]} \norm[\calB(X)]{(I-\calK_x)^{-1}} \lesssim 1 $$
\end{lemma}

\begin{proof}
First, for data $\calD \in Y \times (\C^+ \times \C^\times)^N$, the operator $\calK_x$ is bounded from $X$ to itself, as follows from \eqref{K00}--\eqref{K11} and the fac that $\rho$ and $\lam \rho(\lam)$ are $L^\infty$ functions.

Second,Proposition \ref{prop:nusharp.int.unique} implies that the resolvent exists for $\calD = \left(\rho,\{(\lam_j,C_j)\}\right) \in Y \times (\C^+ \times \C^\times)^N$ and the map
$$ \left(\rho,\{(\lam_j,C_j)\}_{j=1}^N\right) \mapsto (I-\calK_x)^{-1} $$
is continuous.

Third, if $V'$ is a bounded subset of $V_N$, then $V'$ is compactly embedded in
$Y \times (\C^+ \times \C^\times)^N$, as follows from he compact embedding of $Y$ in $H^{2,2}(\R)$ and the fact that the $(\lam_j, C_j)$ that occur in a bounded subset of $V_N$ run through compact subsets of $\C^+$ and $\C^\times$.

Together, these three facts imply that the image of $V' \times [a,r]$ in $\calB(X \rarr X)$ is a precompact set, hence bounded, for any $-\infty< a < r < \infty$.
\end{proof}

For $\calD = \left( \rho, \{(\lam_j, C_j) \}_{j=1}^N \right)$, define
\smallskip
$$ \norm[V_N]{\calD} = \norm[H^{2,2}(\R)]{\rho} + \sup_j |\lam_j| + \sup_j |C_j|. $$
Although $V_N$ is not complete in this norm (it does not control lower bounds on $C_j$ and $|\imag \lam_j|$) it will suffice for the continuity estimates we need.

\begin{lemma}
\label{lemma:res}
Let $V'$ be a bounded subset of $V_N$.
The resolvent $(I-\calK_x)^{-1}$ satisfies the uniform estimate
$$ \sup_{(x,\calD) \in [a,\infty) \times V'} \norm[X \rarr X]{(I-\calK_x)^{-1}} \lesssim 1 $$
and the Lipschitz estimate
$$ \norm[X \rarr X]{(I-K_x(\calD_1))^{-1} - (I-\calK_x(\calD_2))^{-1}}
\lesssim \norm[V_N]{\calD_1 - \calD_2}
$$
with uniform Lipschitz constant for $\calD_1, \calD_2 \in V'$.
\end{lemma}

\begin{proof}
The uniform resolvent estimate is an immediate consequence of Lemmas \ref{lemma:Kx.large} and \ref{lemma:Kx.small}. The Lipschitz estimate follows from the second resolvent identity and the estimate
$$ \norm[X \rarr X]{K_x(\calD_1) - K_x(\calD_2)}
\lesssim \norm[V_N]{\calD_1 - D_2}$$
which is easily proved from the formulas \eqref{K00}--\eqref{K11}.
\end{proof}

Owing to the uniform estimate, we can lift $(I-\calK_x)^{-1}$ to an operator
which we denote by $(I-\calK)^{-1}$ on $L^2((a,\infty);X)$ by the formula
$$ \left[(I-\calK)^{-1}f\right](x,\lam) = \left[(I-\calK_x)^{-1}f(x,\dotarg)\right](\lam). $$
As an immediate consequence of Lemma \ref{lemma:res}, we have:

\begin{proposition}
\label{prop.res}
Let $V'$ be a bounded subset of $V_N$.
The resolvent $(I-\calK)^{-1}$ satisfies the estimates
$$ 
\norm[\calB(L^2([a,\infty);X)]{(I-{\calK})^{-1}}
\lesssim 1
$$
with constants uniform in  $\calD \in V'$ and
$$
\norm[\calB(L^2([a,\infty);X)]
		{(I-{\calK}(\calD_1))^{-1}- (I-{\calK}(\calD_2))^{-1}} 
\lesssim
\norm
	[V_N]	
	{\calD_1 - \calD_2}
$$
with constants uniform in  $\calD_1$, $\calD_2 \in V'$. 
\end{proposition}

To solve \eqref{nusharp.int}, we also need to control the inhomogeneous term $\calK_x e$. Let
\begin{equation}
\label{bff}
\bff	=	\calK_x \bfe 
		= 	\left(
				f^\sharp(x,\lam),
				\{ f^\sharp(x,\overline{\lam_j})\}_{j=1}^N 
			\right). 
\end{equation}
From \eqref{calK} we see that
$$
f^\sharp =	f_1^\sharp + f_2^\sharp + f_3^\sharp + f_4^\sharp
$$
where
\begin{align*}
f_1^\sharp(x,\lam)
	&=	
	\sum_k \frac{C_{k,x} \lam_k }{\lam-\lam_k}
			\left(
					\sum_j 
						\frac	{-\overline{C}_{j,x} }
								{\overline{\lambda}_j-\lambda_k} 
			\right)
\\
f_2^\sharp(x,\lam)
	&=	 \sum_k \frac{1}{\lam - \lam_k}
			\left(
				{\dint_{\mathbb{R}}\frac{\rho_x(s)}{s-\lambda_k}
				\dfrac{ds}{2\pi i} C_{k,x} \lam_k }
			\right)
\\
f_3^\sharp(x,\lam)
	&= -\sum_k\overline{C}_{k,x} 
			C^-\left[
					\frac{ (\dotarg)\overline{\rho_x(\dotarg)}  }
						{(\dotarg)-\overline{\lambda}_k} 
				\right](\lambda)	\\
f_4^\sharp(x,\lam)
	&=	 C^-
				\left\lbrace
						C^+
							\left[
								\rho_x(\diamond) 
							\right]  
							(\dotarg) 
							\overline{\rho_x(\dotarg)} 
				\right\rbrace(\lam)
\end{align*}
We can get $f^\sharp (x,\overline{\lambda}_j)$, $1 \leq j \leq N$, by substituting $\overline{\lambda}_j$  for $\lambda\in\mathbb{R}$  and changing the corresponding Cauchy projection $C^-$ to a Cauchy integral over the real line.

\begin{lemma}
\label{lemma:f1}
For $f^\sharp$ given by \eqref{bff} and indices $i=1,2,3,4$ and $1 \leq j \leq N$,
$$
\left| f^\sharp_i(x,\overline{\lam}_j) \right|+
\left|\dee f^\sharp_i(x,\overline{\lam}_j)/\dee x\right| +
\left|\dee^2 f^\sharp_i(x,\overline{\lam}_j)/\dee x^2\right|
\lesssim \left(1+ \norm[H^{2,2}(\R)]{\rho}\right)^2
$$
uniformly for $x \geq a$ and 
$\{ (\lam_j, C_j)\}$ in a fixed compact subset of $(\C^+ \times C^\times)^N$.
\end{lemma}

\begin{lemma}
\label{lemma:f2}
For $f^\sharp$ given by \eqref{bff} and indices $i=1,2,3 $ and $1 \leq j \leq N$ 
\begin{equation}
\label{f2.1}
\sup_{0 \leq k \leq 2}
	\norm[L^2([a,\infty))]
			{\dee^k f^\sharp_i(\dotarg,\overline{\lam}_j)/\dee x^k} 
\lesssim \left(1+ \norm[H^{2,2}(\R)]{\rho}\right)^2
\end{equation}
uniformly for 
$\{ \lam_j, C_j\}$ in a fixed compact subset of $(\C^+ \times \C^\times)^N$.  Moreover
\begin{equation}
\label{f2.2}
\sup_{0 \leq k \leq 2}
	\norm[L^2([a,\infty)\times \R)]
			{\dee^k f^\sharp_i(\dotarg,\diamond)/\dee x^k}
	\lesssim  \left(1+ \norm[H^{2,2}(\R)]{\rho}\right)^2\\
\end{equation}
and
\begin{equation}
\label{f2.3}
\norm[L^2([a,\infty)\times \R)]
			{\dee f_i^\sharp(\dotarg,\diamond)/\dee \lam} +
\norm[L^2([a,\infty)\times \R)]
			{\langle \diamond \rangle^{-1} 
			\dee^2 f_i^\sharp/\dee \lam^2 (\diamond,\dotarg)}\\
	\lesssim  \left(1+ \norm[H^{2,2}(\R)]{\rho}\right)^2
\end{equation}
with the same uniformity.
For $f_4$ we have the following estimates:
\begin{subequations}
\begin{align}
\label{f4}
\norm[L^2_\lam]{{f^\sharp_4}(x,\dotarg)}	
	&\lesssim (1+|x|)^{-1} \norm[H^{2,0}]{\rho} \norm[L^2]{\rho}, \\
\label{f4.x}
\norm[L^2((-a,\infty) \times \R)]{(f^\sharp_4)_x}
	&\leq		\norm[L^{2,1}]{\rho} \norm[L^2]{\rho}\\
\label{f4.xx}
\norm[L^2((-a,\infty)\times \R)]{(f^\sharp_4)_{xx}}
	&\leq	\norm[L^{2,2}]{\rho} \norm[L^{2,1}]{\rho}\\
\label{f4.l}
\norm[L^2_\lam]{(f^\sharp_4)_{\lam}(x,\dotarg)}
	&\lesssim	(1+|x|)^{-1}\norm[H^2]{\rho}\norm[H^{2,2}]{\rho}\\
\label{f4.ll}
\norm[L^2_\lam]{ \langle \dotarg \rangle^{-1} (f^\sharp_4)_{\lam\lam}(x,\dotarg)}		
	&\lesssim	(1+|x|)^{-1} \norm[H^{2,2}]{\rho}^2
\end{align}
\end{subequations}
\end{lemma}

\begin{remark}
\label{rem:f.infty}
These estimates and Sobolev embedding imply that
$$
\norm[L^\infty([a,\infty);X)]{f^\sharp}
	\lesssim	 \left( 1 + \norm[H^{2,2}(\R)]{\rho} \right)^2
$$
\end{remark}

\begin{remark}
Since $\textbf{f}$ and its derivatives are bilinear in the data $\left(	\rho,\{ \lam_k \}, \{ C_k \}	\right)$, the estimates used to prove the lemma above can easily be adapted to show that $\textbf{f}$ and its derivatives are Lipschitz continuous as a function of $\left(	\rho,\{(\lam_k,C_k) \}_{k=1}^N	\right) \in V_N $. 
\end{remark}
\begin{proof}
We can establish the inequalities (\ref{f4})-(\ref{f4.ll}) using the conclusion from Paper I, Lemma 6.5 and (6.14). The other part of the lemma is trivial.
\end{proof}

Lemma \ref{lemma:res}, Proposition \ref{prop.res}, Lemmas \ref{lemma:f1}, and \ref{lemma:f2} together with Remark \ref{rem:f.infty} immediately imply Proposition \ref{prop:nu.L2}.

\subsubsection{$\lam$-Derivatives}

Next,  we consider $\lam$-derivatives of $\nu^\sharp$.  Write
$$
\nu^\sharp  
	= \left(\nu_0, \{ \widetilde{\nu}^*_j \}_{j=1}^N \right)
	=	\left(\nu_0, \widetilde{\nu}^*\right).
$$
Note that only $\nu_0$ depends on $\lam$, and obeys 
the integral equation
\begin{equation}
\label{int.nu1}
\nu_0	
	=	f_1 	+ \calK_{00}(\nu_0) 
				+ \calK_{01}(\widetilde{\nu}^*).
\end{equation}

We will prove Proposition \ref{prop:nu.lam} by differentiating \eqref{int.nu1}
and using resolvent bounds to estimate the derivatives.

The following estimates allow us to control the third term
$\calK_{01} (\widetilde{\nu}^*)$ in \eqref{int.nu1} and its derivatives in 
$\lam$.

\begin{lemma}
\label{lemma:K01}
Suppose that $\nu^\sharp \in L^2([a,\infty),X)$. Then
\begin{align}
\label{K01.lam}
\norm[L^2([a,\infty) \times \R)]
		{\left[\calK_{01} (\widetilde{\nu}^*)\right]_\lam}
&\lesssim \norm[H^{2,2}(\R)]{\rho} \norm[L^2(\R,\C^N)]{\widetilde{\nu}^*}	\\
\label{K01.lamlam}
\norm[L^2([a,\infty) \times \R)]
		{\left[\calK_{01} (\widetilde{\nu}^*)\right]_{\lam\lam}}
&\lesssim	\norm[H^{2,2}(\R)]{\rho} \norm[L^2(\R,\C^N)]{\widetilde{\nu}^*}  
\end{align}
hold, with estimates uniform 
in  $\{ (\lam_j, C_j)\}_{k=1}^N$ in a compact subset of $(\C^+ \times \C^\times)^N$.
Similar estimates hold, with the same uniformity, if the $L^2([a,\infty) \times \R)$-norm is replaced by the $L^\infty([a,\infty),L^2(\R))$-norm.
\end{lemma}

\begin{proof}
These estimates follow from the explicit formula
\begin{align*}
\calK_{01}\left(\widetilde{\nu}^*\right) 
	&=	-\sum_k 
				\widetilde{\nu}_k^*(x)
				\overline{{C}_{k,x}}
				C_{-}
				\left [ 
					\frac{(\dotarg)\overline{\rho_x(\dotarg)}} 
						{(\dotarg)-\overline{\lambda}_k}
				\right](\lambda) 
				\\
\nonumber
	&\quad +
		\sum_j
			\widetilde{\nu}_j^*(x)
			\eps\overline{C_{j,x}}
			\left(
				\sum_k
					\frac{\lambda_k C_{k,x} }
						{(\lambda-\lambda_k)
						(\lambda_k-\overline{\lambda}_j)} 
			\right)
\end{align*}
noting that the functions $C_{j,x}$ are bounded and exponentially decaying in $x$.
\end{proof}

The operator $\calK_{00}=S + S'$ (cf. \eqref{K00}) where
$S$ was studied in Paper I. We recall the following estimates from Paper I. Proofs may be constructed from computations in the  proof of Paper I, Lemma 6.13.

\begin{lemma}
\label{lemma:S0}
Suppose that $\rho \in H^{2,2}(\R)$. The following estimates hold for 
$h \in L^2([a,\infty) \times \R) \cap L^2([a,\infty),L^1(\R))$
\begin{align}
\label{S_0.lam}
\norm[L^2([a,\infty) \times \R)]{\frac{\dee {S}}{\dee \lam} \left[h\right]}	
	&\lesssim	
		\norm[H^{2,2}]{\rho}^2 \norm[L^\infty([a,\infty),L^2(\R))]{h} \\
\label{S_0.lamlam}
\norm[L^2([a,\infty) \times \R)]
	{\langle \lam \rangle^{-1}\frac{\dee^2 {S} }{\dee \lam^2} \left[h\right]}
		&\lesssim	
			\norm[L^1]{\widehat{\rho}'}
			\norm[L^{2,2}]{\widehat{\rho}}
			\norm[L^2([a,\infty),L^1(\R))]{\widehat{h}}		
\end{align}
where $\widehat{h}$ denotes the partial Fourier transform of $h(x,\lam)$ in the second variable, and the implied constants depend only on $a$.
\end{lemma}

\begin{remark}
\label{rem:S_0.lamlam}
Using the inequality
$$\norm[L^1(\R)]{\widehat{\psi}\,}
	\leq	\norm[L^2(\R)]{\psi} + \norm[L^2(\R)]{\psi_\lam}$$
we may rewrite the estimate \eqref{S_0.lamlam} as 
\begin{equation}
\label{S_0.lamlam.bis}
\norm[L^2([a,\infty) \times \R)]
	{\langle \lam \rangle^{-1}\frac{\dee^2 {S} }{\dee \lam^2} \left[h\right]}	
		\lesssim\norm[H^{2,2}(\R)]{\rho}^2 
				\left( 
					\norm[L^2([a,\infty) \times \R)]{h} +
					\norm[L^2([a,\infty) \times \R)]{h_\lam}	
				\right)		
\end{equation}
\end{remark}

Next, we study the operator $S'$.

\begin{lemma}
\label{lemma:S1}
The following estimates hold for the operator $S'$.
\begin{align}
\label{S1.lam}
\norm[L^2([a,\infty) \times \R)]
		{(S' h)_\lam}
	&\lesssim	
		\norm[L^2(\R)]{\rho} 
		\norm[L^2([a,\infty) \times \R)]{h}	\\
\label{S1.lamlam}
\norm[L^2([a,\infty) \times \R)]
		{(S' h)_{\lam\lam}}
	&\lesssim	
		\norm[L^2(\R)]{\rho} 
		\norm[L^2([a,\infty) \times \R)]{h}\\
\label{S1.x}
\norm[L^2([a,\infty) \times \R)]{(S')_x h}
	&\lesssim  \norm[H^{2,2}(\R)]{\rho} 
		\left( 
			\norm[L^2([a,\infty) \times \R)]{h}
		\right)\\
\label{S1.xx}
\norm[L^2([a,\infty) \times \R)]{(S')_{xx} (h)}
	&\lesssim	\norm[H^{2,2}(\R)]{h}
		\left(	
			\norm[L^2([a,\infty) \times \R)]{h}
		\right)
\end{align}
where the implied constants depend only on $a$ and $\{ (\lam_j , C_j) \}_{k=1}^N$.
\end{lemma}

\begin{remark}
In the estimates above, the derivative is taken with respect to $\lam$
on the composition $S' h$
for \eqref{S1.lam} and \eqref{S1.lamlam}, but the derivative of the \emph{operator only} with respect to the parameter $x$ is taken in \eqref{S1.x} and \eqref{S1.xx}.
\end{remark}

\begin{proof}
These estimates are easy consequences of the explicit formula
$$
(S'h)(x,\lam) = 
\sum_{j=1}^N \frac{C_{j,x} \lam_j}{\lam-\lam_j}
	C_\R \left(\rho_x(\dotarg) h(x,\dotarg)\right)(\lam_j)
$$
recalling \eqref{rhoCx}.
\end{proof}


\begin{proof}[Proof of Proposition \ref{prop:nu.lam}]
 We may use \eqref{int.nu1} together with Lemmas \ref{lemma:K01}, \ref{lemma:S0}, and \ref{lemma:S1} to estimate
\begin{align*}
\norm[L^2([a,\infty) \times \R)]{{\nu_0}_\lam} 
	&	\lesssim 
	\norm[L^2([a,\infty) \times \R)]{(f^\sharp)_\lam} +
	\norm[L^2([a,\infty) \times \R)]{(\calK_{00} \nu_0)_\lam} \\
	&\quad +
	\norm[L^2([a,\infty) \times \R)]{(\calK_{01} \nu_0)_\lam}\\
	& 	\lesssim
	\left(1+\norm[H^{2,2}(\R)]{\rho}\right)^2 
		\left(
				1+ \norm[L^\infty([a,\infty);L^2(\R))]{\nu_0}
				+ \norm[L^2([a,\infty);X)]{\nu^\sharp} 		
		\right)
\end{align*}
which shows that $\norm[L^2([a,\infty) \times \R)]{{\nu_0}_\lam}$ is bounded uniformly for $\calD \in V'$.

Differentiating \eqref{int.nu1} a second time yields the 
equation
$$
\langle \lam \rangle^{-1} (\nu_0)_{\lam\lam}
	=	\langle \lam \rangle^{-1} (f)_{\lam\lam} +
		\langle \lam \rangle^{-1} 
			\left( \calK_{00}(\nu_0) \right)_{\lam\lam}  +
		\langle \lam \rangle^{-1} 
			\left( \calK_{01} (\widetilde{\nu}^*) \right)_{\lam\lam} .
$$
Using estimates \eqref{f2.3} and \eqref{f4.ll} in the first right-hand term, \eqref{S_0.lamlam.bis}, \eqref{S1.lamlam} and Remark \ref{rem:S_0.lamlam} in the second right-hand term,
and \eqref{K01.lamlam}
in the third term, we conclude that
\begin{multline*}
\norm[L^2([a,\infty) \times \R)]
		{\langle \diamond \rangle^{-1}
			{\nu_0}_{\lam\lam}(\dotarg,\diamond)}\\
\lesssim	
	\left( 1 + \norm[H^{2,2}]{\rho} \right)^2 
	\left( 
			1 	+  \norm[L^2([a,\infty),X)]{\nu^\sharp} 
				+	\norm[L^2([a,\infty) \times \R)]{{\nu_0}_\lam} 
	\right)
\end{multline*}
which shows that 
$\norm[L^2([a,\infty) \times \R)]
		{\langle \diamond \rangle^{-1}
			{\nu_0}_{\lam\lam}(\dotarg,\diamond)}$
is also bounded uniformly for $\calD \in V'$. Thus, 
we obtain Proposition \ref{prop:nu.lam}.
\end{proof}

\subsubsection{$x$-Derivatives}

We now turn to estimates on ${\nu_0}_x$ and ${\nu_0}_{xx}$.
We will differentiate \eqref{nusharp.int} with respect to $x$ and therefore will need the following estimates on the operators $\calK_x$ and $\calK_{xx}$.

\begin{lemma}
\label{lemma:K-estimates}
Fix $a \in \R$ and suppose that $h = (\widetilde{h}, h^*) \in X$.
\begin{align}
\label{Kx.est}
\sup_{x \geq a}
	\norm[\calB(X)]{\calK_x}
	&	\lesssim		\left( 1 + \norm[H^{2,2}(\R)]{\rho}\right)^2\\
\label{Kxx.est}
\sup_{x \geq a}
	\norm[X]{\calK_{xx} h}
	& \lesssim	
	\left(
		1+ \norm[H^{2,2}(\R)]{\rho}
	\right)^2
	\norm[X]{h} + \norm[L^2(\R)]{\widetilde{h}_\lam}			
\end{align}
where constants are uniform for $\{(\lam_j,C_j )\}_{j=1}^N$ in a compact subset of 
$(\C^+ \times \C^\times)^N$.
\end{lemma}

\begin{proof}
In Paper I, Lemma 6.7, we proved
that $S$, $\left(S\right)_x$, and $\left(S\right)_{xx}$ are all bounded operators on $L^2(\R)$. The boundedness of $S'$ and its first two derivatives in $x$, uniform in $x \in [a,\infty)$ is a 
trivial consequence of the formulas, showing that $\calK_{00}$ is twice-differentiable in $x$ with uniform estimates. 

The remaining operators are of finite rank. The differentiability of $\calK_{11}$ and $\calK_{01}$ is an immediate consequence of the formulas. Finally, to treat $\calK_{01}$ it suffices to study mapping properties of Cauchy projection $C^+$.  For fixed $x$ and some $h\in L^2(\mathbb{R})$, we have
\begin{equation*}
C^{+}
	\left[ h(\dotarg)\rho(\cdot) e^{-2ix(\cdot)}\right](s)
=\frac{1}{\pi}
	\int_{0}^{\infty}
		\left[\widehat{h}(\dotarg)*\widehat{\rho}(\dotarg+x)\right](\xi)
		e^{2i\xi s} \, d\xi
\end{equation*}
Computing the first and second derivatives in $x$, we get
\begin{align*}
C^{+}
	\left[ 
		h(\dotarg)\rho(\cdot) e^{-2ix(\cdot)}
	\right]_{x}(s)
&=	\frac{1}{\pi}
			\int_{0}^{\infty}
				\left[\widehat{h}(\dotarg)*\widehat{\rho}'(\dotarg+x)\right](\xi)
			e^{2i\xi s} \, d\xi	\\
C^{+}
	\left[ 
		h(\dotarg)\rho(\cdot) e^{-2ix(\cdot)}
	\right]_{xx}(s)
&=	\frac{1}{\pi}
		\int_{0}^{\infty}
			\left[
				\widehat{h}(\dotarg)*\widehat{\rho}''(\dotarg+x)
			\right](\xi)
			e^{2i\xi s} \, d\xi
\end{align*}
Using Plancherel's identity we have 
$$
\left\Vert 
	C^{+}
		\left[ 
			h(\dotarg)\rho(\cdot) e^{-2ix(\cdot)}
		\right]_{x}     
\right\Vert_{L^2}
\lesssim 
\Vert	\widehat{h}\Vert_{L^2} 
\Vert\widehat{\rho}'\Vert_{L^1}
$$
and
$$
\left\Vert 
	C^{+}
		\left[ 
			h(\dotarg)\rho(\cdot) e^{-2ix(\cdot)}
		\right]_{xx}     
\right\Vert_{L^2}
\lesssim 
	\Vert
		\widehat{h}
	\Vert_{L^1} 
	\Vert\widehat{\rho}''\Vert_{L^2}.
$$
The estimate \eqref{Kx.est} is now easily deduced.
Since
$$  
\Vert\widehat{h}\Vert_{L^1}
\leq 
\norm[L^2]{\langle\xi\rangle \widehat{h}}
\leq 
\left\Vert \frac{\dee h}{\dee \lambda}\right\Vert_{L^2}+
\Vert {h}\Vert_{L^2}
$$
the estimate \eqref{Kxx.est}  follows.
\end{proof}

\begin{remark}
Since all estimates in the proof of Proposition \ref{lemma:K-estimates} are bilinear in the scattering data $\left(	\rho,\{ (\lam_k, C_k) \}_{k=1}^N \right)$, it follows that 
$\left( \rho,\{ (\lam_k, C_k) \}_{K=1}^N \right) \mapsto \calK_x$ and 
$\left( \rho,\{ (\lam_k, C_k) \}_{k=1}^N \right)\mapsto\calK_{xx}$ 
are locally Lipschitz maps from 
$H^{2,2}(\R) \times (\C^{+} \times \C^\times)^N$ to the bounded operators on $L^2_\lambda\oplus \mathbb{C}^N $.
\end{remark}

\begin{proof}[Proof of Proposition \ref{prop:nu.x}]
Differentiating the integral equation \eqref{nusharp.int} we see that
\begin{align*}
\nu_x^\sharp		&= (f)_x + ({\calK})_x \nu^\sharp + {\calK} (\nu_x^\sharp)	\\
\nu_{xx}^\sharp 	&= (f)_{xx} + ({\calK})_{xx} \nu^\sharp + ({\calK})_x (\nu^\sharp)_x + {\calK} (\nu^\sharp_{xx})
\end{align*}
so that we can conclude from  that $\nu_x^\sharp$, $\nu_{xx}^\sharp \in L^2([a,\infty),X)$ provided we show that
\begin{align*}
h_1	&=	(f)_x + ({\calK})_x\nu^\sharp\\
h_2	&=	(f)_{xx} + ({\calK})_{xx}\nu^\sharp + ({\calK})_x (\nu^\sharp)_x + {\calK}(\nu_{xx}^\sharp)
\end{align*}
belong to $L^2([a,\infty),X)$.
It follows from \eqref{f2.1}, \eqref{f2.2}, and \eqref{f4.x} that 
$f_x \in L^2([a,\infty),X)$, while $(\calK)_x \nu^\sharp \in L^2([a,\infty),X)$ by \eqref{Kx.est}. Hence 
$h_1 \in L^2([a,\infty),X)$, so that $(\nu^\sharp)_x \in L^2([a,\infty),X)$. 

To see that $h_2 \in L^2([a,\infty),X)$, we use \eqref{f2.1}, \eqref{f2.2} and \eqref{f4.xx} to conclude 
that $f_{xx} \in L^2([a,\infty),X)$; \eqref{Kxx.est}
and the fact that $\nu^\sharp$ and ${\nu_0}_\lam$ both belong to $L^2([a,\infty),X)$ to show that $(\calK)_{xx}(\nu^\sharp)$ belongs to $L^2([a,\infty),X)$; and \eqref{Kx.est} and our previous result to show that 
$(\calK)_x \nu_x^\sharp \in L^2([a,\infty),X)$. 
Hence $h_2 \in L^2([a,\infty),X)$ and so 
$\nu^\sharp_{xx} \in  L^2([a,\infty),X)$.
\end{proof}


\section*{Acknowledgements}

PAP and JL would like to thank the University of Toronto for hospitality during part of the time that this work was done. Conversely, RJ and CS would like to thank the department of Mathematics at the University of Kentucky for their hospitality during  part of the time that this work was done.  The authors are grateful to Peter Miller for numerous helpful discussions and particularly for his suggestion to use the ideas of Tovbis and Venakides \cite{TV00} 
to construct soliton-free initial data of large $L^2$-norm. They also thank Tom Trogdon for a helpful conversation on Zhou's vanishing lemma.

PAP was supported in part by  Simons Foundation Research and Travel Grant 359431, and CS was supported in part by Grant 46179-13 from the Natural Sciences and Engineering Research Council of Canada.

\appendix			   			

%
%
%

\section{Beals-Coifman Integral Equations}
\label{app:BC}

Here we record Beals-Coifman integral equations for Problem \ref{RHP2.row} and \ref{RHPX},
both for the case $\eps=-1$. 

\subsection{Beals-Coifman Integral Equations for Problem \ref{RHP2.row}}
\label{app:BC.RHP2}

The inhomogenous Beals-Coifman integral equations for Problem 
\ref{RHP2.row} with $t=0$, using the factorization \eqref{RHP2.W}, are as follows.  
Recall the convention \eqref{rhoCx}. 
The homogeneous equations are 
obtained  by deleting $1$ from the right-hand
sides of \eqref{RHP2c.11.R} 
and \eqref{RHP2c.11.Lam*}.
\begin{align}
\label{RHP2c.11.R}
\left. \nu_{11}(x,\lam)\right|_\R
	&=	1	+	
				C^-	\left(
							 (\dotarg) \overline{\rho_x (\dotarg)}
							\nu_{12}(x,\dotarg)
						\right)(\lam)
			+  \sum_j C_{\Gamma_j}
						\left(
							\frac
								{C_{j,x} \lam_j \nu_{12}(x,\dotarg) } 
								{(\dotarg - \lam_j)}
						\right)(\lam)\\[5pt]
\label{RHP2c.12.R}
\left. \nu_{12}(x,\lam) \right|_\R
	&=	
				C^+	\left(
								\rho_x(\dotarg) \nu_{11}(x,\dotarg)
						\right)(\lam)
				+	\sum_j C_{\Gamma_j^*}
						\left(
							\frac
								{  \overline{C_{j,x}}
									\nu_{11}(x,\dotarg) 
								}
								{(\dotarg - \overline{\lambda_j}) }
						\right)(\lam)\\[5pt]
\label{RHP2c.12.Lam}
\left. \nu_{12}(x,\lam)\right|_{\Gamma_i}
	&=	
				C_\R	\left(
								\rho_x(\dotarg) \nu_{11}(x,\dotarg)
						\right)(\lam)
				+	\sum_j C_{\Gamma_j^*}
						\left(
							\frac
								{  \overline{C_{j,x}}
									\nu_{11}(x,\dotarg) 
								}
								{(\dotarg - \overline{\lambda_j}) }
						\right)(\lam)\\[5pt]
\label{RHP2c.11.Lam*}
\left. \nu_{11}(x,\lam)\right|_{\Gamma_i^*}
	&=	1	+	
				C_\R	
						\left(
							 (\dotarg) \overline{\rho_x (\dotarg)}
							\nu_{12}(x,\dotarg)
						\right)(\lam)
			+  \sum_j C_{\Gamma_j}
						\left(
							\frac
								{\nu_{12}(x,\dotarg) C_{j,x} \lam_j}
								{(\dotarg - \lam_j)}
						\right)(\lam)
\end{align}

\subsection{Beals-Coifman Integral Equations for Problem \ref{RHPX}}

The Beals-Coifman integral equations for Problem \ref{RHPX} with $t=0$,
choosing the factorization \eqref{RHPX.w}, are as follows.
 A straightforward  computation shows that functions of the form
\begin{equation*}
\mu(x,\zeta) = 
\begin{pmatrix}
\mu_{11}(x,\zeta)	&	\mu_{12}(x,\zeta)	\\[5pt]
-\overline{\mu_{12}(x,\zetabar)}	&	\overline{\mu_{11}(x,\zetabar)}
\end{pmatrix}
\end{equation*}
solve the equation $\mu = I + \calC_w \mu$  provided $\mu_{11}$ and $\mu_{12}$ satisfy the first
row of these equations. Thus we only give equations for $\mu_{11}$ and $\mu_{12}$.
We set
$$ 
r_x(s) = r(s) e^{-2ix s^2}, \quad
\br_x(s) = \br(s) e^{2ixs^2}, \quad
c_{j,x} 	=  c_j e^{2ixs^2}.
$$
As before, the corresponding homogeneous equations are obtained by deleting the $1$ in all equations involving $\mu_{11}$.
\begin{align}
\label{RHP1c.11.Sig}
\left. \mu_{11}(x,\zeta)\right|_\Sigma
	&=	1	+
			 C^-\left(-\mu_{12}(x,\dotarg) \br_x(\dotarg)  \right)(\zeta) 
	 +	\sum_{j,\pm}
						 C_{\pm \gamma_j}
						 	\left(	
						 		\frac	{\mu_{12}(x,\dotarg) c_{j,x}}
						 				{\dotarg - {\pm  \zeta_j} }
						 	\right)(\zeta)\\[5pt]
\label{RHP1c.12.Sig}
\left. \mu_{12}(x,\zeta)\right|_\Sigma
	&=	C^+\left(\mu_{11}(x,\dotarg) r_x(\dotarg) \right)(\zeta) \\
\label{RHP1c.12.gam}
\left. \mu_{12}(x,\zeta)\right|_{\pm \gamma_i}
	&=	\calC_\Sigma 
				\left(\mu_{11}(x,\dotarg) r_x(\dotarg) \right)(\zeta)
	+	\sum_{j,\pm}
						C_{\pm \gamma_j*}
							\left(
								\frac{ \overline{c_{j,x}} 
										\mu_{11}(x,\dotarg)}
										{\dotarg - {\pm  \overline{\zeta_j}} }
							\right)(\zeta).\\
\label{RHP1c.11.Gam*}
\left. \mu_{11}(x,\zeta)\right|_{\pm \overline{\gamma_i}}
	&=	1	+
			 C_\Sigma \left(-\mu_{12}(x,\dotarg) \br_x(\dotarg)  \right)(\zeta)
\end{align}
In the restricted summations over $j$, we fix an index $i$ and one sign for the contour 
$\gamma_i$,  
then sum over all 
$(\pm,j)$ for which either $i \neq j$ or $i=j$ but the signs do not coincide.

%
%
%

\section{Soliton-free initial Data of Large \texorpdfstring{$L^2$}{L2}-norm}
\label{app:empty}

We  establish the existence of $q \in H^{2,2}(\R)$ with arbitrary $L^2$-norm which have no spectral singularities and no solitons. { Our computations use ideas of Tovbis-Venakides \cite{TV00} and diFranco-Miller \cite{dFM08}}.

\begin{proposition}
\label{prop:empty}
	The scattering map $\mathcal{R}(q)$ described by Definition~\ref{def:R} can be explicitly computed for the family of potentials 
	\begin{equation}\label{q.fam}
		q(x) = \nu \sech(x)^{1-2i\mu} e^{i \lp S_0 - \eps \nu^2 \tanh(x) - 2\delta x \rp},
		\qquad
		\| q \|_{L^2(\R)}^2 = 2 \nu^2,
	\end{equation}
	where $\nu > 0$, and $\mu, \delta, S_0 \in \R$. 
	The condition 
	\begin{equation}\label{cond.empty}
		\eps \delta < \mu^2/\nu^2
	\end{equation}
	is sufficient to guarantee the discrete spectrum is empty, i.e., $q \in U_0$ .
\end{proposition}
\begin{remark}  This construction exhibits a family of initial conditions of arbitrarily large $L^2$-norm giving rise to  globally well-posed solutions. 
A recent numerical simulation performed by C. Klein \cite{K17} with initial conditions of this form shows indeed that the solutions exist for all time, however they display stiff dispersive shocks.
\end{remark}

We prove this proposition by showing that the linear system
\begin{equation}
\label{LS.bis}
\psi_x = 
\begin{pmatrix}
-i \lam - \dfrac{i\eps|q|^2}{2}	&	 q	\\
 \eps \lam \qbar						&	i\lam + \dfrac{i\eps}{2}|q|^2
\end{pmatrix}
\psi
\end{equation}
can be reduced to solving the hypergeometric differential equation for any potentials $q$ in the family \eqref{q.fam}. This reduction involves several changes of dependent variable as well as the map 
\begin{equation}\label{svar.def}
	s = \frac{1}{2} \lp 1 + \tanh(x) \rp
\end{equation}
which compactifies $\R$ to $(0,1)$. The $x \rarr \pm \infty$ asymptotics are equivalent to Taylor expansions at $s=0$ or $s=1$ modulo a leading singular term. 
The identity
\begin{equation}
\label{exs}
e^x = s^{1/2}(1-s)^{-1/2}
\end{equation} 
implies that $e^{i\lam x} \sim s^{i\lam/2}$ near
$s=0$, while $e^{i\lam x} \sim (1-s)^{-i\lam/2}$
near $s=1$. 
We normalize the solution as $s \darr 0$ (i.e., $x \rarr -\infty$) and use  transformation formulas for hypergeometric functions to compute the scattering data by finding asymptotics as $s \uarr 1$ (i.e., $x \rarr  +\infty$).

\subsection{Reduction to a Hypergeometric Equation}
We write  the potential $q$ in \eqref{LS.bis} as
$	q(x) = A(x) e^{iS(x)}.$
We make a change of dependent variable to remove the oscillatory factor $e^{iS}$ from $q$. Setting $\psi = e^{iS(x) \sigma_3/2} \varphi$, then
\begin{equation*}
   \varphi_x	=	
   \begin{pmatrix}
      -i\lam - \dfrac{i\eps}{2} A^2 - \dfrac{i}{2}S_x
      &	 A				\\[5pt]
      \eps \lam A								
      &	i\lam + \dfrac{i\eps}{2} A^2 + \dfrac{i}{2}S_x
   \end{pmatrix}
   \varphi.
\end{equation*}
Introducing the change of independent variable \eqref{svar.def},  noting that
$d/dx = 2s(1-s) d/ds$, we obtain
\begin{equation}
\label{LS.phi.t}
   2s(1-s) \varphi_s = 
   \begin{pmatrix}
      -i\lam - \dfrac{i\eps}{2} A^2 - \dfrac{i}{2}S_x
      &	A				\\[5pt]
      \eps \lam A										
      &	i\lam + \dfrac{i\eps}{2} A^2 + \dfrac{i}{2}S_x.
   \end{pmatrix}
   \varphi
\end{equation}
Next, we reduce the system to a single second order ODE, by introducing the change of dependent variables
\begin{equation}\label{WW.def}
   W(s) = g(s) \diag{1}{ \dfrac{A}{2s(1-s)} } \varphi(s),
\end{equation}
where the scalar factor $g$ removes the $(1,1)$-entry of the coefficient matrix on the right hand side of \eqref{LS.phi.t} by choosing it to 
satisfy
\begin{equation}
\label{g}
   2s(1-s) g'(s) = \left(i\lam + \dfrac{i\eps}{2} A^2 + \dfrac{i}{2}S_x \right) g(s),
\end{equation}
while the diagonal matrix factor in \eqref{WW.def} normalizes the $(1,2)$-entry of the coefficient matrix to $1$. 
A short computation shows that $W$ obeys the differential equation
\begin{equation*}
   W_s = 
   \begin{pmatrix}
      0	  & 
      1	  \\[5pt]
      \dfrac{\eps \lam A^2}{4s^2(1-s)^2}   &		
      i\dfrac{2\lam+ \eps A^2 + S_x}{2s(1-s)} 
         + \left( \dfrac{A_s}{A} + \dfrac{2s-1}{s(1-s)} \right)
   \end{pmatrix}
   W
\end{equation*}
Setting
$   W = \begin{pmatrix} w \\ w_s \end{pmatrix} ,$
this system reduces to the single second-order equation
\begin{equation}
\label{LS.w}
   s(1-s) w'' 
   - \left[ 
      i \left( \lam + \frac{\eps A^2}{2} + \frac{S_x}{2}  \right) 
     + \frac{s(1-s)A_s}{A} + 2s-1 
   \right] w' 
	- \frac{\eps \lam A^2}{4s(1-s)}w =0.
\end{equation}
Observe that for the family of potentials \eqref{q.fam}, we have
\begin{gather*}
	A^2 = \nu^2 \sech^2(x) = 4\nu^2 s(1-s) \\
	S_x = -\eps \nu^2 \sech^2(x) + 2 \mu \tanh(x) - 2\delta
	    = -\eps A^2 + 4 \mu s - 2(\mu+\delta)
\end{gather*}
which reduces \eqref{LS.w} to
\begin{equation}
\label{LS.w.2}	
	s(1-s) w'' 
	+ \lb -i\lam + i(\mu+\delta) + \frac{1}{2} - (1+2i\mu) s \rb w' 
	- \eps \nu^2 \lam w = 0.
\end{equation}
which is the hypergeometric equation
\begin{gather}
\label{hyper}
	s(1-s)w'' + \left[c - (1+ a + b)s \right]w' - ab w = 0 \\
\shortintertext{with} 
\label{alpha.beta.imp}
	a + b	=	2i\mu, \qquad
	a b	=	\eps \nu^2 \lam, \qquad
	c			=	-i\lam + i (\mu+ \delta) + \frac{1}{2}.
\shortintertext{Hence}
\label{alpha.beta}
	a			=	i\mu + i\nu \sqrt{-\eps} R(\lam), \qquad
	b			=	i\mu -  i\nu \sqrt{-\eps} R(\lam),
\shortintertext{where}
\label{Rzeta}
	R(\lam) = \sqrt{\lam - \eps\left(\dfrac{\mu}{\nu}\right)^2}
\end{gather}
is principally branched. Observe that $R$ maps $\C^{+}$ onto $\C^{++}$ for either sign of $\eps$. We take $\sqrt{\eps} = 1$ or $i$ when $\eps = 1$ or $-1$ respectively.

\subsection{The Jost Solutions}
Denote by $F(a,b,c;s)$ the hypergeometric function, analytic in the disk $|s|<1$, with 
\begin{equation*} 
	F(a,b,c;s) = 
	1 	+ \frac{ab}{c}s 
		+ \frac{1}{2!}\frac{a(a+1) b(b+1)}{c(c+1)}s^2 + \ldots
	\qquad
	|s| < 1	 
\end{equation*}
(see \cite[15.2.1]{DLMF}).
A basis for the solution space of \eqref{LS.w.2} with singularity at $s=0$ is given by  \cite[15.10.2-15.10.3]{DLMF}
\begin{equation*}
w_1(s)	=	F(a,b,c;s) ~, \quad
w_2(s)	=	s^{1-c} F(a-c+1,b-c+1,2-c;s),
\end{equation*}
while a basis for the solution space of \eqref{LS.w.2} with singularity at $s=1$ is
given by \cite[15.10.4-15.10.5]{DLMF}
\begin{equation*}
\begin{split}
w_3(s)	&=	F(a, b, a+b-c+1;1-s)\\
w_4(s)	&=	(1-s)^{c-a-b}F(c-a,c-b,c-a-b+1;1-s)
\end{split}
\end{equation*}
We need the connection formula \cite[15.10.21]{DLMF}
\begin{equation}
\label{w1c}
w_1(s) = 
\frac	{\Gamma(c)\Gamma(c-a-b) }
		{\Gamma(c-a) \Gamma( c-b)} 
		w_3(s) 
+ 
\frac	{\Gamma(c) \Gamma(a+b-c)}
		{\Gamma(a) \Gamma(b)} 
		w_4(s) 
\end{equation}
to compute scattering data for the potential $q$. 

In what follows, we denote by $N_1^\pm$ and $N_2^\pm$ the first and second columns of the normalized Jost solutions \eqref{Jost.N}.
We now seek the correctly normalized solution
\begin{equation}
\label{m1-.form}
N_1^-(x;\lam) = e^{iS(x)\sigma_3/2} \dfrac{ e^{i\lam x} }{g(s)}
	\begin{pmatrix}
	1		&		0	\\[3pt]
	0		&		\dfrac{\sqrt{s(1-s)}}{\nu}
	\end{pmatrix}
	\begin{pmatrix}
	w(s)		\\[3pt]
	w'(s)
	\end{pmatrix}
\end{equation}
by setting $w(s)=c_1 w_1(s) + c_2 w_2(s)$ and using the asymptotics
\begin{equation}
\label{m1-.left}
N_1^-(x;\lam) =
\begin{pmatrix}
1\\0
\end{pmatrix}
+ o(1)
\end{equation}
for $x \rarr -\infty$, i.e., $s \darr 0$, to choose $c_1$ and $c_2$.  
One may solve \eqref{g} for $g$, and use \eqref{q.fam} and \eqref{exs} to compute
\begin{equation}\label{three.form}
\begin{aligned}
	e^{iS(x)/2}	&=	
	2^{-i\mu} e^{iS_0/2-i\eps \nu^2(s-1/2)} s^{-i(\mu+\delta)/2}(1-s)^{-i(\mu-\delta)/2}\\
	g(s) &=	s^{i\lam/2-i(\mu+\delta)/2}(1-s)^{-i\lam/2-i(\mu-\delta)/2} ~, \quad 
	e^{i\lam x}	
	=	s^{i\lam/2}(1-s)^{-i\lam/2}
\end{aligned}
\end{equation}
so that 
\begin{equation} \label{m11-m21.pre}
	\begin{aligned}
		e^{iS(x)/2}  e^{i\lam x} / g(s) 
		  &= 2^{-i\mu} e^{iS_0/2-i\eps\nu^2(s-1/2)}\\
		e^{-iS(x)/2} e^{i\lam x} / g(s) 
		 &= 2^{i\mu} e^{-iS_0/2+i\eps\nu^2(s-1/2)}s^{i(\mu+\delta)}(1-s)^{i(\mu-\delta)}.
	\end{aligned}
\end{equation}
Using \eqref{m11-m21.pre} in \eqref{m1-.form} we obtain
\begin{align*}
	N_1^-(x;\lam)	=	
	\begin{pmatrix}
	   2^{-i\mu} e^{iS_0/2-i\eps \nu^2(s-1/2)} w(s)	\\
       \frac{2^{i\mu}}{\nu} e^{-iS_0/2+i\eps\nu^2(s-1/2)} 
	     s^{i(\mu+\delta)+1/2}(1-s)^{i(\mu-\delta)+1/2} \frac{dw}{ds}
	\end{pmatrix}.
\end{align*}
Note that $\real(\gamma)=1/2$.
Setting $w(s) = c_1w_1(s) + c_2w_2(s)$ and using the convergent series expansions
\[
    w_1(s) \sim 1 + \bigO{s}, \quad 
    w_2(s) \sim s^{1-\gamma}\left(1+\bigO{s} \right), 
\]
we find that $c_1= 2^{i\mu} e^{-iS_0/2-i\eps \nu^2/2}$, $c_2=0$ so that
\begin{equation}
\label{m1-.sol}
N_1^-(x;\lam)	=	
\begin{pmatrix}
e^{-i\eps \nu^2 s} w_1(s)\\[4pt]
\frac{2^{2i\mu}}{\nu} e^{-iS_0+i\eps \nu^2(s-1)}
s^{c+i\lam}(1-s)^{1-c+a+b-i\lam} w_1'(s)
\end{pmatrix} ~.
\end{equation}
Repeating this calculation for the Jost functions $N_1^+$ and $N_2^+$ 
yields
\begin{equation}\label{m+.sol}
   \begin{aligned}
		N_1^+(x;\lam) &= 
		\begin{pmatrix}
		   e^{-i \eps \nu^2(s-1)} w_3(s) \\[4pt]
		   \frac{2^{2i\mu}}{\nu} e^{-i S_0 + i \eps \nu^2 s} 
		     s^{c+i\lam} (1-s)^{1-c+a+b-i\lam} w_3'(s)
		\end{pmatrix} \\
		N_2^+(x;\lam) &= (a+b-c)^{-1} 
		\begin{pmatrix}
		   \nu 2^{-2i\mu} e^{i S_0 - i \eps \nu^2 s}
		     s^{-i\lam} (1-s)^{i\lam} w_4(s) \\[4pt] 		     
		   e^{i\eps \nu^2(s-1)} 
		     s^{c} (1-s)^{ 1-c+a+b} w_4'(s)
		\end{pmatrix} ~.
   \end{aligned}
\end{equation}

\subsection{Scattering Data}

The scattering coefficients $\alpha(\lam)$ and $\beta(\lam)$ defined by \eqref{Jost.T}  satisfy
\begin{equation}
\label{m1-.right}
   N_1^-(x;\lam) =  
     \overline{\alpha(\lam)} N_1^+(x;\lam) 
     - \eps \lam \overline{\beta(\lam)} e^{2i\lam x} N_2^+(x;\lam),
   \qquad
   \lambda \in \R.
\end{equation}
Using the connection formula \eqref{w1c}, the Jost functions \eqref{m1-.sol}-\eqref{m+.sol}, and the last line of \eqref{three.form} we may read off that 
\begin{align}
\label{ex.ba}
   \overline{\alpha(\lam)}	=	
       e^{-i\eps \nu^2} 		
	      \frac	{\Gamma(c)\Gamma(c-a-b)}
	            {\Gamma(c-a)\Gamma(c-b)} ~,	
				\quad
   \overline{\beta(\lam)}		=	
      - \frac{\eps 2^{2i\mu}}{\nu \lam} e^{-iS_0} 
	  \frac{\Gamma(c)\Gamma(1+a+b-c)}
	       {\Gamma(a)\Gamma(b)}.
\end{align}
From the formulas
$	c	= -i\lam + i(\mu+\delta)+\frac{1}{2} $ and $
	c-a-b =	-i\lam -i(\mu-\delta) + \frac{1}{2}$,
it is easy to see that if $\lam \in \R$ then $\real(c)=\real(c-a-b)=1/2$, so that $\alpha(\lam)$ has no spectral singularities. 
Moreover, $\real(c) > 1/2$ and $\real(c-a-b) > 1/2$ for $\lam \in \C^{+}$ so that $\overline{\alpha(\lambar)}$ is holomorphic\footnote{The parameters $a$ and $b$ appearing in the argument of $\overline{\alpha(\bar\lam)}$ and $\overline{\beta(\bar\lam)}$ inherit a branch cut from $R(\lam)$; their boundary values satisfy $a_\pm = b_{\mp}$. Since both $\overline{\alpha(\lambar)}$ and $\overline{\beta(\bar\lam)}$ are invariant under the map $a \leftrightarrow b$, both scattering coefficients are analytic across the branch of $R$. 
} in $\C^{+}$ as required.   
It is also clear from \eqref{ex.ba}. and observing that \eqref{alpha.beta.imp} implies that at $\lam = 0$ exactly one of $a$ and $b$ vanishes, that  
$\overline{\beta(\lambar)}$ extends meromorphically to $\lam \in \C^{+}$; $\overline{\beta(\lambar)}$ will have isolated simple poles at those values of $\lam$ where
\begin{equation}\label{b.pole}
	c - a - b = m, \quad m = 1,2,3,\dots
\end{equation}
and neither $a$ nor $b$ is a nonpositive integer.

The discrete spectrum of \eqref{LS.bis} are the pairs $\lam_n, \bar \lam_n$, with $\lam_n \in \C^{+}$ such that $\overline{\alpha(\lambar_n)} = 0$. From \eqref{ex.ba}, $\overline{\alpha(\lambar_n)}=0$ whenever $c - a = 1-n$ or $c-b = 1-n$. We have
\begin{align*}
   c-a &= -i\lam + i\delta - i\nu \sqrt{-\eps}R(\lam) + \frac{1}{2}\\
   c-b &= -i\lam + i\delta + i\nu \sqrt{-\eps}R(\lam) + \frac{1}{2}
\end{align*}
Notice that $\Re c-a = \Im \lam + \nu \Im( \sqrt{-\eps} R(\lam)) > 0$ for either sign of $\eps$ as $R(\lam) \in \C^{++}$ when $\lam \in \C^{+}$. Thus zeros of $\overline{\alpha(\lambar_n)}$ must satisfy $c-b = 1-n$ or equivalently
\begin{equation}
\label{cond 0}
   \lam_n -  \delta - \nu \sqrt{-\eps} R(\lam_n) + i(n-1/2) = 0, \qquad \lam_n \in \C^{+} \quad n=1,2,3,\ldots .
\end{equation}

Finally, we point out that $\overline{\beta(\lambar)}$ is analytic at any $\lam_n$ since if 
$\lam_n$ satisfies both $c - b = 1-n$ and \eqref{b.pole} then $a = 1 - n - m$ is a negative integer and the pole of $\overline{\beta(\lambar)}$ is removable. Then because the scattering relation \eqref{m1-.right} extends analytically to $\lam_n$ we can compute the connection coefficient $C_n$ (see \eqref{direct:bk.bis} and text immediately following) as
\begin{equation}\label{ex.C}
	C_n = -\eps \frac{ \overline{\beta(\lambar_n)}}{\overline{\alpha'(\lambar_n)}}.
\end{equation}

\begin{proof}[Proof of Proposition~\ref{prop:empty}]

From the arguments above, the scattering map is easily constructed for data in the family \eqref{q.fam}. Using \eqref{ex.ba}, \eqref{cond 0}, and \eqref{ex.C}, we have
\[
	\mathcal{D} = \left\{ \rho(\lam) = 
	\frac{\beta(\lam)}{\alpha(\lam)},
	\{
	(\lam_n , C_n) \}_{n=1}^N 
	\right\}
\]
It remains to show that \eqref{cond.empty} guarantees that the discrete spectrum is empty.

We want to find values of $\mu$ and $\nu$ such that \eqref{cond 0} cannot be satisfied for any $\lam \in \C^{+}$ and all $n = 1,2,3,\dots$.
Introduce the parameters
\[ 
	z \coloneqq \lambda - \delta, \qquad L_n \coloneqq n-1/2
\]
and note that $\imag z>0$ and $L_n > 0$. The condition \eqref{cond 0}
becomes
\begin{equation}
\label{cond 1}
	z + iL_n = \nu \sqrt{-\eps}  \sqrt{z+\delta - \eps(\mu/\nu)^2},
\end{equation}
where we used \eqref{Rzeta}.
Writing $z=x+iy$ with $y>0$, squaring both sides of \eqref{cond 1},
and taking real and imaginary parts we find
\begin{subequations}
\label{cond 2}
\begin{align}
\label{cond 2a}
x^2 - (y+L_n)^2 + \eps \nu^2 x		&=	-\eps \nu^2 \delta + \mu^2,	\\
\label{cond 2b}
2x \left(y+L_n \right)					&=	-\eps \nu^2 y.
\end{align}
\end{subequations}
Solving \eqref{cond 2b} for $L_n$ we have
\begin{equation}
\label{L def}
L_n 	=	-y \frac{\eps \nu^2 +  2x}{2x}.
\end{equation}
Since $L_n$ and $y$ are both strictly positive, it follows that  $-(\eps \nu^2 + 2x)/(2x)$ is strictly positive as well. Thus, any solution $z=x+iy$ lies in the vertical half-strip
\begin{equation*}
S_\epsilon=	\left\{  (x,y) \in \R^2: y >0, \,\, -\frac{\nu^2}{2} < \eps x < 0 \right\}. 
\end{equation*}
Using \eqref{L def} in \eqref{cond 2a}, we see that any solutions $z$ of \eqref{cond 1} lie along the curve 
\begin{equation}
\label{cond 3}
\calC \coloneqq \left\{ (x,y) \in \R \times (0,\infty): x^2 Q(x) = \nu^4 y^2/4 \right\}
\end{equation}
where
$\displaystyle{	Q(x) = 	\left( \eps x + \frac{\nu^2}{2} \right)^2 - 
			\left[
			   \nu^2 \left( \frac{\nu^2}{4} - \eps \delta \right) + \mu^2
			\right]}.$
To find parameter values of $(\mu,\delta)$ for which $\overline{\alpha(\bar\lam)}$ has no zeros, it suffices to find $(\mu,\delta)$ so that $\calC$ does not intersect the strip $S_\eps$. From the definition of $Q$ it is clear that $\calC$ will have empty intersection with $S_\eps$ provided $Q(x)<0$ for $x$ with $0 < -\eps x < \nu^2/4$.  For either sign of $\eps$ this is guaranteed if $Q(0) < 0$. Since $Q(0)= \eps \nu^2 \delta - \mu^2$ it follows that \eqref{cond.empty} guarantees the discrete spectrum is empty.
\end{proof}

%
%

\endgroup

\begin{thebibliography}{00}

\bibitem{BC84}
Beals, R., Coifman, R. R. Scattering and inverse scattering for first order systems. \emph{Comm.\ Pure Appl.\ Math.} \textbf{37}, no. 1 (1984),  39--90.

\bibitem{DZ03}
 Deift, P.,  Zhou, X. Long-time asymptotics for solutions of the NLS equation with initial data in a weighted Sobolev space. Dedicated to the memory of J\"urgen K. Moser. \emph{Comm. Pure Appl. Math.} 
 \textbf{56} (2003), 1029--1077.

\bibitem{dFM08}
diFranco, J., Miller, P.
The semi-classical modified nonlinear Schr\"odinger equation I : Modulation theory and spectral analysis.
 \emph{Physica D}, \textbf{237} (2008), 947--997.
 
 \bibitem{DLMF}
NIST Digital Library of Mathematical Functions. 
\url{http://dlmf.nist.gov/}, Release 1.0.11 of 2016-06-08. 
Online companion to \cite{OLBC10}.

\bibitem{DZ17}
Dyatlov, S., Zworski, M. Mathematical Theory of Scattering Resonances, Version 0.1 (March 2017). 
\href{http://math.mit.edu/~dyatlov/res/res_20170323.pdf}{http://math.mit.edu/\~{}dyatlov/res/res\_20170323.pdf}.

\bibitem{FHI17}
Fukaya, N., Hayashi, M.,  Inui, T.
A sufficient condition for global existence of solutions to a generalized derivative nonlinear Schr\"{o}dinger equation. 
\emph{Anal.\ PDE} \textbf{10} no. 5 (2017), 1149--1167.

\bibitem{GuoWu17}
Guo, Z.,  Wu, Y.
Global well-posedness for the derivative nonlinear Schr\"{o}dinger equation in $H^{1/2}(\R)$. 
\emph{Discrete Contin.\ Dyn.\ Syst.\ } \textbf{37}, no.1  (2017),  257--264.

\bibitem{HO92} Hayashi, N., Ozawa, T.
On the derivative nonlinear Schr\"odinger equation. 
 \emph{Physica D} \textbf{55} (1992), 14--36.

\bibitem{JLPS17}
Jenkins, R., Liu, J., Perry, P., and Sulem. C. Global well-posedness and soliton resolution for the derivative nonlinear Schr\"{o}dinger equation. Preprint, \href{https://arxiv.org/abs/1706.06252}{arXiv:1706.06252}, 2017.

\bibitem{JLPS18}
Jenkins, R., Liu, J., Perry, P., and Sulem. C. Soliton resolution for the derivative nonlinear Schr\"{o}dinger equation.  Submitted to \emph{Comm.\ Math.\ Phys.}

\bibitem{KN78}
Kaup, D. J., Newell, A. C. (1978).  An exact solution for a derivative nonlinear Schr\"{o}dinger equation. 
\emph{J. Math. Phys.}, \textbf{19} (1978), 798--801.
  
\bibitem{KV97}
Kitaev, A. V., Vartanian, A. H. Leading-order temporal asymptotics of the modified nonlinear Schr\"{o}dinger equation: solitonless sector. \emph{Inverse Problems} \textbf{13},
no. 5 (1997), 1311--1339.


\bibitem{K17}
Klein, C.
private communication, 2017.

\bibitem{Lee83}
Lee, J.-H. (1983).  Analytic properties of Zakharov-Shabat inverse scattering problem with a polynomial spectral dependence of degree 1 in the potential. Thesis (Ph.\ D.), Yale University.

\bibitem{Lenells14} Lenells, J. Matrix Riemann-Hilbert problems with jumps across Carleson contours. Preprint, \href{https://arxiv.org/abs/1401.2506}{arxiv.org/abs/1401.2506}.

\bibitem{Liu17}
Liu, J. Global  well-posedness for the derivative nonlinear Schr\"{o}dinger equation through inverse scattering. Thesis, University of Kentucky, 2017.

\bibitem{LPS15}
Liu, J.,  Perry, P.,   Sulem, C.  Global existence for the derivative nonlinear Schr\"{o}dinger equation by the method of inverse scattering. \emph{Comm.\ Part. Diff. Eqs} \textbf{41}  no. 11 (2016), 1692--1760. 

\bibitem{LPS16}
Liu, J.,  Perry, P.,   Sulem, C.  Long-time behavior of solutions to the derivative nonlinear Schr\"{o}dinger equation by the method of inverse scattering.  \href{https://arxiv.org/abs/1608.07659}{arXiv1608.07659}, to appear in 
\emph{Ann.\ Inst.\ H. Poincar\'{e} C - Analyse non-lin\'{e}aire}.

\bibitem{OLBC10}
Olver, F.,  Lozier,  D., Boisvert, R., and Clark, C., editors. 
NIST Handbook of Mathematical Functions. 
Cambridge University Press, New York, NY, 2010. 
Print companion to \cite{DLMF}.

\bibitem{PS17}
Pelinovsky D., and  Shimabukuro,  Y.
Existence of global solutions to the derivative NLS equation with the inverse
scattering transform method, 
\emph{Intern. Math. Res. Notices}, (2017).

\bibitem{PSS17}
Pelinovsky, D.,  Saalmann, A.,  Shimabukuro, Y..
The derivative NLS equation: global existence with solitons.
\emph{Dynamics of PDEs},  \textbf{14}, no. 3 (2017), 217--294.

\bibitem{Takaoka99}
Takaoka, H.
Well-posedness for the one-dimensional nonlinear Schr\"{o}dinger equation with the derivative nonlinearity. 
\emph{Adv.\ Differential Equations} \textbf{4}, no. 4 (1999),  561--580. 

\bibitem{TO16}
Trogdon, T., Olver, S.
Riemann-Hilbert problems, their numerical solution, and the computation of nonlinear special functions. 
Society for Industrial and Applied Mathematics (SIAM), Philadelphia, PA, 2016. 

\bibitem{TV00}
Tovbis, A., Venakides, S.
The eigenvalue problem for the focusing nonlinear Schr\"{o}dinger equation:  new solvable cases.
\emph{Physica D}, \textbf{146} (2000), 250--264.

\bibitem{Wu14} Wu, Y. 
Global well-posedness on the derivative nonlinear Schr\"odinger equation.
\emph{Anal. PDE},  \textbf{8}, no. 5 (2015), 1101Ð-1112.
 
\bibitem{Zhou89}
Zhou, X. The Riemann-Hilbert problem and inverse scattering. 
\emph{SIAM J.\ Math. Anal.\  } \textbf{20}, no. 4 (1989), 966--986.

\end{thebibliography}
\end{document}